\documentclass{aart} 

\usepackage{titlesec}
\titleformat{\subsection}[runin]{\normalfont\bfseries}{\thesubsection.}{.5em}{}[.]\titlespacing{\subsection}{0pt}{2ex plus .1ex minus .2ex}{.8em}
\titleformat{\subsubsection}[runin]{\normalfont\itshape}{\thesubsubsection.}{.3em}{}[.]\titlespacing{\subsubsection}{0pt}{1ex plus .1ex minus .2ex}{.5em}

\usepackage[labelfont=sc,font=small,labelsep=period]{caption}
\usepackage[letterpaper, hmargin=1in, top=1.5in, bottom=1.9in, footskip=0.7in]{geometry}

\usepackage{amsmath} 
\usepackage{amssymb}
\usepackage{amsfonts}
\usepackage{latexsym}
\usepackage{amsthm}
\usepackage{amsxtra}
\usepackage{amscd}
\usepackage{bbm}
\usepackage{mathrsfs}
\usepackage{bm}

\usepackage{graphicx, color}
\usepackage[nottoc,notlof,notlot]{tocbibind} 
\usepackage{cite} 
\usepackage{xspace} 

\numberwithin{equation}{section}
\numberwithin{figure}{section}
\theoremstyle{plain} 
\newtheorem{theorem}{Theorem}[section]
\newtheorem*{theorem*}{Theorem}
\newtheorem{lemma}[theorem]{Lemma}
\newtheorem*{lemma*}{Lemma}

\newtheorem*{corollary*}{Corollary}
\newtheorem{proposition}[theorem]{Proposition}
\newtheorem*{proposition*}{Proposition}
\newtheorem{definition}[theorem]{Definition}
\newtheorem*{definition*}{Definition}

\theoremstyle{definition} 

\newtheorem*{example*}{Example}
\newtheorem{remark}[theorem]{Remark}

\newtheorem*{remark*}{Remark}
\newtheorem*{remarks*}{Remarks}

\newcommand{\f}[1]{\boldsymbol{\mathrm{#1}}} 
\newcommand{\bb}{\mathbb} 
\renewcommand{\cal}{\mathcal} 
 
\newcommand{\fra}{\mathfrak} 
\newcommand{\ol}[1]{\overline{#1} \!\,} 
\newcommand{\wh}{\widehat}
\newcommand{\wt}{\widetilde}

\renewcommand{\P}{\mathbb{P}}
\newcommand{\E}{\mathbb{E}}
\newcommand{\R}{\mathbb{R}}
\newcommand{\C}{\mathbb{C}}
\newcommand{\N}{\mathbb{N}}
\newcommand{\Z}{\mathbb{Z}}

\newcommand{\ee}{\mathrm{e}}
\newcommand{\ii}{\mathrm{i}}
\newcommand{\dd}{\mathrm{d}}
\newcommand{\col}{\mathrel{\vcenter{\baselineskip0.75ex \lineskiplimit0pt \hbox{.}\hbox{.}}}}
\newcommand*{\deq}{\mathrel{\vcenter{\baselineskip0.65ex \lineskiplimit0pt \hbox{.}\hbox{.}}}=}
\newcommand*{\eqd}{=\mathrel{\vcenter{\baselineskip0.65ex \lineskiplimit0pt \hbox{.}\hbox{.}}}}

\renewcommand{\leq}{\leqslant}
\renewcommand{\geq}{\geqslant}
\renewcommand{\epsilon}{\varepsilon}

\newcommand{\ind}[1]{\f 1 (#1)}
\newcommand{\indb}[1]{\f 1 \pb{#1}}

\newcommand{\indBB}[1]{\f 1 \pBB{#1}}

\newcommand{\p}[1]{({#1})}
\newcommand{\pb}[1]{\bigl({#1}\bigr)}
\newcommand{\pB}[1]{\Bigl({#1}\Bigr)}
\newcommand{\pbb}[1]{\biggl({#1}\biggr)}
\newcommand{\pBB}[1]{\Biggl({#1}\Biggr)}

\newcommand{\qb}[1]{\bigl[{#1}\bigr]}

\newcommand{\h}[1]{\{{#1}\}}
\newcommand{\hb}[1]{\bigl\{{#1}\bigr\}}

\newcommand{\hBB}[1]{\Biggl\{{#1}\Biggr\}}

\newcommand{\abs}[1]{\lvert #1 \rvert}
\newcommand{\absb}[1]{\bigl\lvert #1 \bigr\rvert}

\newcommand{\absbb}[1]{\biggl\lvert #1 \biggr\rvert}

\newcommand{\normbb}[1]{\biggl\lVert #1 \biggr\rVert}

\newcommand{\avg}[1]{\langle #1 \rangle}
\newcommand{\avgb}[1]{\bigl\langle #1 \bigr\rangle}

\DeclareMathOperator{\tr}{Tr}
\DeclareMathOperator{\var}{Var}

\DeclareMathOperator{\supp}{supp}

\DeclareMathOperator{\re}{Re}
\DeclareMathOperator{\im}{Im}

\DeclareMathOperator{\dist}{dist}

\begin{document}
\author{L\'aszl\'o Erd\H{o}s\footnote{IST Austria, Am Campus 1, Klosterneuburg A-3400, lerdos@ist.ac.at. On leave from Institute of Mathematics,
University of Munich. Partially supported by SFB-TR 12 Grant of the German Research Council  and by ERC 
Advanced Grant, RANMAT 338804.}
\and Antti Knowles\footnote{ETH Z\"urich, knowles@math.ethz.ch. Partially supported by Swiss National Science Foundation grant 144662.}}
\title{The Altshuler-Shklovskii Formulas for Random Band Matrices I:  the Unimodular Case}
\maketitle

\begin{abstract}
We consider the spectral statistics of large random band matrices
on mesoscopic energy scales. We show that
the correlation function of the local eigenvalue density exhibits  a universal power law behaviour
that differs from the Wigner-Dyson-Mehta statistics. 
This law had  been  predicted  in the physics literature by
Altshuler and Shklovskii \cite{AS}; it
describes the correlations of the eigenvalue density in general metallic samples with weak disorder.
Our  result  rigorously
establishes the Altshuler-Shklovskii formulas for band matrices.
In two  dimensions, where the leading term vanishes owing to an  algebraic 
cancellation, we identify the first non-vanishing term  and show that it  differs substantially
from the prediction of Kravtsov and Lerner \cite{KrLe}.  The proof is given in the current paper and its companion \cite{EK4}. 
\end{abstract}


\section{Introduction} \label{sec:intro}

The eigenvalue statistics of large random Hermitian 
matrices with independent entries are known to exhibit universal behaviour.
Wigner proved \cite{Wig} that the eigenvalue density converges (on the macroscopic scale)  to the semicircle  law  as the
dimension of the matrix tends to infinity. He also observed
that the local statistics of individual eigenvalues (e.g.\ the gap statistics)
are universal, in the sense that they depend only on the
symmetry class of the matrix but are otherwise independent of
the distribution of the matrix entries. In the Gaussian
case, the local spectral statistics were identified by Gaudin, Mehta, and Dyson
\cite{Meh}, who proved that they are governed by the celebrated sine kernel.

In this paper and its companion \cite{EK4}, 
we focus the universality of the eigenvalue density statistics on intermediate, so-called \emph{mesoscopic}, scales, which lie between the macroscopic and the local scales. 
We study
 \emph{random band matrices},  commonly used to model quantum transport in disordered media.
 Unlike the mean-field Wigner matrices, band matrices possess
 a nontrivial spatial structure. 
 Apart from the obvious mathematical interest,
an important  motivation for this question arises from physics,
 namely from the theory of conductance fluctuations  developed by 
Thouless \cite{Th}. 
In the next sections we explain the physical background 
of the problem. Thus, readers mainly interested in the mathematical aspects of our results may skip  much of the introduction.

\subsection{Metal-insulator transition}

According to the Anderson metal-insulator transition
\cite{And}, general disordered quantum systems are believed to fall into one of two very distinctive regimes. 
In the \emph{localized regime} (also called the \emph{insulator regime}), 
physical quantities depending on the position, such
as eigenvectors and resolvent entries, decay on a length scale $\ell$ (called the \emph{localization length}) that is independent
of the system size. The  unitary time evolution generated by the Hamiltonian
remains localized for all times and the local spectral statistics are Poisson.
In contrast,  in the \emph{delocalized regime} (also called  the \emph{metallic regime}),
 the localization length is comparable with the linear system size. 
The overlap of the 
 eigenvectors induces strong correlations in the local 
eigenvalue statistics, which are believed to be universal and to coincide
with those of a Gaussian matrix ensemble of the appropriate symmetry class.
Moreover, the unitary  time evolution generated by the Hamiltonian  is diffusive  for large times. 
Strongly disordered systems are in the localized regime. In the weak disorder regime, the localization 
properties depend on the dimension and on the energy.

Despite compelling theoretical arguments and numerical evidence, the Anderson
metal-insulator transition has been rigorously proved only in a few very special cases.
The basic model is the random Schr\"odinger operator, 
$-\Delta +V$,
  typically defined  on $\R^d$ or on a graph
(e.g.\ on a subset of $\Z^d$). Here $V$ is a random potential with short-range spatial correlations; for instance, in the case of a graph, $V$ is a family of independent random variables indexed by the vertices.
The localized regime is relatively well understood since
the pioneering work of Fr\"ohlich and Spencer \cite{FroSpe, FroSpe2}, followed by an alternative approach by Aizenman and Molchanov \cite{AizMol}.  The Poissonian nature of the local spectral statistics  was proved by Minami \cite{Min}.  On the other hand, the delocalized regime has seen far less progress.  With the exception of the Bethe lattice \cite{Kl, ASW, FHS}, only partial results
are available. They indicate delocalization and quantum diffusion in certain
limiting regimes \cite{ESalY1, ESalY2, ESalY3, EY00}, or in a somewhat different model where the static random potential is replaced with 
 a dynamic phonon field in a thermal state at positive temperature  \cite{FdeR, DRK}.

 Another much studied family of  models describing disordered quantum systems is random matrices.
Delocalization is well understood for random Wigner matrices \cite{ESY2, EYY1}, but, owing to their mean-field character, they are always in the delocalized regime, and hence no phase transition takes place. The local eigenvalue statistics are universal. This fundamental fact about
random matrices, also known as the Wigner-Dyson-Mehta conjecture, has been recently proved
\cite{ESY4, EYY3, EYBull} (see also \cite{TV1} for a partially alternative
argument in the Hermitian case).

\subsection{Mesoscopic statistics}

In a seminal paper \cite{AS}, Altshuler and Shklovskii
computed a new physical quantity: the  variance of the number ${\cal N}_\eta$  of eigenvalues
on a mesoscopic energy scale $\eta$ in $d$-dimensional metallic samples with disorder for $d \leq 3$; here \emph{mesoscopic} refers to scales $\eta$ that are much larger then the typical eigenvalue spacing $\delta$ but much smaller than the total (macroscopic) energy scale of the  system.
Their motivation was to study fluctuations of the conductance in mesoscopic metallic
samples; see also \cite{Alt} and \cite{LS}. The relationship between $\cal N_\eta$
and the conductance is given by a fundamental result of Thouless \cite{Th}, asserting that the
conductance of a sample of linear size $L$  is determined
by the (one-particle) energy levels in an energy band of a specific width $\eta$
around the Fermi energy. In particular, the variance of $\cal N_{\eta}$ directly contributes to the conductance fuctuations.
This specific value of $\eta$ is  given by $\eta=\max\{ \eta_c, T\}$, where
$T$ is the temperature and $\eta_c$ is the \emph{Thouless energy} \cite{Th}. In diffusive models the Thouless energy is defined
as $\eta_c \deq D/L^2$, where $D$ is the diffusion coefficient.  
(In a conductor the dynamics of the particles, i.e.\ the itinerant electrons,
 is typically diffusive.)
The Thouless energy  may also be interpreted as the inverse diffusion time, i.e.\ the time needed for the
 particle to diffuse through the sample.

As it turns out, the mesoscopic linear statistics $\cal N_\eta$ undergo a sharp transition precisely at\footnote{We use the notation $a \asymp b$ to indicate that $a$ and $b$ have comparable size. See the conventions at the end of Section \ref{sec:intro}.} $\eta \asymp \eta_c$. 
For small energy scales, $\eta\ll \eta_c$, Altshuler and Shklovskii found that the variance of $\cal N_\eta$ behaves according to\footnote{We use the notation $\avg{\cdot \,; \cdot}$ to denote the covariance and abbreviate $\var X \deq \avg{X \, ; X}$. See \eqref{def_avg} below.}
\begin{equation}\label{ASvar1}
\var \cal N_\eta \;\asymp\; \log \cal N_\eta \;\asymp\; \log L\,,
\end{equation}
as predicted by the Dyson-Mehta statistics \cite{DM}. The unusually small variance is
due to the strong correlations among eigenvalues (arising from  level repulsion).
In the opposite regime, $\eta\gg \eta_c$, the variance is typically much larger, and behaves according to 
\begin{equation}\label{ASvar2}
\var \cal N_\eta \;\asymp\; (\eta/\eta_c)^{d/2} \;=\; L^d (\eta/D)^{d/2} \qquad (d = 1,2,3)\,.
\end{equation}

The threshold $\eta_c$ may be understood by introducing the concept of (an energy-dependent) \emph{diffusion length} $\ell_\eta$,  which is the typical spatial scale on which the off-diagonal matrix entries of those observables decay
that live on an energy scale $\eta$ (e.g.\ resolvents whose spectral parameters have imaginary part $\eta$).
Alternatively, $\ell_\eta$ is the linear scale of an initially localized state evolved up to time $\eta^{-1}$. 
The diffusion length is related to the localization length $\ell$ through $\ell = \lim_{\eta \to 0} \ell_\eta$. 
Assuming that the dynamics of the quantum particle can be described by a classical diffusion
process, one can show that $\ell_\eta \asymp\sqrt{D/\eta}$
and the relation  $\eta\ll \eta_c = D/L^2$ may be written as 
$L\ll \ell_\eta$. The physical interpretation is that the sample is 
so small that the system is essentially 
mean-field from the point of view of observables on the energy scale $\eta$, 
 so that the spatial structure and dimensionality of the system are immaterial. 
The opposite regime $\eta\gg \eta_c$ corresponds to large samples, $L \gg \ell_\eta$,
where the behaviour of the system can be approximated by a diffusion that has not reached the boundary of the sample. 
These two regimes are commonly referred to as \emph{mean-field} and \emph{diffusive} regimes, respectively.

A similar transition occurs if one considers  the correlation of the number of eigenvalues
 $\cal N_\eta(E_1)$ and  $\cal N_\eta(E_2)$  around
two distinct energies $E_1 < E_2$ whose separation is much larger than the energy window $\eta$ (i.e.\ $E_2-E_1 \gg \eta$). For small samples, $\eta \ll \eta_c$, the correlation decays according to
\begin{equation}\label{AScorr1}
 \avgb{\cal N_\eta(E_1) \,; \cal N_\eta(E_2)}  \;\asymp\; (E_2-E_1)^{-2}\,.
\end{equation}
This decay holds for systems both with and without time reversal symmetry.
The decay \eqref{AScorr1}  is in agreement with the Dyson-Mehta statistics, which in the complex Hermitian case  (corresponding to a system without time reversal symmetry)  predict a correlation
$$
\pbb{\frac{\sin\pb{(E_2-E_1)/\delta}}{(E_2-E_1)/\delta}}^2
$$
for highly localized observables on the scale $\eta \asymp \delta$.
 For mesoscopic scales, $\eta \gg \delta$, the oscillations in the numerator are averaged out and may be replaced with a positive constant to yield 
\eqref{AScorr1}.  A similar formula with the same decay holds for the real symmetric case (corresponding to a system with time reversal symmetry). 
On the other hand, for large samples, $\eta\gg \eta_c$, we have
\begin{equation}\label{AScorr2}
 \avgb{{\cal N}_\eta(E_1)\,; {\cal N}_\eta(E_2)}  \;\asymp\;
(E_2-E_1)^{-2+d/2} \qquad  (d=1,3)\,;
\end{equation}
for $d=2$ the correlation vanishes to leading order.
The power laws in the energies $\eta$ and $E_2 - E_1$ given in  \eqref{ASvar2} and \eqref{AScorr2}  respectively 
 are called the
\emph{Altshuler-Shklovskii formulas}. They express the 
 variance  and the correlation of the density of states  in the regime where the diffusion approximation is valid and the spatial extent of the diffusion, $\ell_\eta$, is much less than the system size $L$. 
In contrast, the mean-field formulas \eqref{ASvar1} and \eqref{AScorr1} describe the situation where 
the diffusion has reached equilibrium. Note that the behaviours \eqref{ASvar2} and \eqref{AScorr2} 
as well as \eqref{ASvar1} and \eqref{AScorr1} are very different from the ones obtained if the 
distribution of the eigenvalues were governed by Poisson statistics; in that case, for instance,
 \eqref{AScorr1} and \eqref{AScorr2} would be zero.

From a mathematical point of view, the significance of these mesoscopic quantities 
is that their statistics are  amenable to  rigorous analysis
even in the  delocalized  regime.
In this paper we demonstrate this by proving the Altshuler-Shklovskii formulas for random band matrices.

\subsection{Random band matrices}

We consider $d$-dimensional random band matrices,
which interpolate between random Schr\"odinger operators and mean-field Wigner matrices
by varying the band width $W$; see \cite{Spe} for an overview of this idea. These matrices represent 
 quantum systems in a $d$-dimensional discrete box of side length $L$, where
the quantum transition probabilities are random and their range is of order $W \ll L$.
We scale the matrix so that its spectrum is bounded, i.e.\ the macroscopic energy scale is of order 1, and hence the eigenvalue spacing is
of order $\delta\asymp L^{-d}$. Band matrices exhibit diffusion in all dimensions $d$, with a
 diffusion coefficient
$D\asymp W^2$; see \cite{FM} for a physical argument in the general case and \cite{EK1, EK2}
 for a proof  up to certain large time scales. 
In \cite{EKYY3} it was  showed that the  resolvent entries 
with spectral parameter $z= E+\ii\eta$ decay exponentially on a scale $\ell_\eta\asymp W/\sqrt{\eta}$,
as long as this scale is smaller than the system size, $W/\sqrt{\eta} \ll L$.
(For technical reasons the proof is valid only if $L$ is not too large, $L\ll W^{1+d/4}$.)
The resolvent entries do not decay if $W/\sqrt{\eta} \gg L$, in which case the system is in the mean-field regime for observables living  on energy scales of order $\eta \ll \eta_c$.
Notice that the crossover at $W/\sqrt{\eta} \asymp L$  corresponds exactly to the crossover at 
 $\eta \asymp \eta_c$ mentioned above.

\subsection{Outline of results}

Our main result is the proof of the formulas \eqref{ASvar2} (with $D=W^2$) and \eqref{AScorr2}
for $d$-dimensional band matrices for $d\leq 3$; 
we also obtain similar results for $d=4$, 
 where the powers of $\eta$ and $E_2 - E_1$ are replaced with a logarithm. 
This rigorously justifies the asymptotics of Silvestrov \cite[Equation (40)]{Sil}, which in turn
reproduced the earlier result of \cite{AS}.  
For technical reasons, we have to restrict ourselves to the
regime $\eta \gg W^{-d/3}$. For convenience, we also assume that
 $L \gg W^{1+d/6}$, which guarantees that $L\gg \ell_\eta$ (or, equivalently, $\eta\gg \eta_c$).
Hence we work in the diffusive regime. However, our method may be easily extended to the case $L\ll \ell_\eta$ as well (see Remark \ref{rem: diffusion} and Section \ref{sec:Wigner} below).
We also show that for $d\geq 5$ the universality of the formulas \eqref{ASvar2} and \eqref{AScorr2} 
 breaks down, and the variance and the correlation functions of ${\cal N}_\eta$  depend on
the detailed structure of the band matrix. We also compute the leading correction to the density-density correlation. In summary, we find that for $d = 1,2$ the leading and subleading terms are universal, for $d = 3,4$ only the leading terms are universal, and for $d \geq 5$ the density-density correlation is not universal.

The case $d=2$ is special, since the coefficient of the 
leading term in \eqref{AScorr2} vanishes owing to an algebraic  cancellation.
The first non-vanishing term was predicted in \cite{KrLe}.
We rigorously identify this term in the regime $E_2-E_1 \gg \eta \gg W^{-2/3}$, 
and find a substantial discrepancy between it and the prediction of \cite{KrLe}.

For an outline of our proof, and the relation between this paper and its companion \cite{EK4}, see Section \ref{sec:struc}.

\subsection{Summary of previous related results}

Our analysis is valid in the mesoscopic regime, i.e.\ when $\delta \ll \eta \ll 1$,
and concerns only density fluctuations.
For completeness, we mention what was previously known in this and other regimes.

\subsubsection*{Macroscopic statistics}
In the macroscopic regime,  $\eta \asymp 1$, the quantity $\cal N_\eta$ should fluctuate on the scale $(L/W)^{d/2}$ according to \eqref{ASvar2}.
For the Wigner case, $L=W$, it has been proved that a smoothed version of  $\cal N_\eta$, 
 the linear statistics of eigenvalues $\sum_i \phi(\lambda_i) = \tr \phi(H)$,
 is asymptotically Gaussian. The first result in this direction for analytic $\phi$
was given in \cite{SS2}, and this was later extended by several authors to more general test functions; see \cite{SoWo} for the latest result.
The first central limit theorem  for matrices with a nontrivial spatial structure and for 
polynomial test functions was proved in \cite{AZ}.  Very recently, it was proved in \cite{SoshLi} for one-dimensional band matrices that, provided that $\phi\in C^1(\R)$, the quantity $\tr \phi(H)$ is asymptotically Gaussian 
with variance of order $(L/W)^{d/2}$.  For a complete list of references in this direction, see \cite{SoshLi}.

\subsubsection*{Mesoscopic statistics}

The asymptotics \eqref{AScorr1} in the completely 
mean-field case, corresponding to Wigner matrices (i.e.\ $W=L$ so that $\eta_c\asymp 1$), 
was  proved
 in \cite{Kho1, Kho2}; see the
remarks following Theorem~\ref{thm: Wigner} for more details  about this work.  
We note  that the formula \eqref{AScorr2} for random band matrices with $d=1$ was derived
in \cite{Ayadi}, using an unphysical double limit procedure,  in which the limit $L \to \infty$ was first computed for a fixed $\eta$, and subsequently the limit of small $\eta$ was taken.
Note that the mesoscopic  correlations cannot be recovered after the limit $L \to \infty$. Hence the result of \cite{Ayadi} describes only the macroscopic, and not the mesoscopic, correlations.

\subsubsection*{Local spectral statistics}

Much less is known about the local spectral statistics of random band matrices, even for $d=1$.
The Tracy-Widom law at the spectral edge was proved in \cite{So1}. Based on a computation of the 
localization length, the metal-insulator
transition is predicted to occur at $W^2\asymp L$; see \cite{FM} for a non-rigorous 
argument and \cite{Sch, EK1, EK2, EKYY3} for the best currently known lower and upper bounds.
 Hence,  the local spectral statistics are expected to be governed by the sine kernel
 from random matrix theory in the regime $W^2\gg L$.
Very recently, the sine kernel was proved \cite{Scher13} for a 
special  Hermitian Gaussian random band matrix with band width $W$ comparable with $L$.
Universality for a more general class of band matrices but with an additional tiny 
mean-field component was proved in \cite{EKYY4}. We also mention that
 the local correlations of determinants of a special  Hermitian Gaussian random band matrix
have been shown to follow the sine kernel \cite{Scher12}, up to the expected threshold $L\lesssim W^2$.

\subsection{Transition to Poisson statistics}

The diagrammatic calculation of \cite{AS}
uses the diffusion approximation, and formulas \eqref{ASvar1}--\eqref{AScorr2} are
 supposed to be valid in  the delocalized  regime. Nevertheless, our results also hold in the localized regime, in particular
even $d=1$ and for $L\gg W^8$, in which case the eigenvectors are known to be localized \cite{Sch}.
In this regime, and for $\eta\gg W^{-1/3}$, we also prove \eqref{ASvar2} and \eqref{AScorr2}.
Both formulas show  that Poisson statistics do not hold on large mesoscopic scales,
despite the system being in the localized regime.  Indeed, if $\cal N_\eta$ were Poisson-distributed, then we would have
$\var {\cal N_\eta} \asymp {\cal N_\eta} \asymp L\eta$. On the other hand, \eqref{ASvar2} gives $\var \cal N_\eta \asymp L\sqrt{\eta}/W$. We conclude that the prediction of \eqref{ASvar2} for the magnitude of $\var \cal N_\eta$ is much smaller than that predicted by Poisson statistics provided that $\eta\gg W^{-2}$. 

The fact that the Poisson statistics breaks down on
mesoscopic scales is not surprising. Indeed, the basic intuition
behind the  emergence of  Poisson statistics is that eigenvectors belonging to different eigenvalues
are exponentially localized on a scale $\ell \asymp W^2$, typically at different
spatial locations. Hence the associated eigenvalues are independent.
For larger $\eta$, however, the observables depend on many eigenvalues,
which exhibit nontrivial correlations since the supports of their eigenvectors overlap. A simple counting argument 
shows that such overlaps become significant if $\eta \gg 1/\ell$, 
 at which point correlations are expected to develop. In other words, we expect a transition
 to/from Poisson statistics at $\eta \asymp 1/\ell$. 
In the previous paragraph, we noted that \eqref{ASvar2} predicts a transition in the behaviour
 of $\cal N_\eta$ to/from Poisson statistics at $\eta \asymp W^{-2}$. Combining these observations,
 we therefore expect a transition to/from Poisson statistics for $\ell \asymp W^2$. This argument predicts the correct localization length $\ell \asymp W^2$ 
without resorting to Grassmann integration. It remains on a heuristic level, however, since our results
do not cover the full range $\eta \gg W^{-2}$.  We note that this argument may also be applied to $d \geq 2$, in which case it predicts the absence of a transition provided that $W \gg 1$.

The main conclusion of our results is that the local eigenvalue statistics, characterized
by either Poisson or sine kernel statistics, do not in general extend to
mesoscopic scales.  On mesoscopic scales, a 
different kind of universality emerges,
 which is expressed by the Altshuler-Shklovskii formulas. 

\subsubsection*{Conventions}
We use $C$ to denote a generic large positive constant, which may depend 
on some fixed parameters and whose value may change from one expression 
to the next. Similarly, we use $c$ to denote a generic small positive constant. 
We use $a \asymp b$ to mean $c a \leq b \leq C a$ for some constants $c,C > 0$. 
Also, for any finite set $A$ we use $\abs{A}$ to denote the cardinality of $A$. If the implicit constants in the usual notation $O(\cdot)$ depend on some parameters $\alpha$, we sometimes indicate this explicitly by writing $O_\alpha(\cdot)$.

\subsubsection*{Acknowledgements} We are very grateful to Alexander Altland for
detailed discussions on the physics of the problem and for providing references.

\section{Setup and Results} \label{sec: setup}

\subsection{Definitions and assumptions} \label{sec:setup1}
Fix $d \in \N$, the physical dimension of
the configuration space. For $L \in \N$ we define the discrete torus of size $L$ 
\begin{equation*}
\bb T \;\equiv\; \bb T_L^d \;\deq\; \pb{[-L/2, L/2) \cap \Z}^d\,,
\end{equation*}
and abbreviate
\begin{equation}\label{Ndef}
N \;\deq\; \abs{\bb T_L} \;=\; L^d.
\end{equation}
Let $1 \ll W \leq  L$ denote the band width, and define the deterministic matrix $S = (S_{xy})$ through
\begin{equation} \label{step S}
S_{xy} \;\deq\; \frac{\ind{1 \leq \abs{x - y} \leq W}}{M - 1}\,, \qquad M \;\deq\; \sum_{x \in \bb T} \ind{1 \leq \abs{x} \leq W}\,,
\end{equation}
where $\abs{\cdot}$ denotes the periodic Euclidean norm on $\bb T$, i.e.\ $\abs{x} \deq \min_{\nu \in \Z^d} \abs{x + L \nu}_{\Z^d}$. Note that
\begin{equation} \label{link between M and W}
M \;\asymp\; W^d\,.
\end{equation}
The fundamental parameters of our model are the linear dimension of the torus, $L$, and the band width, $W$. The quantities $N$ and $M$ are introduced for notational convenience, since most of our estimates depend naturally on $N$ and $M$ rather than $L$ and $W$. We regard $L$ as the independent parameter, and $W \equiv W_L$ as a function of $L$.

Next, let $A = A^* = (A_{xy})$ be a Hermitian random matrix whose upper-triangular\footnote{We introduce an arbitrary and immaterial total ordering $\leq$ on the torus $\bb T$.} entries $(A_{xy} \col x \leq y)$ are independent random variables with zero expectation. We consider two cases.
\begin{itemize}
\item 
\emph{The real symmetric  case ($\beta = 1$)}, where $A_{xy}$ satisfies $\P(A_{xy} = 1) = \P(A_{xy} = -1) = 1/2$.
\item
\emph{The complex Hermitian case ($\beta = 2$)}, where $A_{xy}$ is uniformly distributed on the unit circle $\bb S^1 \subset \C$.
\end{itemize}
Here the index $\beta = 1,2$ is the customary symmetry index of random matrix theory.

We define the \emph{random band matrix} $H = (H_{xy})$ through
\begin{equation}\label{HA}
H_{xy} \;\deq\; \sqrt{S_{xy}} \, A_{xy}\,.
\end{equation}
Note that $H$ is Hermitian and $\abs{H_{xy}}^2 = S_{xy}$, i.e.\ $\abs{H_{xy}}$ is deterministic. Moreover, we have for all $x$
\begin{equation} \label{normalization of S}
\sum_{y} S_{xy} \;=\; \frac{M}{M - 1}\,.
\end{equation}
With this normalization, as $N, W\to \infty$ the bulk of the spectrum of  $H/2$  lies in  $[-1,1]$ 
and the eigenvalue density is given by the Wigner semicircle law with density
\begin{equation} \label{def_sc_density}
\nu(E) \;\deq\; \frac{2}{\pi}\sqrt{1-E^2} \qquad \text{for} \quad E \in [-1,1]\,.
\end{equation}

Let $\phi$ be a smooth, integrable, real-valued function on $\R$ satisfying $\int \phi(E) \, \dd E \neq 0$. 
We call such functions $\phi$ \emph{test functions}. We also require that our test functions $\phi$ satisfy one of the two following conditions.
\begin{itemize}
\item[\textbf{(C1)}]
$\phi$ is the Cauchy kernel
\begin{equation} \label{Cauchy}
\phi(E) \;=\; \im \frac{2}{E - \ii} \;=\; \frac{2}{E^2 + 1}\,.
\end{equation}
\item[\textbf{(C2)}]
For every $q > 0$ there exists a constant $C_q$ such that
\begin{equation} \label{non_Cauchy}
\abs{\phi(E)}  \;\leq\; \frac{C_q}{1 + \abs{E}^q}\,.
\end{equation}
\end{itemize}

A typical example of a test function $\phi$ satisfying \textbf{(C2)} is the Gaussian $\phi(E) = \sqrt{2 \pi} \, \ee^{-E^2 / 2}$. We introduce the rescaled test function $\phi^\eta(E) \deq \eta^{-1} \phi(\eta^{-1} E)$. We shall be interested in correlations of observables depending on $E \in (-1,1)$ of the form
\begin{equation*}
Y^\eta_\phi(E) \;\deq\; \frac{1}{N} \sum_i \phi^\eta(\lambda_i - E) \;=\; \frac{1}{N} \tr \phi^\eta (H/2 - E)\,,
\end{equation*}
where $\lambda_1, \dots, \lambda_N$ denote the eigenvalues of $H/2$. 
(The factor $1/2$ is a mere convenience, chosen because, as noted above, the asymptotic spectrum of $H/2$ is the interval $[-1,1]$.)
The quantity $Y^\eta_\phi(E)$ 
is the smoothed local density of states around the energy $E$ on the scale $\eta$.
We always choose
\begin{equation*}
\eta \;=\; M^{-\rho}
\end{equation*}
for some fixed $\rho \in (0,1/3)$, and we frequently drop the index $\eta$ from our notation. The strongest
results are for large $\rho$, so that one should think of $\rho$ being slightly less than $1/3$. 

We are interested in the correlation function of the local densities of states, $Y^\eta_{\phi_1}(E_1)$ and  $Y^\eta_{\phi_2}(E_2)$, around two energies
$E_1 \leq E_2$. We shall investigate two regimes: $\eta \ll E_2 - E_1$ and $E_1 = E_2$. In the former regime, 
we prove that the correlation decay in the energy difference $E_2-E_1$ is universal 
(in particular, independent of $\eta$, $\phi_1$, and $\phi_2$), and we compute the correlation function explicitly. In the latter regime, we prove that the variance has a universal dependence on $\eta$, and depends on $\phi_1$ and $\phi_2$ via their inner product in a homogeneous Sobolev space.

The case \textbf{(C2)} for our test functions is the more important  one, since we are typically interested in the statistics of eigenvalues contained in an interval of size $\eta$. The Cauchy kernel from the case \textbf{(C1)} has a heavy tail, which introduces unwanted correlations arising from the overlap of the test functions and not from the long-distance correlations that we are interested in.
Nevertheless, we give our results also for the special case \textbf{(C1)}. We do this for two reasons. First, the case \textbf{(C1)} is pedagogically useful,  since it results in a considerably simpler computation of the main term 
 (see \cite[Section 3]{EK4} for more details). 
Second, the case \textbf{(C1)} is often the only one considered in the physics literature (essentially because it corresponds to the imaginary part of the resolvent of $H$). Hence, our results in particular decouple the correlation effects arising from the heavy tails of the test functions from those arising from genuine mesoscopic correlations. As proved in Theorem \ref{thm: Theta 1} below, the effect of the heavy tail is only visible in the leading nonzero corrections for $d = 2$.

For simplicity, throughout the following we assume that both of our test functions satisfy \textbf{(C1)} or both satisfy \textbf{(C2)}. 
Since the covariance is bilinear,
one may also consider more general test functions that are linear combinations of the cases \textbf{(C1)} and \textbf{(C2)}.

\begin{definition}
Throughout the following we use the quantities $E_1, E_2 \in (-1,1)$ and
\begin{equation*}
E \;\deq\; \frac{E_1 + E_2}{2}\,, \qquad \omega \;\deq\; E_2 - E_1
\end{equation*}
interchangeably. Without loss of generality we always assume that $\omega \geq 0$.
\end{definition}

For the following we choose and fix a positive constant $\kappa$. We always assume that
\begin{equation} \label{D leq kappa}
E_1, E_2 \;\in\; [-1 + \kappa, 1 - \kappa] \,, \qquad \omega \;\leq\; c_*
\end{equation}
for some small enough positive constant $c_*$ depending on $\kappa$. These restrictions are required since the nature of the correlations changes near the spectral edges $\pm 1$. Throughout the following we regard the constants $\kappa$ and $c_*$ as fixed and do not track the dependence of our estimates on them.

We now state our results on the density-density correlation for band matrices in the diffusive
regime (Section \ref{sec:band_results}). The proofs are given in the current paper and its companion \cite{EK4}.  As a reference, we also state similar results for Wigner matrices, corresponding to the mean-field regime (Section \ref{sec:Wigner}).

\subsection{Band matrices} \label{sec:band_results}

Our first theorem gives the leading behaviour of the density-density correlation function in terms of a function $\Theta^\eta_{\phi_1, \phi_2}(E_1, E_2)$, which is explicit but has a complicated form.
In the two subsequent theorems we determine the asymptotics of this function in two physically relevant regimes, where its form simplifies substantially. We use the abbreviations
\begin{equation} \label{def_avg}
\avg{X} \;\deq\; \E X\,, \qquad \avg{X \mspace{1mu} ; Y} \;\deq\; \E (XY) - \E X \, \E Y\,.
\end{equation}

\begin{theorem}[Density-density correlations] \label{thm: main result}
Fix $\rho \in (0, 1/3)$ and $d \in \N$, and set $\eta \deq M^{-\rho}$. Suppose that the test functions $\phi_1$ and $\phi_2$ satisfy either both \textbf{(C1)} or both \textbf{(C2)}.
Suppose moreover that
\begin{equation} \label{LW_assump}
W^{1 + d/6} \;\leq\; L \;\leq\; W^C
\end{equation}
for some constant $C$.

Then there exist a constant $c_0 > 0$ and a function $\Theta_{\phi_1,\phi_2}^\eta(E_1, E_2)$ -- which is given explicitly in \eqref{def:theta} and \eqref{common expression for V(D)} below, and whose asymptotic behaviour is derived in Theorems \ref{thm: Theta 2} and \ref{thm: Theta 1} below -- such that, for any $E_1, E_2$ satisfying \eqref{D leq kappa} for small enough $c_* > 0$,  the local density-density correlation satisfies
\begin{equation} \label{EY_result}
\frac{\avg{Y^\eta_{\phi_1}(E_1) \, ; Y^\eta_{\phi_2}(E_2)}}{\avg{Y^\eta_{\phi_1}(E_1)} \avg{Y^\eta_{\phi_2}(E_2)}} \;=\; \frac{1}{(LW)^d} \pB{\Theta_{\phi_1,\phi_2}^\eta(E_1,E_2) + O \pb{M^{-c_0} R_2(\omega + \eta)}}\,,
\end{equation}
where we defined
\begin{equation} \label{def_R2}
R_2(s) \;\deq\; 1 + \ind{d = 1} s^{-1/2} + \ind{d = 2} \abs{\log s}\,.
\end{equation}

Moreover, if $\phi_1$ and $\phi_2$ are analytic in a strip containing the real axis (e.g.\ as in the case \textbf{(C1)}), we may replace the upper bound $L \leq W^C$ in \eqref{LW_assump} $L \leq \exp(W^c)$ for some small constant $c > 0$.
\end{theorem}

We shall prove that the error term in \eqref{EY_result} is smaller than the main term $\Theta$ for all $d \geq 1$. The main term $\Theta$ has a simple, and universal, explicit form only for $d \leq 4$. Why $d=4$ is the critical dimension for the universality of the  correlation decay is explained in Section \ref{sec: heuristic leading term} below. The two following theorems give the leading behaviour of the function $\Theta$ for $d \leq 4$ in the two regimes $\omega = 0$ and $\omega \gg \eta$. In fact, one may also compute the subleading corrections to $\Theta$. These corrections turn out to be universal for $d \leq 2$ but not for $d \geq 3$; see Theorem \ref{thm: Theta 1} and the remarks following it.


In order to describe the leading behaviour of the variance, i.e.\ the case $\omega = 0$, we introduce the Fourier transform
\begin{equation*}
\phi(E) \;=\; \int_{\R} \dd t \, \ee^{-\ii E t} \, \wh \phi(t)\,, \qquad \wh \phi(t) \;=\; \frac{1}{2 \pi} \int_{\R}
 \dd E \, \ee^{\ii E t} \, \phi(E) \,.
\end{equation*}
For $d \leq 4$ we define the quadratic form $V_d$ through
\begin{equation} \label{def_V_d}
V_d(\phi_1, \phi_2) \;\deq\; \int_\R \dd t \, \abs{t}^{1 - d/2} \, \ol {\wh \phi_1(t)} \, \wh \phi_2(t) \qquad (d \leq 3) \,, \qquad \qquad V_4(\phi_1,\phi_2) \;\deq\; 2 \ol {\wh \phi_1(0)} \, \wh \phi_2(0)\,.
\end{equation}
Note that $V_d(\phi_1, \phi_2)$ is real since both $\phi_1$ and $\phi_2$ are.  In the case \textbf{(C1)} we have the explicit values
\begin{equation*}
V_0(\phi_1, \phi_2) \;=\; \frac{1}{2}  \,, \quad \;
V_1(\phi_1, \phi_2) \;=\; \frac{\sqrt{\pi}}{2 \sqrt{2}}\,, \quad \;
V_2(\phi_1, \phi_2) \;=\; 1\,, \quad \;
V_3(\phi_1, \phi_2) \;=\; \sqrt{2 \pi}\,, \quad \;
V_4(\phi_1, \phi_2) \;=\; 2\,.
\end{equation*}
For the following statements of results, we recall the density $\nu(E)$ of the semicircle law from \eqref{def_sc_density}, and remind the reader of the index $\beta = 1,2$ describing the symmetry class of $H$.

\begin{theorem}[The leading term $\Theta$ for $\omega = 0$] \label{thm: Theta 2}
Suppose that the assumptions in the first paragraph of Theorem \ref{thm: main result} hold, and let $\Theta_{\phi_1,\phi_2}^\eta(E_1,E_2)$ be the function from Theorem \ref{thm: main result}. Suppose in addition that $\omega = 0$. Then there exists a constant $c_1 > 0$ such that the following holds for $E = E_1 = E_2$  satisfying \eqref{D leq kappa}.
\begin{enumerate}
\item
For $d =1,2,3$ we have
\begin{equation} \label{Theta3D0}
\Theta_{\phi_1,\phi_2}^\eta(E,E) \;=\; \frac{(d + 2)^{d/2} }{2 \beta \pi^{2 + d} \nu(E)^4} \pbb{\frac{\eta}{\nu(E)}}^{d/2 - 2} \pb{V_d(\phi_1, \phi_2) + O(M^{-c_1})}\,.
\end{equation}
\item
For $d = 4$ we have
\begin{equation} \label{Theta4D0}
\Theta_{\phi_1,\phi_2}^\eta(E,E) \;=\; \frac{36 }{\beta \pi^6 \nu(E)^4}
 \pb{   V_4(\phi_1, \phi_2) \abs{\log\eta} +
O(1)}\,.
\end{equation}
\end{enumerate}
\end{theorem}

In order to describe the behaviour of $\Theta$ in the regime $\omega \gg \eta$, for $d = 1,2,3$ we introduce the constants 
\begin{equation}\label{def_K13}
K_d \;\deq\;  2 \re  \int_{\R^d} \frac{\dd x}{(\ii+\abs{x}^2)^2}\,;
\end{equation}
explicitly, 
\begin{equation*} 
K_1 = - \frac{\pi}{\sqrt{2}}\,, \qquad  K_2 = 0 \,,\qquad
K_3 = \sqrt{2} \pi^2\,.
\end{equation*}

\begin{theorem}[The leading term $\Theta$ in the regime $\omega \gg \eta$] \label{thm: Theta 1}
Suppose that the assumptions in the first paragraph of Theorem \ref{thm: main result} hold, and let $\Theta_{\phi_1,\phi_2}^\eta(E_1,E_2)$ be the function from Theorem \ref{thm: main result}.  Suppose in addition that
\begin{equation} \label{eta_Delta_2}
\eta \;\leq\; M^{-\tau} \omega
\end{equation}
for some arbitrary but fixed $\tau > 0$.  Then there exists a constant $c_1 > 0$ such that the following holds for $E_1,E_2$ satisfying \eqref{D leq kappa} for small enough $c_* > 0$. 

\begin{enumerate}
\item
For $d = 1, 2, 3$ we have
\begin{equation} \label{Theta_23}
\Theta_{\phi_1,\phi_2}^\eta(E_1,E_2) \;=\; \frac{(d + 2)^{d/2}}{2 \beta \pi^{2 + 3 d / 2}\nu(E)^4} \pbb{\frac{\omega}{\nu(E)}}^{d/2 - 2} \pB{K_d + O \pb{\sqrt{\omega} + M^{-c_1}}}\,.
\end{equation}
\item
For $d = 2$  \eqref{Theta_23} does not identify 
the leading term since  $K_2=0$. The leading nonzero correction to the
vanishing leading term is 
\begin{equation}\label{C1d2}
\Theta_{\phi_1,\phi_2}^\eta(E_1,E_2) \;=\; \frac{8}{\beta \pi^5 \nu(E)^4}
\pbb{\pi \nu(E) \, \frac{\eta } {\omega^2 + 4 \eta^2}  - \frac{\abs{\log \omega}}{3} + O(1)}
\end{equation} 
 in the  case \textbf{(C1)}
and 
\begin{equation}\label{C2d2}
\Theta_{\phi_1,\phi_2}^\eta(E_1,E_2) \;=\; \frac{8}{\beta \pi^5 \nu(E)^4}
\pbb{  - \frac{\abs{\log \omega}}{3} + O(1)}\, 
\end{equation}
in the case \textbf{(C2)}.
\item
For $d = 4$ we have
\begin{equation} \label{Theta_4}
\Theta_{\phi_1,\phi_2}^\eta(E_1, E_2) \;=\; \frac{36}{\beta \pi^6 \nu(E)^4} \pb{ \abs{\log \omega} + O (1)
}\,.
\end{equation}
\end{enumerate}
\end{theorem}

Note that the leading non-zero terms in the expressions \eqref{Theta3D0}, \eqref{Theta4D0}, \eqref{Theta_23}--\eqref{Theta_4} 
are much larger than the additive error term in \eqref{EY_result}. Hence,
Theorems \ref{thm: main result} and \ref{thm: Theta 2} give a proof of the first Altshuler-Shklovskii formula, \eqref{ASvar2}.
Similarly, Theorems \ref{thm: main result} and \ref{thm: Theta 1} give a proof of the second Altshuler-Shklovskii formula, \eqref{AScorr2}.

The additional term in \eqref{C1d2} as compared to \eqref{C2d2} originates from the heavy Cauchy tail in the test functions $\phi_1, \phi_2$ at large distances. In Theorem \ref{thm: Theta 1} (ii), we give the leading correction, of order $\abs{\log \omega}$, to the vanishing main term for $d = 2$. Similarly, for $d = 1$ one can also derive the leading correction to the nonzero main term (which is of order $\omega^{-3/2}$). This correction turns out to be of order $\omega^{-1/2}$; we omit the details.

\begin{remark}
The leading term  in \eqref{EY_result}  originates from the  so-called one-loop diagrams in the terminology of physics.
The next-order term after the vanishing leading term for $d=2$ (recall that to $K_2=0$) 
 was first computed by Kravtsov and Lerner \cite[Equation (13)]{KrLe}.
They found that  for $\beta = 1$ 
 it is of order $ (LW)^{-2} W^{-2} \omega^{-1}$
and for $\beta = 2$  even smaller, of order $ (LW)^{-2} W^{-4} \omega^{-1}$.
Part (ii) of Theorem \ref{thm: Theta 1} shows that, at least in the regime $\omega\gg \eta \gg M^{-1/3} = W^{-2/3}$,
the  true  behaviour is much larger. The origin of this term is a more precise
computation of the one-loop diagrams, in contrast to \cite{KrLe} where the authors
attribute the next-order term to the two-loop diagrams. (See \cite[Section 3]{EK4} for more details.) 
\end{remark}

\begin{remark}
If the distribution of the eigenvalues $\lambda_i$ of $H/2$ were governed by Poisson statistics, the behaviour of the covariance \eqref{EY_result} would be very different. Indeed, suppose that $\{\lambda_i\}$ is a stationary Poisson point process with intensity $N$. Then, setting $Y^\eta_\phi(E) \deq \frac{1}{N} \sum_{\alpha} \phi^\eta(\lambda_i - E)$  and supposing that $\int \phi_1 = \int \phi_2 = 1$,  we find
\begin{equation*}
\frac{\avg{Y^\eta_{\phi_1}(E_1) \, ; Y^\eta_{\phi_2}(E_2)}}{\avg{Y^\eta_{\phi_1}(E_1)} \avg{Y^\eta_{\phi_2}(E_2)}} \;=\; \frac{1}{N\eta} \big( \phi_1\ast \wt \phi_2\big) \pbb{\frac{E_2-E_1}{\eta}}
  \;=\; \frac{1}{N} \big( \phi_1^\eta\ast \wt \phi_2^\eta)(\omega) \,,
\end{equation*}
where $\wt\phi_2 (x) \deq \phi_2(-x)$.
This is in stark contrast to \eqref{Theta3D0}, \eqref{Theta4D0}, \eqref{Theta_23}, and \eqref{Theta_4}. In particular, in the case $\omega \gg \eta$ the behaviour of the covariance on $\omega$ depends on the tails of $\phi_1$ and $\phi_2$, unlike in \eqref{Theta_23} and \eqref{Theta_4}.
Hence, if $\phi_1$ and $\phi_2$ are compactly supported then the covariance for the Poisson process is zero, while for the eigenvalue process of a band matrix
it has a power law decay in $\omega$.
\end{remark}

\begin{remark}\label{rem1}
We emphasize that Theorem \ref{thm: main result} is true under the sole restrictions \eqref{D leq kappa} on $\omega$ and $E_1,E_2$. 
However, the leading term $\Theta$ only has a simple and universal form in the two (physically relevant) regimes $\omega = 0$ and
 $\eta \ll \omega \ll 1$ of Theorems \ref{thm: Theta 2} and \ref{thm: Theta 1}. If neither of these conditions holds, 
the expression for $\Theta$ is still explicit but much more cumbersome and opaque. It is given by the sum of the values of eight (sixteen for $\beta = 1$)  skeleton graphs, after a ladder resummation; these skeleton graphs are referred to as the ``dumbbell skeletons'' $D_1, \dots, D_8$ in Section \ref{sec: main argument} below, and are depicted in Figure \ref{fig: dumbbell} below. (They are the analogues of the \emph{diffusion} and \emph{cooperon} Feynman diagrams in the physics literature.) 
\end{remark}

\begin{remark} \label{rem: diffusion}
The upper bound in the 
assumption \eqref{LW_assump} is technical and can be relaxed.
The lower bound in \eqref{LW_assump}, however, is a natural restriction, and is related to the \emph{quantum diffusion} generated by the band matrix $H$. In \cite{EK1}, it was proved that the propagator $\abs{(\ee^{-\ii t H/2})_{x0}}$ behaves diffusively for $1 \ll t \ll M^{1/3} \asymp W^{d/3}$, whereby the spatial extent of the diffusion is $x \asymp \sqrt{t} W \ll W^{1+d/6}$. Similarly, in \cite{EKYY3}, it was proved that the resolvent $\absb{(H/2 - E - \ii \eta)^{-1}_{x0}}^2$ has a nontrivial profile on the scale $x\asymp \eta^{-1/2} W$.
(Note that $\eta$ is the conjugate variable to $t$, i.e.\ the time evolution up to time $t$
describes the same regime as the resolvent with a spectral parameter $z$ whose
imaginary part is $\eta \asymp 1/t$.)
Since in Theorem \ref{thm: main result} we assume that $\eta \gg M^{-1/3} \gg W^{-d/3}$, the condition \eqref{LW_assump} simply states that the diffusion profile associated with the spectral resolution $\eta$ does not reach the edge of the torus $\bb T$. Thus, the lower bound in \eqref{LW_assump} imposes a 
 regime in which boundary effects are irrelevant. Hence  we are in the diffusive regime -- a basic
assumption  of  the Altshuler-Shklovskii formulas \eqref{ASvar2} and 
\eqref{AScorr2}. 
\end{remark}

\subsection{A remark on  Wigner matrices} \label{sec:Wigner}

Our method can easily be applied to the case where the lower bound 
in \eqref{LW_assump} is not satisfied. In this case, however, the leading behaviour $\Theta^\eta(E_1,E_2)$ is modified
by boundary effects. To illustrate this phenomenon, we state the analogue of Theorems \ref{thm: main result}--\ref{thm: Theta 1} 
for the case of  $W = L$. In this case,  the physical dimension $d$ in Section \ref{sec:setup1} is irrelevant.
 The off-diagonal entries of $H$ are all identically distributed, i.e.\ $H$ is a standard Wigner matrix (neglecting the irrelevant diagonal entries), and we have $M = N - 1$, and $S = N (N - 2)^{-1} (\f e \f e^* - N^{-1})$ where $\f e \deq N^{-1/2} (1,1, \dots, 1)^*$. In particular, $H$ is a mean-field model in which the geometry of $\bb T$ plays no role; the effective dimension is $d=0$. In this case \eqref{EY_result} remains valid, and we get the following result.

\begin{theorem}[Theorems \ref{thm: main result}--\ref{thm: Theta 1} for Wigner matrices] \label{thm: Wigner}
Suppose that $W = L = N$. Fix $\rho \in (0, 1/3)$ and set $\eta \deq N^{-\rho}$. Suppose that the test functions $\phi_1$ and $\phi_2$ satisfy either both \textbf{(C1)} or both \textbf{(C2)}. Then there exists a constant $c_0 > 0$ and a function $\wt \Theta_{\phi_1,\phi_2}^\eta(E_1, E_2)$ such that for any $E_1,E_2$ satisfying \eqref{D leq kappa} for small enough $c_* > 0$  the following holds.

\begin{enumerate}
\item
The local density-density correlations satisfy
\begin{equation}
\frac{\avg{Y^\eta_{\phi_1}(E_1) \, ; Y^\eta_{\phi_2}(E_2)}}{\avg{Y^\eta_{\phi_1}(E_1)} \avg{Y^\eta_{\phi_2}(E_2)}} \;=\; \frac{1}{N^2} \pB{\wt \Theta_{\phi_1,\phi_2}^\eta(E_1,E_2) + O \pb{N^{-c_0} (\omega + \eta)^{-1}}} \,.
\end{equation}
\item
If \eqref{eta_Delta_2} holds then
\begin{equation*}
\wt \Theta_{\phi_1, \phi_2}^\eta(E_1,E_2) \;=\; \frac{4}{\beta \pi^4  \nu(E)^4} \, \frac{1}{\omega^2} \pB{
-1 + O \pb{\sqrt{\omega} + N^{-\tau/2}}}\,.
\end{equation*}
\item
If $\omega = 0$ then
\begin{equation*}
\wt \Theta_{\phi_1, \phi_2}^\eta(E,E) \;=\; \frac{2}{\beta \pi^{4} \nu(E)^4} \, \frac{1}{\eta^2} \, \pb{V_0(\phi_1, \phi_2) + O(N^{-c_0})}\,,
\end{equation*}
where $V_0$ was defined in \eqref{def_V_d}. 
\end{enumerate}
\end{theorem}

The proof of Theorem \ref{thm: Wigner} proceeds along the same lines as that of Theorems \ref{thm: main result}--\ref{thm: Theta 1}. In fact, the simple form of $S$ results in a much easier proof; we omit the details. We remark that a result analogous to parts (i) and (ii) of Theorem \ref{thm: Wigner}, in the case where $\phi_1 = \phi_2$ are given by \eqref{Cauchy}, was derived in \cite{Kho1, Kho2}. More precisely, in \cite{Kho1}, the authors assume that $H$ is a GOE matrix and derive (i) and (ii) of Theorem \ref{thm: Wigner} for any $0 < \rho < 1$; in \cite{Kho2}, they extend these results to arbitrary Wigner matrices under the additional constraint that $0 < \rho < 1/8$. Moreover, results analogous to (iii) for the Gaussian Circular Ensembles were proved in \cite{Sosh00}. More precisely, in \cite{Sosh00} it is proved that in Gaussian Circular Ensembles the appropriately scaled mesoscopic linear statistics $Y^\eta_\phi(E)$ with $1/N \ll \eta \ll 1$ are asymptotically Gaussian with variance proportional to $V_0(\phi,\phi)$. We remark that for random band matrices the mesoscopic linear statistics also satisfy a Central Limit Theorem; see \cite[Corollary 2.6]{EK4}.

\subsection{Structure of the proof}\label{sec:struc}

The starting point of the proof is to use the Fourier transform to
rewrite $\tr \phi^\eta(H/2-E)$, the spectral density on scale $\eta$, 
in terms of $\ee^{\ii tH}$ up to times $\abs{t}\lesssim \eta^{-1}$. The large-$t$ behaviour
of this unitary group has been extensively analysed in \cite{EK1, EK2}
by developing a graphical expansion method which we also use in this paper. 
 The main difficulty is to control highly oscillating sums. 
Without any resummation, the sum of the absolute values of
 the summands diverges exponentially in $L$, although their actual sum remains bounded. 
The leading divergence in this expansion is removed using a resummation that is implemented by 
 expanding $\ee^{\ii tH}$ in terms of Chebyshev polynomials of $H$ instead of  powers of $H$. 
 This step, motivated by \cite{FS}, is algebraic and requires the  deterministic condition 
 $\abs{H_{xy}}=1$.
(The removal of this condition is possible, but requires substantial
 technical efforts that mask the essence of
the argument; see Section \ref{sec:outlook}). In the jargon of diagrammatic 
perturbation theory, this resummation step corresponds to the self-energy renormalization.

The goal of  \cite{EK1, EK2} was to show that the unitary propagator  $\ee^{\ii tH}$ can be
described by a diffusive equation on large space and time scales. This analysis 
identified only the leading behaviour of $\ee^{\ii tH}$, which was sufficient
to prove quantum diffusion emerging from the unitary time evolution.
The quantity studied in the current paper -- the
local density-density correlation -- is considerably more difficult to analyse
because it arises from higher-order terms of $\ee^{\ii tH}$ than the quantum diffusion.
Hence, not only does the leading term have to be computed more precisely,
 but the error estimates also require a much more delicate analysis. 
 In fact,  we have to perform a second algebraic resummation procedure,
where oscillatory sums corresponding
to families of specific subgraphs, the so-called ladder subdiagrams, are
 bundled together and computed with high
precision.
Estimating individual ladder graphs in absolute value is not affordable: a term-by-term estimate is possible only 
after this second renormalization step.
Although the expansion in nonbacktracking powers of $H$ is the same as in \cite{EK1, EK2}, our proof in fact has little in common with that of \cite{EK1, EK2}; the only similarity is the basic graphical language. In contrast to \cite{EK1, EK2}, almost all of the work in this paper involves controlling 
oscillatory sums, both in the error estimates and in the computation of the main term.

The complete proof is given in the current paper and its companion \cite{EK4}.
In order to highlight the key ideas, the current paper contains the proof
assuming three important simplifications,  given precisely in \textbf{(S1)}--\textbf{(S3)} in Section \ref{sec_41} below.
They concern certain specific terms in the multiple summations arising from our diagrammatic  expansion. 
 Roughly, these simplifications amount
 to only dealing with typical summation label configurations (hence ignoring exceptional label coincidences) and restricting the summation over all partitions to a summation over pairings. As explained in \cite{EK4}, dealing with exceptional label configurations and non-pairings requires significant additional efforts, which are however largely unrelated to the essence of the argument presented in the current paper. How to remove these simplifications, and hence complete the proofs, is explained in \cite{EK4}. 
In addition,  the precise calculation of the leading term is also given in \cite{EK4}; 
in the current paper we give a sketch of the calculation 
(see Section \ref{sec: heuristic leading term} below).

We close this subsection by noting that the restriction $\rho< 1/3$ for the exponent of $\eta = M^{-\rho}$  is technical 
 and stems from a fundamental combinatorial fact  that underlies our proof---the so-called \emph{2/3-rule}.
The 2/3-rule was introduced in \cite{EK1, EK2} and is stated in the current context in Lemma \ref{lem:2/3_rule} below. In \cite[Section 11]{EK1}, it was shown that the 2/3-rule is sharp, and is in fact saturated for a large family of graphs. For more details on the 2/3-rule and how it leads the the restriction on $\rho$, we refer to the end Section \ref{sec:large_Sigma} below.

\subsection{Outlook and generalizations} \label{sec:outlook}

We conclude this section by summarizing some extensions of our results from the companion paper \cite{EK4}. First, our results easily extend from the two-point correlation functions of  \eqref{EY_result} to arbitrary $k$-point correlation functions of the form
\begin{equation*}
\E \prod_{i = 1}^k \pBB{\frac{Y^\eta_{\phi_i}(E_i) - \E Y^\eta_{\phi_i}(E_i)}{\E Y^\eta_{\phi_i}(E_i)}}\,.
\end{equation*}
In \cite[Theorem 2.5]{EK4}, we prove that the joint law of the smoothed densities $Y^\eta_{\phi_i}(E_i)$
 is asymptotically Gaussian with covariance matrix $(\Theta^\eta_{\phi_i, \phi_j}(E_i, E_j))_{i,j}$, given by the Altshuler-Shklovskii formulas. This result may be regarded as a Wick theorem for the mesoscopic densities, i.e.\ a central limit theorem for the mesoscopic linear statistics of
eigenvalues. In particular, if $E_1 = \cdots = E_k$, the finite-dimensional marginals of the process $(Y^\eta_\phi(E))_\phi$ converge (after an appropriate affine transformation) to those of a Gaussian process with covariance $V_d(\cdot\,, \cdot)$.

Second, in \cite[Section 2.4]{EK4} we introduce a general family of band matrices, where we allow the second moments $S_{xy} = \E \abs{H_{xy}}^2$ and $T_{xy} = \E H_{xy}^2$ to be arbitrary translation-invariant matrices living on the scale $W$.
In particular, we generalize  the sharp step profile from \eqref{step S}  and relax  the deterministic condition $\abs{A_{xy}} = 1$.  Note that we allow $T_{xy}$ to be arbitrary up to the trivial constraint $\abs{T_{xy}} \leq S_{xy}$, thus embedding the real symmetric matrices and the complex Hermitian matrices into a single  large family of band matrices. In particular, this generalization allows us
to probe the transition from $\beta = 1$ to $\beta = 2$  by rotating the entries of $H$ or by scaling $T_{xy}$.
Note that $S=T$ corresponds to the real symmetric case, while $T=0$ corresponds to
a complex Hermitian case where the real and imaginary parts of the matrix elements
are uncorrelated and have the same variance. We can combine this rotation and scaling
into a two-parameter family of models; roughly, we 
consider $T_{xy} \approx (1 - \varphi) \ee^{\ii \lambda} S_{xy}$ where $\varphi,\lambda \in [0,1]$ are real parameters.  We 
 show that the mesoscopic statistics described by Theorems \ref{thm: Theta 2} and \ref{thm: Theta 1} take on a more complicated form in the case of the general band matrix model; they depend on the additional parameter $\sigma = \lambda^2 + \varphi$, which also characterizes the transition from $\beta = 1$ (small $\sigma$) to $\beta = 2$ (large $\sigma$). 
We refer to \cite[Section 2.4]{EK4} for the details.

\section{The renormalized path expansion}

Since the left-hand side of \eqref{EY_result} is invariant under the scaling $\phi \mapsto \lambda \phi$ for $\lambda \neq 0$, we assume without loss of generality that $\int \dd E \, \phi_i(E) = 2 \pi$ for $i = 1,2$. We shall make this assumption throughout the proof without further mention.

\subsection{Expansion in nonbacktracking powers}

We expand $\phi^\eta(H / 2 - E)$ in nonbacktracking powers $H^{(n)}$ of $H$, defined through
\begin{equation} \label{def: nb}
H^{(n)}_{x_0 x_n} \;\deq\; \sum_{x_1, \dots, x_{n - 1}} H_{x_0 x_1} \cdots H_{x_{n - 1} x_n} \prod_{i = 0}^{n - 2} 
\ind{x_i \neq x_{i + 2}}\,.
\end{equation}
From \cite{EK1}, Section 5, we find that
\begin{equation} \label{H^n and U_n}
H^{(n)} \;=\; U_n(H/2) - \frac{1}{M - 1} U_{n - 2}(H / 2)\,,
\end{equation}
where $U_n$ is the $n$-th Chebyshev polynomial of the second kind, defined through
\begin{equation} \label{definition of Un}
U_n(\cos \theta) \;=\; \frac{\sin (n + 1) \theta}{\sin \theta}\,.
\end{equation}
The identity \eqref{H^n and U_n} first appeared in \cite{FS}.
Note that it requires the deterministic condition $\abs{A_{xy}} = 1$ on the entries of $H$.  However, our basic approach still works even if this condition is not satisfied; in that case the proof is more complicated due to the presence of a variety of error terms in \eqref{H^n and U_n}. See \cite[Section 5.3]{EK4} for more details.

From \cite{EK1}, Lemmas 5.3 and 7.9, we recall the expansion in nonbacktracking powers of $H$.
\begin{lemma} \label{lemma: Uexp}
For $t \geq 0$ we have
\begin{equation} \label{Chebyshev expansion of propagator}
\ee^{-\ii t H/2} \;=\; \sum_{n \geq 0} a_n(t) H^{(n)}\,,
\end{equation}
where
\begin{equation} \label{def of an}
a_n(t) \;\deq\; \sum_{k \geq 0} \frac{\alpha_{n + 2k}(t)}{(M - 1)^k}\,, \qquad \alpha_k(t) \;\deq\; 2 (- \ii)^k 
\frac{k+1}{t} J_{k + 1}(t)
\end{equation}
and $J_\nu$ denotes the $\nu$-th Bessel function of the first kind.
\end{lemma}

Throughout the following we denote by $\arcsin$ the analytic branch of $\arcsin$ extended to the real axis by continuity from the upper half-plane. 
The following coefficients will play a key role in the expansion.
For $n \in \N$ and $E \in \R$ define
\begin{equation*}
\gamma_n(E) \;\deq\; \int_0^\infty \dd t \, \ee^{\ii E t}\, a_n(t)\,.
\end{equation*}

\begin{lemma}
We have
\begin{equation} \label{claim about gamma n}
\gamma_n(E)
\;=\; \frac{2 (-\ii)^n \ee^{\ii (n+1) \arcsin E}}{1 - (M - 1)^{-1} \ee^{2\ii \arcsin E}}\,.
\end{equation}
\end{lemma}
\begin{proof}
Using \eqref{def of an} we find
\begin{equation} \label{computation of ft of an}
\gamma_n(E) \;=\; \sum_{k = 0}^\infty \frac{\int_0^\infty \dd t \, \ee^{\ii E t} \, \alpha_{n + 2k}(t)}{(M - 1)^k} \;=\; \sum_{k = 0}^\infty \frac{2 (-\ii)^n \ee^{\ii (n + 2k +1) \arcsin E}}{(M - 1)^k}\,,
\end{equation}
where in the second step we used the identity
\begin{align} \label{key integral}
\int_0^\infty \dd t \; t^{-1} \ee^{\ii E t} J_\nu(t) \;=\; \frac{1}{\nu} \, \ee^{\ii \nu \arcsin E}\,,
\end{align}
which is an easy consequence of \cite[Formulas 6.693.1--6.693.2]{GR07} and analytic continuation. This concludes the proof.
\end{proof}

Define
\begin{equation} \label{def_F_eta}
F_{\phi_1, \phi_2}^\eta(E_1,E_2) \;\equiv\; F^\eta(E_1,E_2) \;\deq\; \avgb{\tr \phi_1^\eta (H/2 - E_1) \,; \tr \phi_2^\eta (H/2 - E_2)}\,,
\end{equation}
where we used the notation \eqref{def_avg}.
Note that the left-hand side of \eqref{EY_result} may be written as
\begin{equation}\label{ThetatoF}
\frac{\avg{Y^\eta_{\phi_1}(E_1) \, ; Y^\eta_{\phi_2}(E_2)}}{\avg{Y^\eta_{\phi_1}(E_1)} \avg{Y^\eta_{\phi_2}(E_2)}} \;=\; \frac{1}{N^2} \frac{ F^\eta(E_1,E_2)}{\E Y^\eta_{\phi_1}(E_1) \, \E Y^\eta_{\phi_2}(E_2)}\,.
\end{equation}
The expectations in the denominator are easy to compute using the local semicircle law for band matrices; see
Lemma~\ref{lem:EY} below. Our main goal is to compute $ F^\eta(E_1,E_2)$.

Throughout the following we use the abbreviation
\begin{equation}\label{phipsi}
\psi(E) \;\deq\; \phi(-E)\,,
\end{equation}
and define $\psi^\eta$, $\psi_i$, and $\psi_i^\eta$ similarly similarly in terms of $\phi^\eta$, $\phi_i$, and $\phi_i^\eta$.
We also use the notation
\begin{equation} \label{def of convolution}
( \varphi * \chi )(E) \;\deq\; \frac{1}{2 \pi} \int \dd E' \, \varphi(E - E') \, \chi(E')
\end{equation}
to denote convolution.  The normalizing factor $(2 \pi)^{-1}$ is chosen so that $\wh {\varphi * \chi} = \wh \varphi \, \wh \chi$. Observe that
\begin{equation} \label{claim about gamma}
(\psi^\eta * \gamma_n)(E) \;=\; \int_0^\infty \dd t \, \ee^{\ii E t} \, \wh \phi(\eta t)\, a_n(t)\,.
\end{equation}
We note that in the case where $\phi(E) = \frac{2}{E^2 + 1}$, we have $\wh \phi(t) = \ee^{-\abs{t}}$. Hence \eqref{claim about gamma} implies
\begin{equation} \label{Cauchy case identity}
(\psi^\eta * \gamma_n)(E) \;=\; \int_0^\infty \dd t \, \ee^{\ii (E + \ii \eta) t} \, a_n(t) \;=\; \gamma_n(E + \ii \eta)\,.
\end{equation}
This may also be interpreted using the identity
\begin{equation*}
\frac{1}{\pi} \int \dd E' \, \ee^{\ii n \arcsin E'} \frac{\eta}{(E -E')^2 + \eta^2} \;=\; \ee^{\ii n \arcsin(E + \ii \eta)}\,.
\end{equation*}
We now return to the case of a general real $\phi$. Since $\phi$ is real, we have $\ol{\wh \phi(t)} = \wh \phi(-t)$.  We may therefore use Lemma \ref{lemma: Uexp} and Fourier transformation to get
\begin{multline} \label{calculation of phi eta}
\phi^\eta (H/2 - E) \;=\; \int_{-\infty}^\infty \dd t \,  \wh \phi(\eta t) \, \ee^{-\ii t (H/2 - E)}
\;=\; 2 \re \int_0^\infty \dd t \, \wh \phi(\eta t)   \, \ee^{\ii t E}  \, \ee^{-\ii t H/2}
\\
=\; 2 \re  \sum_{n = 0}^\infty H^{(n)} \int_0^\infty \dd t \, \wh \phi(\eta t)  \, \ee^{\ii t E} a_n(t)
\;=\; \sum_{n = 0}^\infty H^{(n)} \, 2 \re (\psi^\eta * \gamma_n)(E)\,,
\end{multline}
where $\re$ denotes the Hermitian part of a matrix, i.e.\ $\re A \deq (A + A^*)/2$, and in the last step we used \eqref{claim about gamma} and the fact that 
$H^{(n)}$ is Hermitian.
We conclude that
\begin{equation} \label{expansion without trunction}
F^\eta(E_1, E_2) \;=\; \sum_{n_1, n_2 \geq 0} 2 \re \pb{(\psi_1^\eta * \gamma_{n_1})(E_1)} \, 2 \re \pb{(\psi_2^\eta * \gamma_{n_2})(E_2)} \,  \avgb{\tr H^{(n_1)}\,; \tr H^{(n_2)}}\,.
\end{equation}

Because the combinatorial estimates of Section \ref{sec: main argument} deteriorate rapidly for $n \gg \eta^{-1}$, it is essential to cut off the terms $n > M^\mu$ in the expansion \eqref{expansion without trunction}, where $\rho < \mu < 1/3$. Thus, we choose a cutoff exponent $\mu$ satisfying $\rho < \mu < 1/3$. All of the estimates in this paper depend on $\rho, \mu$, and $\phi$; we do not track this dependence. The following result gives the truncated version of \eqref{expansion without trunction}, whereby the truncation is done in $n_i$ and in the support of $\wh \phi_i$.

\begin{proposition}[Path expansion with truncation] \label{prop: expansion with trunction}
Choose $\mu < 1/3$ and $\delta > 0$ satisfying $2 \delta < \mu - \rho < 3 \delta $. Define
\begin{equation} \label{definition of gamma tilde}
\wt \gamma_n(E, \phi) \;\deq\; \int_0^{M^{\rho + \delta}} \dd t \, \ee^{\ii E t} \, \wh \phi(\eta t)\, a_n(t)
\end{equation}
and
\begin{equation} \label{truncated series}
\wt F^\eta(E_1,E_2) \;\deq\; \sum_{n_1 + n_2 \leq M^\mu} 2 \re \pb{\wt \gamma_{n_1}(E_1,\phi_1)} \, 2 \re \pb{\wt \gamma_{n_2}(E_2, \phi_2)} \, \avgb{\tr H^{(n_1)}\,; \tr H^{(n_2)}}\,.
\end{equation}
Let $q > 0$ be arbitrary. Then for any $n\in \N$ and recalling \eqref{phipsi} we have the estimates
\begin{equation} \label{gamma - g}
\abs{(\psi_i^\eta * \gamma_{n})(E_i) - \wt \gamma_{n}(E_i, \phi_i)} \;\leq\; C_q M^{-q} \qquad (i = 1,2)
\end{equation}
and
\begin{equation} \label{F - wt F}
\absb{F^\eta(E_1,E_2) - \wt F^\eta(E_1,E_2)} \;\leq\; C_q N^2 M^{-q}.
\end{equation}
Moreover, for all $q > 0$ we have 
\begin{equation} \label{bound on gamma}
\absb{\wt \gamma_{n}(E_i, \phi_i)} +  \absb{(\psi_i^\eta * \gamma_{n})(E_i)} \;\leq\; \min \hb{C, C_q (\eta n)^{-q}}\,.
\end{equation}

If $\phi_1$ and $\phi_2$ are analytic in a strip containing the real axis, the factors $C_q M^{-q}$ on the right-hand sides of \eqref{gamma - g} and \eqref{F - wt F} may be replaced with $\exp(-M^c)$ for some $c > 0$, and the factor $C_q (\eta n)^{-q}$ on the right-hand side of \eqref{bound on gamma} by $\exp(- (\eta n)^c)$.
\end{proposition}

The proof of Proposition \ref{prop: expansion with trunction} is given in \cite[Appendix A]{EK4}. 

\subsection{Heuristic calculation of the leading term} \label{sec: heuristic leading term}

At this point we make a short digression to outline how we compute the leading term of  $F^\eta(E_1,E_2)$.  The precise calculation is given in the companion paper \cite{EK4}.  In Section \ref{sec: main argument} below, we express the right-hand side of \eqref{expansion without trunction} as a sum of terms indexed by graphs, reminiscent of Feynman graphs in perturbation theory.
We prove that the leading contribution is given by a certain set of relatively simple graphs, which we call the \emph{dumbbell skeletons}. Their value $\cal V_{\rm{main}}$ may be explicitly computed and is essentially given by
\begin{equation} \label{skeletons_approx}
 \cal V_{\rm{main}} \;\approx\; \sum_{b_1, b_2,b_3,b_4 = 0}^\infty 2 \re \pb{\gamma_{2 b_1 + b_3 + b_4} * \psi_1^\eta}(E_1) \, 2 \re \pb{\gamma_{2 b_2 + b_3 + b_4} * \psi^\eta_2}(E_2) \tr S^{b_3 + b_4}
\end{equation}
(see \eqref{common expression for V(D)} below for the precise statement). 
The summations represent ``ladder subdiagram resummations'' in the terminology of graphs.
Proving that the contribution of all other graphs is negligible, and hence that \eqref{skeletons_approx} gives the leading behaviour of \eqref{expansion without trunction}, represents the main work, and is done in Sections \ref{sec: main argument}. Assuming that this approximation is valid, we compute \eqref{skeletons_approx} as follows. We use
\begin{equation}\label{2re2re}
(2 \re x_1) (2 \re x_2) = 2  \re (x_1 \ol x_2 + x_1 x_2)
\end{equation}
on the right-hand side of \eqref{skeletons_approx}, and only consider the first resulting term; the second one 
will turn out to be subleading in the regime $\omega, \eta \ll \kappa$, owing to a phase cancellation. Recalling the definition of $\gamma_n$
from \eqref{computation of ft of an},  we find that the summations over $b_1, \dots, b_4$ are simply geometric series, so that
\begin{equation} \label{F_eta_approx}
F^\eta(E_1, E_2) \;\approx\; 2 \re \pBB{ 4 \frac{\ee^{\ii A_1}}{1 + \ee^{2 \ii A_1}} \frac{\ee^{-\ii A_2}}{1 + \ee^{-2 \ii A_2}} \tr \frac{\ee^{\ii (A_1 - A_2)} S}{\pb{1 - \ee^{\ii (A_1 - A_2)} S}^2}} * \psi_1^{\eta}(E_1) * \psi_2^{\eta}(E_2)\,,
\end{equation}
where we abbreviated $A_i \deq \arcsin E_i$, and wrote, by a slight abuse of notation, $(\varphi * \chi)(E) \equiv \varphi(E) * \chi(E)$.

In order to understand the behaviour of this expression, we make
some basic observations about the spectrum of $S$.
Since $S$ is translation invariant, i.e.\ $S_{x y} = S_{x - y \, 0}$, it may be diagonalized by Fourier transformation,
\begin{equation*}
\wh S_W(q) \;\deq\; \sum_{x \in \bb T}\ee^{-\ii q \cdot x / W}  S_{x0} \;\approx\; \int \ee^{-\ii q \cdot x} f(x) \, \dd x\,,
\end{equation*}
 where $f$ is the normalized indicator function of the unit ball in $\R^d$; in the last step we used the definition of $S$ and a Riemann sum approximation. (Note that, since $S$ lives on the scale $W$, it is natural to rescale the argument $q$ of the Fourier transform by $W^{-1}$.) From this representation it is not hard to see that $S \geq -1 + c$ for some constant $c > 0$. Moreover, for small $q$ we may expand $\wh S_W(q)$ to get $\wh S_W(q) \approx 1 - q \cdot D q$, where we defined the covariance matrix\footnote{To avoid confusion, we remark that this $D$ differs from the $D$ used in the introduction by a factor of order $W^{-2}$.} $D_W \equiv D = (D_{ij})$ of $S$ through
\begin{equation} \label{definition of D 2}
D_{ij} \;\deq\; \frac{1}{2}\sum_{x \in \bb T} \frac{x_i x_j}{W^2} S_{x0}\,.
\end{equation}
We deduce that $S$ has a simple eigenvalue at $1$, with associated eigenvector $(1,1, \dots, 1)$, and all remaining eigenvalues lie in the interval $[-1 + c, 1 - c (W/L)^2]$ for some small constant $c > 0$. Therefore the resolvent on the right-hand side of \eqref{F_eta_approx} is near-singular (hence yielding a large contribution) for $\ee^{\ii (A_1 - A_2)} \approx 1$. This implies that the leading behaviour of \eqref{F_eta_approx} is governed by small values of $q$ in Fourier space.

We now outline the computation of \eqref{F_eta_approx} in more detail. 
Let us first focus on the regime $\omega \gg \eta$, i.e.\ the regime from Theorem \ref{thm: Theta 1}. Thus, the function $\psi_i^\eta$ may be approximated by $2 \pi$ times a delta function, so that the convolutions may be dropped. What therefore remains is the calculation of the trace.
We write
$$ 
\alpha \;\deq\; \ee^{\ii (A_1 - A_2)} \;\approx\; 1 -\ii \omega (1 - E^2)^{-1/2} \;=\; 1-\ii \frac{2\omega}{\pi \nu}\,, \qquad \nu \;\equiv\; \nu(E)\,,
$$ 
in the regime $\omega\ll 1$. We use the Fourier representation of $S$ and only consider the contribution of small values of $q$.  After some elementary computations we get, for $d \leq 3$,
\begin{multline*}
\tr \frac{S}{(1 - \alpha S)^2} \;\approx\; \frac{L^d}{W^d} \int_{\R^d} \dd q \, \frac{\wh S_W(q)}{(1 - \alpha \wh S_W(q))^2} \;\approx\; \frac{L^d}{W^d}  \int_{\R^d} \frac{\dd q}{(1 - \alpha + q \cdot D q)^2}
\\
\approx\; \frac{L^d}{W^d}  \frac{1}{\sqrt{\det D}}  \pbb{\frac{2\omega}{\pi\nu}}^{d/2 - 2} \int_{\R^d} \frac{\dd q}{(\ii + q^2)^2}\,.
\end{multline*}
 A similar calculation may be performed for $d = 4$, which results in a logarithmic behaviour in $\omega$.
This yields the right-hand sides of \eqref{Theta_23} and \eqref{Theta_4}. 
For $d \leq 4$  the main contribution arises from the regime $q\approx 0$ and is therefore universal. 
If $d \geq 5$ the leading contribution to \eqref{skeletons_approx} arises from all values of $q$.
 While \eqref{skeletons_approx} may still be computed for $d \geq 5$, it loses its universal character and depends on the whole function $\wh S_W(q)$.

In the regime $\omega = 0$, i.e.\ the regime from Theorem \ref{thm: Theta 2}, we introduce $e \deq \psi_1 * \phi_2$ (recall \eqref{phipsi}) and write (for simplicity setting $E_1 = E_2 = 0$ and $d \leq 3$)
\begin{align*}
\tr \frac{\ee^{\ii (A_1 - A_2)} S}{\pb{1 - \ee^{\ii (A_1 - A_2)} S}^2} * \psi_1^{\eta}(E_1) * \psi_2^{\eta}(E_2)
&\;\approx\;
\int_{\R} \dd v \, e^\eta(v) \tr \frac{S}{\pb{1 - (1 - \ii v) S}^2}
\\
&\;\approx\; \frac{L^d}{W^d} \int_\R \dd v \, e^\eta(v) \int_{\R^d} \dd q \frac{1}{\pb{\ii v + q \cdot D q}^2}
\\
&\;=\; \frac{C L^d}{W^d} \int_{\R^d} \dd q  \, \int_0^\infty \dd t \, \ee^{-t q \cdot D q} \, t \, \ol {\wh e(\eta t)}
\\
&\;=\; \frac{C L^d}{W^d \sqrt{\det D}} \int_0^\infty \dd t \, t^{1 - d/2} \, \ol {\wh e(\eta t)}\,,
\end{align*}
where in the third step we used an elementary identity of Fourier transforms. 
 From this expression it will be easy to conclude  \eqref{Theta3D0}, 
and an analogous calculation for $d = 4$ yields \eqref{Theta4D0}.

\section{Proof of Theorems \ref{thm: main result}--\ref{thm: Theta 1}  for $\beta = 2$}
 \label{sec: main argument}

In this section we prove Theorems \ref{thm: main result}--\ref{thm: Theta 1} by computing the limiting behaviour of $\wt F^\eta(E_1, E_2)$.
For simplicity, throughout this section we assume that we are in the complex Hermitian case, $\beta = 2$. The real symmetric case, $\beta = 1$, can be handled by a simple extension of the arguments of this section, and is presented in Section \ref{sec:sym}. 

Due to the independence of the matrix entries (up to the Hermitian symmetry), the expectation of a product of matrix entries in \eqref{truncated series} can be computed simply by counting how many times a matrix entry (or its conjugate) appears. We therefore group these factors according to the equivalence relation $H_{xy} \sim H_{uv}$ if $\{x, y\}=\{ u, v\}$ as (unordered) sets. Since $\E H_{xy}=0$, every block of the associated partition must contain at least two elements; otherwise the corresponding term is zero. If $H$ were Gaussian, then by Wick's theorem only partitions with blocks of size exactly two (i.e.\ \emph{pairings}) would contribute. Since $H$ is not Gaussian, we have to do deal with blocks of arbitrary size; nevertheless, the pairings yield the main contribution.

In order to streamline the presentation and focus on the main ideas of the proof, in the current paper we do not deal with certain errors resulting from partitions that contain a block of size greater than two (i.e.\ that are not pairings), and from some exceptional coincidences among summation indices. 
Ignoring these issues results in three simplifications, denoted by {\bf (S1)}--{\bf (S3)} below, to the argument. Throughout the proof we use the letter $\cal E$ to denote any error term arising from these simplifications. 
In the companion paper \cite{EK4}, we show that 
the error terms $\cal E$ are indeed negligible; see Proposition \ref{prop:calE} below. 
The proof of Proposition \ref{prop:calE} is presented in a separate paper, as it requires a different argument from the one in the current paper.

In order clarify our main argument, it is actually helpful to generalize the assumptions on the matrix of variances $S$. This more general setup is also used in the generalized band matrix model 
 analysed in \cite{EK4} and outlined in Section \ref{sec:outlook}. 
 Instead of \eqref{step S}, we set
\begin{equation} \label{general S}
S_{xy} \;\deq\; \frac{1}{M - 1} f \pbb{\frac{[x - y]_L}{W}}\,, \qquad M \;\deq\; \sum_{x \in \bb T} f \pbb{\frac{x}{W}}\,,
\end{equation}
where $[x]_L$ denotes the representative of $x \in \Z^d$ in $\bb T$, and 
$f \col \R^d \to \R$ is an even, bounded, nonnegative, piecewise\footnote{We say that $f$ is \emph{piecewise $C^1$} if there exists a finite collection of disjoint open sets $U_1, \dots, U_n$ 
 with piecewise $C^1$ boundaries,  whose closures cover $\R^d$, such that $f$ is $C^1$ on each $U_i$.} $C^1$ function, such that $f$ and $\abs{\nabla f}$ are integrable. We also assume that
\begin{equation} \label{moment of f}
\int \dd x \, f(x) \, \abs{x}^{4 + c} \;<\; \infty
\end{equation}
for some $c > 0$.

Note that $M$ and $W$ satisfy \eqref{link between M and W}. 
We introduce the covariance matrices of $S$ (see also \eqref{definition of D 2}) and $f$, defined through
\begin{equation} \label{definition of D}
D_{ij} \;\deq\; \frac{1}{2}\sum_{x \in \bb T} \frac{x_i x_j}{W^2} S_{x0} \,, \qquad (D_0)_{ij} \;\deq\; \frac{1}{2} \int_{\R^d} x_i x_j f(x) \, \dd x\,.
\end{equation}
It is easy to see that $D = D_0 + O(W^{-1})$.
We always assume that
\begin{equation} \label{D_bounded}
c \;\leq\; D_0 \;\leq\; C 
\end{equation}
in the sense of quadratic forms, for some positive constants $c$ and $C$. Note that, since \eqref{D_bounded} holds for $D_0$, it also holds for $D$ for large enough $W$.
In the case \eqref{step S} we have the explicit diagonal form 
\begin{equation}\label{Dconst}
D_0 \;=\; \frac{1}{2 (d+2)} \, I_{d}  \,.
\end{equation}
In addition, for $d = 2$ we introduce the quantities
\begin{equation} \label{def_Q}
Q \;\deq\; \frac{1}{32} \sum_{x \in \bb T} S_{x0} \, \absbb{ D^{-1/2} \, \frac{x}{W}}^4\,, \qquad
Q_0 \;\deq\; \frac{1}{32} \int_{\R^2} \absb{D_0^{-1/2} x}^4 \, f(x) \, \dd x\,,
\end{equation}
which also depend on the fourth moments of $S$ and $f$ respectively. 
 (Here $\abs{\cdot}$ denotes the Euclidean norm on $\R^2$.)  As above, it is easy to see that $Q = Q_0 + O(W^{-1})$.
In the case \eqref{step S} we have the explicit form 
\begin{equation} \label{Qconst}
Q_0 \;=\; \frac{2}{3}\,.
\end{equation}

The main result of this section is summarized in the following Proposition \ref{prop: main}, which establishes the leading asymptotics of $\wt F^\eta(E_1,E_2)$, defined in \eqref{truncated series}, 
for small $\omega = E_2 - E_1$. 
Once Proposition \ref{prop: main} is proved, our main results, Theorems \ref{thm: main result}--\ref{thm: Theta 1} will follow easily (see Section \ref{sec:conclusion}). Recall the definition of $R_2$ from \eqref{def_R2}. 
\begin{proposition} \label{prop: main}
Suppose that the assumptions of the first paragraph of Theorem \ref{thm: main result} as well as the assumptions on $H$ made in \eqref{general S}--\eqref{D_bounded} 
hold. Then there is a constant $c_0 > 0$ such that, for any $E_1,E_2$ satisfying \eqref{D leq kappa} for small enough $c_* > 0$, we have
\begin{equation*}
\wt F^\eta(E_1, E_2) \;=\; \cal V_{\rm{main}} + \frac{N}{M} O \pb{M^{-c_0} R_2(\omega + \eta)} \,,
\end{equation*}
where the leading contribution $\cal V_{\rm{main}} \equiv \cal (V_{\rm{main}})^\eta_{\phi_1, \phi_2}(E_1,E_2)$ satisfies the following estimates.
\begin{enumerate}
\item
Suppose that \eqref{eta_Delta_2} holds. Then for $d = 1, 2, 3$  we have
\begin{equation} \label{V3C2}
\cal V_{\rm{main}} \;=\;
\frac{ (2/\pi)^{d/2}}{\nu(E)^2 \sqrt{\det D}}\pbb{\frac{L}{2 \pi W}}^d \pbb{\frac{\omega}{\nu(E)}}^{d/2 - 2} \pB{K_d + O \pb{\omega^{1/2} + M^{-\tau/2}}}
\end{equation}
where $K_d$ was defined in \eqref{def_K13}. Moreover, for $d = 4$ we have
\begin{equation} \label{V4C2}
\cal V_{\rm{main}} \;=\; \frac{8}{\nu(E)^2\sqrt{\det D}} \pbb{\frac{L}{2 \pi  W}}^d
 \pb{\abs{\log \omega}  + O (1)}\,.
\end{equation}
\item
Suppose that \eqref{eta_Delta_2} holds and that $d = 2$. If $\phi_1$ and $\phi_2$ satisfy \textbf{(C1)} then
\begin{equation} \label{V2C1}
\cal V_{\rm{main}} \;=\; \frac{8}{\pi\nu(E)^2 \sqrt{\det D}} \pbb{\frac{L}{2 \pi W}}^2  \pBB{\frac{\pi\eta \nu(E)  } {\omega^2 + 4 \eta^2} +
 \p{Q - 1} \abs{\log \omega} + O(1)}\,,
\end{equation}
and if $\phi_1$ and $\phi_2$ satisfy \textbf{(C2)} then
\begin{equation} \label{V2C2}
\cal V_{\rm{main}} \;=\; \frac{8}{\pi \nu(E)^2 \sqrt{\det D}} \pbb{\frac{L}{2 \pi W}}^2 \pb{  (Q - 1) \abs{\log \omega} + O(1)}\,.
\end{equation}
\item
Suppose that $\omega = 0$. Then the exponent $\mu$ from Proposition \ref{prop: expansion with trunction} may be chosen so that there exists an exponent $c_1 > 0$ such that for $d =1,2,3$ we have
\begin{equation} \label{V3C2D0}
\cal V_{\rm{main}} \;=\; \frac{2^{d/2}}{ \nu(E)^2\sqrt{\det D}} \pbb{\frac{L}{2 \pi W}}^d
  \pbb{\frac{\eta}{\nu(E)}}^{d/2 - 2} \pb{ V_d(\phi_1, \phi_2) + O(M^{-c_1})}
\end{equation}
and for $d = 4$ we have
\begin{equation} \label{V4C2D0}
\cal V_{\rm{main}} \;=\; \frac{4}{\nu(E)^2 \sqrt{\det D}} \pbb{\frac{L}{2 \pi W}}^4 \, \pb{V_4(\phi_1, \phi_2) \abs{\log \eta} + O(1)}\,.
\end{equation}
\end{enumerate}
\end{proposition}

The rest of this section is devoted to the proof of Proposition \ref{prop: main}.

\subsection{Introduction of graphs} \label{sec_41}

In order to express the nonbacktracking powers of $H$ in terms of the entries of $H$, it is convenient to index the two multiple summations arising from \eqref{def: nb} when plugged into \eqref{truncated series} using a graph.  We note that a similar graphical language was developed in \cite{EK1}, and many of basic definitions from Sections \ref{sec_41} and \ref{sec_42} (such as bridges, ladders, and skeletons) are similar to those from \cite{EK1}.  We introduce a directed graph $\cal C(n_1, n_2) \deq \cal C_1(n_1) \sqcup \cal C_2(n_2)$ defined as the disjoint union of a directed chain $\cal C_1(n_1)$ with $n_1$ edges and a directed chain $\cal C_2(n_2)$ with $n_2$ edges. Throughout the following, to simplify notation we often omit the arguments $n_1$ and $n_2$ from the graphs $\cal C$, $\cal C_1$, and $\cal C_2$. For an edge $e \in E(\cal C)$, we denote by $a(e)$ and $b(e)$ the initial and final vertices of $e$. Similarly, we denote by $a(\cal C_i)$ and $b(\cal C_i)$ the initial and final vertices of the chain $\cal C_i$. We call vertices of degree two \emph{black} and vertices of degree one \emph{white}. See Figure \ref{fig: cal C} for an illustration of $\cal C$ and for the convention of the orientation.

\begin{figure}[ht!]
\begin{center}
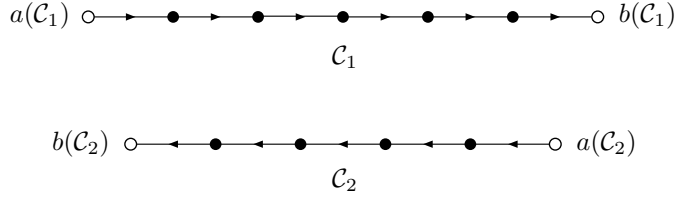
\end{center}
\caption{The graph $\cal C = \cal C_1 \sqcup \cal C_2$. Here we chose $n_1 = 6$ and $n_2 = 5$. We indicate the orientation of the chains $\cal C_1$ and $\cal C_2$ using arrows. In subsequent pictures, we systematically drop the arrows to avoid clutter, but we consistently use this orientation when drawing graphs. \label{fig: cal C}} 
\end{figure}

We assign a label $x_i \in \bb T$ to each vertex $i \in V(\cal C)$, and write $\f x = (x_i)_{i \in V(\cal C)}$.
For an edge $e \in E(\cal C)$ define the associated pairs of ordered and unordered labels
\begin{equation*}
x_{e} \;\deq\; (x_{a(e)}, x_{b(e)}) \,, \qquad [x_e] \;\deq\; \{x_{a(e)}, x_{b(e)}\}\,.
\end{equation*}
Using the graph $\cal C = \cal C(n_1, n_2)$ we may now write the covariance
\begin{equation} \label{exp_HH_sumx}
\avgb{\tr H^{(n_1)}\,; \tr H^{(n_2)}} \;=\; \E \qb{\pb{\tr H^{(n_1)}}\, \pb{\tr H^{(n_2)}}} - \E \pb{\tr H^{(n_1)}} \E \pb{\tr H^{(n_2)}} \;=\; \sum_{\f x \in {\bb T}^{V(\cal C)}} I(\f x) A(\f x)\,,
\end{equation}
where we introduced
\begin{equation} \label{definition of A}
A(\f x) \;\deq\;
\E \pBB{\prod_{e \in E(\cal C)} H_{x_e}} - \E \pBB{\prod_{e \in E(\cal C_1)} H_{x_e}} \E \pBB{\prod_{e \in E(\cal C_2)} H_{x_e}}
\end{equation}
and the indicator function 
\begin{equation} \label{def_I_ind}
I(\f x) \;\deq\; I_0(\f x) \prod_{\substack{i,j \in V(\cal C) \col \\ \dist(i,j) = 2}} \ind{x_i \neq x_j}\,, \qquad
I_0(\f x) \;\deq\; \ind{x_{a(\cal C_1)} = x_{b(\cal C_1)}} \ind{x_{a(\cal C_2)} = x_{b(\cal C_2)}}\,.
\end{equation}
The indicator function $I_0(\f x)$ implements the fact that the final and initial vertices of each chain have the same label, while $I(\f x)$ in addition implements the nonbacktracking condition.  When drawing $\cal C$ as in Figure \ref{fig: cal C}, we draw vertices of $\cal C$ with degree two using black dots, and vertices of $\cal C$ with degree one using white dots. The use of two different colours also reminds us that each black vertex $i$ gives rise to a nonbacktracking condition in $I(\f x)$, constraining the labels of the two neighbours of $i$ to be distinct.

In order to compute the expectation in \eqref{definition of A}, we decompose the label configurations $\f x$ according to partitions of $E(\cal C)$.
\begin{definition} \label{def:partitions}
We denote by $\fra P(U)$ for the set of partitions of a set $U$ and by $\fra M(U) \subset \fra P(U)$ the set of pairings (or matchings) of $U$.  (In the applications below the set $U$ will be either $E(\cal C)$ or $V(\cal C)$.) We call blocks of a pairing \emph{bridges}.  Moreover, for a label configuration $\f x \in \bb T^{V(\cal C)}$ we define the partition $P(\f x) \in \fra P(E(\cal C))$ as the partition of $E(\cal C)$ generated by the equivalence relation $e \sim e'$ if and only if $[x_e] = [x_{e'}]$. 
\end{definition}

Hence we may write
\begin{equation} \label{introduction of Gamma}
\sum_{\f x} I(\f x) A(\f x) \;=\; \sum_{\Pi \in \fra P(E(\cal C))} \sum_{\f x} \ind{P(\f x) = \Pi} I(\f x) A(\f x)\,.
\end{equation}
At this stage we introduce our first simplification.
\begin{enumerate}
\item[{\bf (S1)}]
We only keep the pairings  $\Pi \in \fra M(E(\cal C))$ in the summation \eqref{introduction of Gamma}.
\end{enumerate}
Using Simplification {\bf (S1)}, we write
\begin{equation} \label{sum over pairings after simplification}
\sum_{\f x} I(\f x) A(\f x) \;=\; \sum_{\Pi \in \fra M(E(\cal C))} \sum_{\f x} \ind{P(\f x) = \Pi} I(\f x) A(\f x) + \cal E\,.
\end{equation}
Here, as explained at the beginning of this section, we use the symbol $\cal E$ to denote an error term that arises from any simplification that we make. All such error terms are in fact negligible, as recorded in Proposition \ref{prop:calE} below, and proved in the companion paper \cite{EK4}. We use the symbol $\cal E$ without further comment throughout the following to denote such error terms arising from any of our simplifications {\bf (S1)}--{\bf (S3)}. 

Fix $\Pi \in \fra M(E(\cal C))$. In order to analyse the term resulting from the first term of \eqref{definition of A}, we write
\begin{equation} \label{sum over pairings after simplification 2}
\ind{P(\f x) = \Pi}  \, \E \pBB{\prod_{e \in E(\cal C)} H_{x_e}}
\;=\; \ind{P(\f x) = \Pi}  \, \pBB{\prod_{\{e, e'\} \in \Pi} \ind{x_e \neq x_{e'}} S_{x_e}}\,,
\end{equation}
where we used that $H_{x_e}$ and $H_{x_{e'}}$ are independent if $[x_e] \neq [x_{e'}]$  as well as $\E H_{xy}^2 = 0$ and $\E H_{xy} H_{yx} = S_{xy}$. Note that the indicator function $\ind{P(\f x) = \Pi}$ imposes precisely two things: first, if $e$ and $e'$ belong to the same bridge of $\Pi$ then $[x_e] = [x_{e'}]$ and, second, if $e$ and $e'$ belong to different bridges of $\Pi$ then $[x_e] \neq [x_{e'}]$. The second simplification that we make neglects the second restriction, hence eliminating interactions between the labels associated with different bridges.
\begin{enumerate}
\item[{\bf (S2)}]
After taking the expectation, we replace the indicator function $\ind{P(\f x) = \Pi}$ with the larger indicator function $\prod_{\{e,e'\} \in \Pi} \ind{[x_e] = [x_{e'}]}$.
\end{enumerate}
Thus we have
\begin{equation} \label{sum over pairings after simplification 3}
\ind{P(\f x) = \Pi}  \, \E \pBB{\prod_{e \in E(\cal C)} H_{x_e}}
\;=\;
\pBB{\prod_{\{e, e'\} \in \Pi} \ind{[x_e] = [x_{e'}]} \ind{x_e \neq x_{e'}} S_{x_e}} + \cal E\,.
\end{equation}
A similar analysis may be used for the term resulting from the second term of \eqref{definition of A} to get
\begin{multline} \label{sum over pairings after simplification 4}
\ind{P(\f x) = \Pi}  \, \E \pBB{\prod_{e \in E(\cal C_1)} H_{x_e}} \E \pBB{\prod_{e \in E(\cal C_2)} H_{x_e}}
\\
=\;\pBB{\prod_{\pi \in \Pi} \indb{\abs{\pi \cap E(\cal C_1)} \neq 1}} \pBB{\prod_{\{e, e'\} \in \Pi} \ind{[x_e] = [x_{e'}]} \ind{x_e \neq x_{e'}} S_{x_e}} + \cal E\,,
\end{multline}
where we used that if any bridge $\pi \in \Pi$ intersects both $E(\cal C_1)$ and $E(\cal C_2)$ then the left-hand side vanishes since $\E H_{xy} = 0$. 

Next, we note that
\begin{equation} \label{def_I_sigma}
\ind{[x_e] = [x_{e'}]} \ind{x_e \neq x_{e'}} \;=\; \ind{x_{a(e)} = x_{b(e')}} \ind{x_{a(e')} = x_{b(e)}} \;\eqd\; J_{\{e,e'\}}(\f x)\,.
\end{equation}
Plugging \eqref{sum over pairings after simplification 3} and \eqref{sum over pairings after simplification 4} back into \eqref{sum over pairings after simplification} therefore yields
\begin{equation} \label{expansion in terms of P S}
\avgb{\tr H^{(n_1)}\,; \tr H^{(n_2)}} \;=\; \sum_{\Pi \in \fra M_c(E(\cal C))} \sum_{\f x} I(\f x) \pBB{\prod_{\{e, e'\} \in \Pi} J_{\{e,e'\}}(\f x) \, S_{x_e}} + \cal E\,,
\end{equation}
where we introduced the \emph{subset of connected pairings} of $E(\cal C)$
\begin{equation} \label{conn_pair}
\fra M_c(E(\cal C)) \;\deq\; \hb{\Pi \in \fra M(E(\cal C)) \col \text{there is a $\pi \in \Pi$ such that } \pi \cap E(\cal C_1) \neq \emptyset \text{ and } \pi \cap E(\cal C_2) \neq \emptyset}\,.
\end{equation}
The formula \eqref{expansion in terms of P S} provides the desired expansion in terms of pairings.

A pairing $\Pi$ may be conveniently represented graphically by drawing a line (or \emph{bridge}) joining the edges $e$ and $e'$ whenever $\{e,e'\} \in \Pi$. See Figure \ref{fig: pairing} for an example.

\begin{figure}[ht!]
\begin{center}
\includegraphics{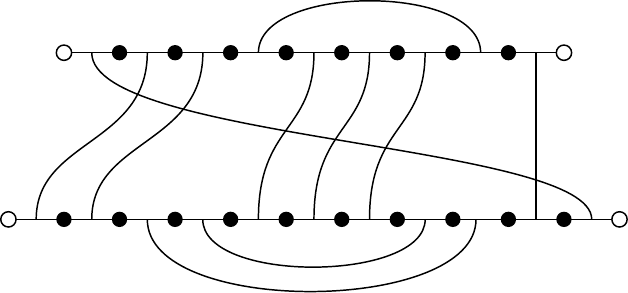}
\end{center}
\caption{A pairing of edges. \label{fig: pairing}} 
\end{figure}

The following notations will prove helpful. We introduce the set of all connected pairings,
\begin{equation*}
\fra M_c \;\deq\; \bigsqcup_{\substack{n_1, n_2 \geq 0 \col \\ n_1 + n_2 \text{ even}}} \fra M_c \pb{E(\cal C(n_1, n_2))}\,.
\end{equation*}
\begin{definition} \label{def:C_Gamma}
With each pairing $\Pi \in \fra M_c$ we associate its underlying graph $\cal C(\Pi)$, and regard $n_1$ and $n_2$ as functions on $\fra M_c$ in self-explanatory notation. We also frequently abbreviate $V(\Pi) \equiv V(\cal C(\Pi))$, and refer to $V(\Pi)$ as the vertices of $\Pi$.
\end{definition}

Next, we observe that the indicator function
\begin{equation} \label{indicator function behind Q}
\ind{x_{a(\cal C_1)} = x_{b(\cal C_1)}} \ind{x_{a(\cal C_2)} = x_{b(\cal C_2)}} \prod_{\pi \in \Pi} J_\pi(\f x)
\end{equation}
in \eqref{expansion in terms of P S} constrains some labels of $\f x$ to coincide.
We introduce a corresponding partition $Q(\Pi) \in \fra P (V(\Pi))$ of the vertices of $\Pi$, whereby $i$ and $j$ are in the same block of $Q(\Pi)$ if and only if $x_i$ and $x_j$ are constrained to be equal by \eqref{indicator function behind Q}. Equivalently, we define $Q(\Pi)$ as the finest partition of $V(\Pi)$ with the following properties.
\begin{enumerate}
\item
$a(e)$ and $b(e')$ belong to the same block of $Q(\Pi)$ whenever $\{e,e'\} \in \Pi$. (Note that, by symmetry, $a(e')$ and $b(e)$ also belong to the same block.)
\item
$a(\cal C_1)$ and $b(\cal C_1)$ belong to the same block of $Q(\Pi)$.
\item
$a(\cal C_2)$ and $b(\cal C_2)$ belong to the same block of $Q(\Pi)$.
\end{enumerate}
Graphically, the first condition means that the two vertices on either side of a bridge are constrained to have the same label. See Figure \ref{fig: Q} for an illustration of $Q(\Pi)$. We emphasize that we constantly have to deal with two different partitions. Taking the expectation originally introduced a  partition on the edges, which, after Simplification {\bf (S1)}, is in fact a pairing. This pairing, in turn, induces constraints on the labels that are assigned to vertices; more precisely, it forces the labels of certain vertices to coincide. Together with the coincidence of the first and last labels on $\cal C_1$ and $\cal C_2$, imposed by taking the trace, this defines a partition on the vertices. Depending on $\f x$ it may happen that more labels coincide than required by $Q(\Pi)$; the partition $Q(\Pi)$ encodes the minimal set of constraints. 
We therefore call $Q(\Pi)$ the \emph{minimal vertex partition induced by $\Pi$}.
Notice that, by construction, $Q(\Pi)$ does not depend on $\f x$.
\begin{figure}[ht!]
\begin{center}
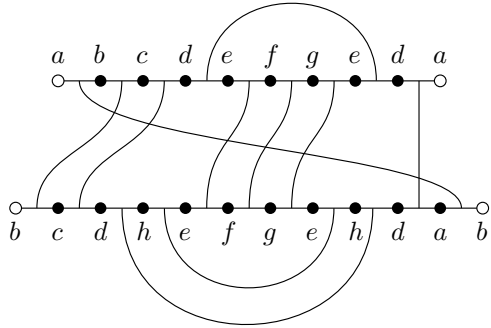
\end{center}
\caption{The pairing $\Pi$ from Figure \ref{fig: pairing}, where we in addition indicate the eight blocks of $Q(\Pi)$ by assigning a letter to each block. \label{fig: Q}} 
\end{figure}

Next, suppose that there is a block $q \in Q(\Pi)$ that contains two vertices $i,j \in q$ such that $\dist(i,j) = 2$. We conclude that the contribution of $\Pi$ to the right-hand side of \eqref{expansion in terms of P S} vanishes, since the indicator function $I(\f x)$ vanishes by the nonbacktracking condition $\ind{x_i \neq x_j}$. Hence we may restrict the summation over $\Pi$ in \eqref{expansion in terms of P S} to the subset of pairings
\begin{equation} \label{def_R}
\fra R \;\deq\; \hb{\Pi \in \fra M_c \col \text{if } \dist(i,j) = 2 \text{ then $i$ and $j$ belong to different blocks of $Q(\Pi)$}}\,.
\end{equation}
\begin{lemma} \label{lem:q_geq_2}
For any $\Pi \in \fra R$, all blocks of $Q(\Pi)$ have size at least two.
\end{lemma}
\begin{proof}
If $\{i\} \in Q(\Pi)$ then, by definition of $Q(\Pi)$, the degree of $i$ is two and both edges incident to $i$ belong to the same bridge. This implies that the two vertices adjacent to $i$ belong to the same block of $Q(\Pi)$, which is impossible by definition of $\fra R$.
\end{proof}

At this point we introduce our final simplification.
\begin{enumerate}
\item[{\bf (S3)}]
After restriction the summation over $\Pi$ to the set $\fra R$ in \eqref{expansion in terms of P S}, we neglect the indicator function $I(\f x)$.
\end{enumerate}
Note that the main purpose of $I(\f x)$ was to restrict the summation over pairings $\Pi$ to the set $\fra R$, which is still taken into account if one assumes {\bf (S3)}. The presence of $I(\f x)$ in \eqref{expansion in terms of P S} simply results in some additional error terms $\cal E$ that are ultimately negligible. Note that $I(\f x)$ also restricts the summation to labels satisfying $x_{a(C_i)} = x_{b(C_i)}$; this condition is still imposed in the definition of $\fra R$.

Hence we get
\begin{multline} \label{expansion in terms of P fully simplified}
\avgb{\tr H^{(n_1)}\,; \tr H^{(n_2)}}
\\
=\; \sum_{\Pi \in \fra R} \ind{n_1(\Pi) = n_1} \ind{n_2(\Pi) = n_2} \sum_{\f y \in \bb T^{Q(\Pi)}} \sum_{\f x \in \bb T^{V(\Pi)}} \pBB{\prod_{q \in Q(\Pi)} \prod_{i \in q} \ind{x_i = y_q}} \pBB{\prod_{\{e, e'\} \in \Pi} S_{x_e}} + \cal E\,,
\end{multline}
where we introduced a set of independent summation labels $\f y$, indexed by the blocks of $Q(\Pi)$.

\subsection{Skeletons} \label{sec_42}

The summation in \eqref{truncated series} is highly oscillatory,
which requires a careful resummation of graphs of different order. 
We perform a local resummation procedure of the so-called \emph{ladder} subdiagrams,
which are subdiagrams with a pairing structure that consists only of parallel bridges.
This is the second resummation procedure mentioned in Section \ref{sec:struc}.
Concretely, we regroup pairings $\Pi$ into families that have a similar structure, differing only in the number of parallel bridges per ladder  subdiagram. Their common structure is represented by the simplest element of the family, the \emph{skeleton}, whose ladders consist of a single bridge.

We now introduce these concepts precisely. The skeleton of a pairing $\Pi \in \fra M_c$ is generated from $\Pi$ by collapsing parallel bridges. By definition, the bridges $\{e_1, e_1'\}$ and $\{e_2, e_2'\}$ are \emph{parallel} if $b(e_1) = a(e_2)$ and $b(e_2') = a(e_1')$. With each $\Pi \in \fra M_c$ we associate a couple $\cal S(\Pi) = (\Sigma, \f b)$, where $\Sigma \in \fra M_c$ has no parallel bridges, and $\f b = (b_\sigma)_{\sigma \in \Sigma} \in \N^\Sigma$. The pairing $\Sigma$ is obtained from $\Pi$ by successively collapsing parallel bridges until no parallel bridges remain. The integer $b_\sigma$ denotes the number of parallel bridges of $\Pi$ that were collapsed into the bridge $\sigma$. Conversely, for any given couple $(\Sigma, \f b)$, where $\Sigma \in \fra M_c$ has no parallel bridges and $\f b \in \N^{\Sigma}$, we define $\Pi = \cal G(\Sigma, \f b)$ as the pairing obtained from $\Sigma$ by replacing, for each $\sigma \in \Sigma$, the bridge $\sigma$ with $b_\sigma$ parallel bridges. Thus we have a one-to-one correspondence between pairings $\Pi$ and couples $(\Sigma, \f b)$. The map $\cal S$ corresponds to the collapsing of parallel bridges of $\Pi$, and the map $\cal G$ to the ``expanding'' of bridges of $\Sigma$ according to the multiplicities $\f b$. Instead of burdening the reader with formal definitions of the operations $\cal S$ and $\cal G$, we refer to Figures \ref{fig: pairing} and \ref{fig: skeleton} for an illustration. When no confusion is possible, in order to streamline notation we shall omit $\cal S$ and $\cal G$ and identify $\Pi$ with $(\Sigma, \f b)$.
In particular, the minimal vertex partition $Q(\Pi)$ induced by $\Pi = \cal G(\Sigma, \f b)$ is
denoted by $Q(\Sigma, \f b)$, and is not to be confused with $Q(\Sigma)$, the minimal vertex partition on
the skeleton $\Sigma$.
\begin{figure}[ht!]
\begin{center}
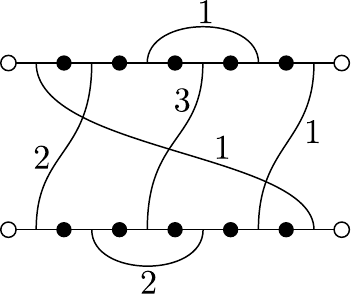
\end{center}
\caption{The skeleton $\Sigma$ of the pairing $\Pi$ from Figure \ref{fig: pairing}. Next to each skeleton bridge $\sigma \in \Sigma$ we indicate the multiplicity $b_\sigma$ describing how many bridges of $\Pi$ were collapsed into $\sigma$. \label{fig: skeleton}} 
\end{figure}

\begin{definition} \label{def:ladders}
Fix $\Sigma \in \fra M_c$ and  $\f b\in  \N^\Sigma$. As above, abbreviate $\Pi \deq \cal G(\Sigma, \f b)$.
\begin{enumerate}
\item
For $\sigma \in \Sigma$ we introduce the \emph{ladder encoded by $\sigma$}, denoted by $L_\sigma(\Sigma, \f b) \subset \Pi$ and defined as the set of bridges $\pi \in \Pi$ that are collapsed into the skeleton bridge $\sigma$ by the operation $\cal S$. Note that $L_\sigma(\Sigma,\f b)$ consists of $\abs{L_\sigma(\Sigma,\f b)} = b_\sigma$ parallel bridges.

\item
We say that a vertex $i \in V(\Pi)$ \emph{touches} the bridge $\{e,e'\} \in \Pi$ if $i$ is incident to $e$ or $e'$. We call a vertex $i$ a \emph{ladder vertex} of $L_\sigma(\Sigma, \f b)$ if it touches two bridges of $L_\sigma(\Sigma, \f b)$.
Note that a ladder consisting of $b$ parallel bridges gives rise to $2(b-1)$ ladder vertices.

\item
We say that $i \in V(\Pi)$ is a \emph{ladder vertex} of $\Pi$ if it is a ladder vertex of $L_\sigma(\Sigma, \f b)$ for some $\sigma \in \Sigma$. We decompose the vertices $V(\Pi) = V_s(\Pi) \sqcup V_l(\Pi)$, where $V_l(\Pi)$ denotes the set of ladder vertices of $\Pi$.
\end{enumerate}
\end{definition}
See Figure \ref{fig: ladder vertices} for an illustration.
\begin{figure}[ht!]
\begin{center}
\includegraphics{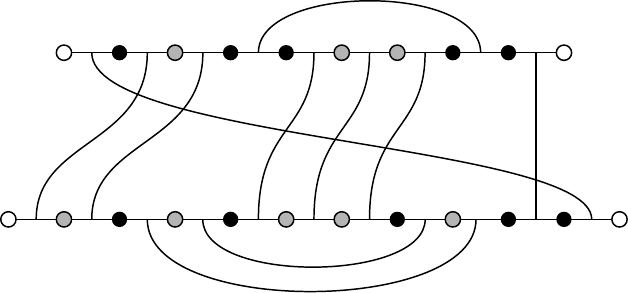}
\end{center}
\caption{The ladder vertices $V_l(\Pi)$, drawn in grey, of the pairing $\Pi$ from Figure \ref{fig: pairing}. The vertices $V_s(\Pi)$ are drawn in black or white. In this example, there are $\abs{\Sigma} = 6$ ladders. \label{fig: ladder vertices}}
\end{figure}
Due to the nonbacktracking condition and the requirement that parallel bridges are collapsed, not every pairing can be a skeleton, and not every family of multiplicities is admissible; however, the few exceptions are easy to describe. The following lemma characterizes the explicit set $\fra S$ of allowed skeletons $\Sigma$ and the set of allowed multiplicities, $B(\Sigma)$, which may arise from some graph $\Pi \in \fra R \subset \fra M_c$.

\begin{lemma} \label{lem: skeletons}
For any $\Pi \in \fra R$ with $(\Sigma, \f b) \deq \cal S(\Pi)$ we have $\Sigma \in \fra S$, where
\begin{equation*}
\fra S \;\deq\; \hb{\Sigma \in \fra M_c \col \Sigma \text{ has no parallel bridges and no block of $Q(\Sigma)$ has size one}}\,.
\end{equation*}
Moreover, defining
\begin{equation}\label{def:B}
B(\Sigma) \;\deq\; \hb{\f b \in \N^\Sigma \col \cal G(\Sigma, \f b) \in \fra R}\,,
\end{equation}
for any $\Sigma \in \fra S$,
we have that $\N^{\Sigma} \setminus B(\Sigma)$ is finite.
\end{lemma}

Roughly, this lemma states two things. First, if a skeleton bridge $\sigma \in \Sigma$ touches two adjacent vertices of  $\Sigma$ that belong to the same block of $Q(\Sigma)$, then we have $b_\sigma \neq 2$. Second, if $Q(\Sigma)$ yields the label structure $aba$ for three consecutive vertices of $\Sigma$, then $b_\sigma + b_{\sigma'} \geq 3$ where $\sigma$ and $\sigma'$ are the two bridges touching the innermost of these three vertices (in such a situation $\sigma=\sigma'$ is impossible by nonbacktracking condition implemented by $\fra R$).
See Figure \ref{fig: D8} below for an illustration of this latter restriction. Both of these restrictions are consequences of the nonbacktracking condition implemented in the definition of $\fra R$.

For example, the skeleton $D_4$, defined in Figure \ref{fig: dumbbell} below, may arise
as a skeleton of some $\Pi$, so that $D_4\in \fra S$.
Using $b_1$, $b_2$, and $b_3$ to denote the multiplicities of the top, bottom, and middle bridges respectively, we have $B(D_4) = \hb{\f b=(b_1,b_2,b_3) \col b_1, b_2, b_3 \geq 1 \,,\, b_3 \neq 2}$.
Indeed, it is easy to check that the condition on the right-hand side of \eqref{def_R} is satisfied if and only if $b_2 \neq 2$.

\begin{proof}[Proof of Lemma \ref{lem: skeletons}]
Let $\Pi \in \fra R$ and $(\Sigma, \f b) \deq \cal S(\Pi)$. Clearly, $\Sigma$ has no parallel bridges. Moreover, if $Q(\Sigma)$ has a block of size one then $\Sigma$ must have a bridge that connects 
two adjacent edges. Hence $\Pi$ also has a bridge that connects two adjacent edges. By definition of $\fra R$, this is impossible. This proves the first claim.

In order to prove the second claim, we simply observe that if $\Sigma \in \fra S$ and $b_\sigma \geq 2$ for all $\sigma \in \Sigma$, then $\cal G(\Sigma, \f b) \in \fra R$. This  follows easily from the definition of $\fra R$ and the fact that the two vertices located between two parallel bridges of $\Pi$ always form a block of size two in $Q(\Pi)$.
\end{proof}

Lemma \ref{lem: skeletons} proves that there is a one-to-one correspondence, given by the maps $\cal S$ and $\cal G$, between pairings $\Pi \in \fra R$ and couples $(\Sigma, \f b)$ with $\Sigma \in \fra S$ and $\f b \in B(\Sigma)$.  Throughout the following, we often make use of this correspondence and tacitly identify $\Pi$ with $(\Sigma, \f b)$.  We now use skeletons to rewrite $\wt F^\eta (E_1, E_2)$: from \eqref{expansion in terms of P fully simplified} we get
\begin{align}
\wt F^\eta (E_1, E_2) &\;=\; \sum_{\Pi \in \fra R} \ind{2 \abs{\Pi} \leq M^\mu}
2 \re \pb{\wt \gamma_{n_1(\Pi)}(E_1,\phi_1)} \, 2 \re \pb{\wt \gamma_{n_2(\Pi)}(E_2, \phi_2)}
\notag \\
&\qquad \times
\sum_{\f y \in \bb T^{Q(\Pi)}} \sum_{\f x \in \bb T^{V(\Pi)}} \pBB{\prod_{q \in Q(\Pi)} \prod_{i \in q} \ind{x_i = y_q}} \pBB{\prod_{\{e, e'\} \in \Pi} S_{x_e}} + \cal E\,,
\notag \\
&\;=\; \sum_{\Sigma \in \fra S}  \sum_{\f b \in B(\Sigma)} \indBB{2 \sum_{\sigma \in \Sigma} b_\sigma \leq M^\mu } 2 \re \pb{\wt \gamma_{n_1(\Sigma, \f b)}(E_1,\phi_1)} \, 2 \re \pb{\wt \gamma_{n_2(\Sigma, \f b)}(E_2, \phi_2)}
\notag \\ \label{expansion in terms of P S 3}
& \qquad \times \sum_{\f y \in \bb T^{Q(\Sigma, \f b)}} \sum_{\f x \in \bb T^{V(\Sigma, \f b)}} \pBB{\prod_{q \in Q(\Sigma, \f b)} \prod_{i \in q} \ind{x_i = y_q}} \pBB{\prod_{\{e, e'\} \in \cal G(\Sigma, \f b)} S_{x_e}} + \cal E\,.
\end{align}

Next, we observe that we have a splitting
\begin{equation*}
Q(\Pi) \;=\; Q(\Pi) \vert_{V_s(\Pi)} \sqcup Q(\Pi) \vert_{V_l(\Pi)}\,,
\end{equation*}
so that the indicator function in \eqref{expansion in terms of P S 3} factors into an indicator function involving only labels $y_q$ and $x_i$ with $q \in Q(\Pi) \vert_{V_s(\Pi)}$ and $i \in V_s(\Pi)$, and another indicator function involving only labels $y_q$ and $x_i$ with $q \in Q(\Pi) \vert_{V_l(\Pi)}$ and $i \in V_l(\Pi)$. Summing over the latter (``ladder'') labels yields
\begin{equation} \label{F in terms of skeletons}
\wt F^\eta (E_1, E_2) \;=\; \sum_{\Sigma \in \fra S} \cal V(\Sigma) + \cal E\,,
\end{equation}
where we defined the \emph{value} of the skeleton $\Sigma \in \fra S$ as
\begin{multline} \label{def of V Sigma}
\cal V(\Sigma) \;\deq\; \sum_{\f b \in B(\Sigma)} \indBB{2 \sum_{\sigma \in \Sigma} b_\sigma \leq M^\mu } 2 \re \pb{\wt \gamma_{n_1(\Sigma, \f b)}(E_1,\phi_1)} \, 2 \re \pb{\wt \gamma_{n_2(\Sigma, \f b)}(E_2, \phi_2)}
\\
\times \sum_{\f y \in \bb T^{Q(\Sigma)}} \sum_{\f x \in \bb T^{V(\Sigma)}} \pBB{\prod_{q \in Q(\Sigma)} \prod_{i \in q} \ind{x_i = y_q}} \pBB{\prod_{\{e, e'\} \in \Sigma} (S^{b_{\{e,e'\}}})_{x_e}}\,.
\end{multline}
 Here we recall Definition \ref{def:C_Gamma} for the meaning of the vertex set $V(\Sigma)$.
The entry $(S^{b_{\{e,e'\}}})_{x_e}$ arises from summing out the $b_{\{e,e'\}} - 1$ independent labels associated with the ladder vertices of $L_{\{e,e'\}}(\Sigma, \f b)$, according to
\begin{equation}\label{intsum}
\sum_{x_1, \dots, x_{b - 1}} S_{x_0 x_1} S_{x_1 x_2} \cdots S_{x_{b - 1} x_b} \;=\; (S^b)_{x_0 x_b}\,.
\end{equation}
The labels $\f x \in \bb T^{V(\Sigma)}$ in \eqref{def of V Sigma} are not free; we use them for notational convenience. They are a function $\f x = \f x(\f y)$ of the independent labels $\f y \in \bb T^{Q(\Sigma)}$. The function $\f x(\f y)$ is defined by the indicator function in the second parentheses on the second line of \eqref{def of V Sigma},  i.e.\ $x_i(\f y) \deq y_q$ where $q \ni i$. 

In summary, we have proved that $\wt F^\eta (E_1, E_2)$ can be written as a sum of
contributions of skeleton graphs (up to errors $\cal E$ that will prove to be negligible). The value of each skeleton is computed by assigning a positive power $b_e$ of $S$ to each bridge of $\Sigma$, and summing up all powers $b_e$ and all labels that are
compatible with $\Sigma$ (in the sense that the vertices touching a bridge, on the same side of the bridge, must have identical labels).

\subsection{The leading term} \label{sec:dumbbell_C1}
We now compute the leading contribution to \eqref{F in terms of skeletons}. As it turns out, it arises from a family of eight skeleton pairings, which we call \emph{dumbbell} skeletons. They are defined in Figure \ref{fig: dumbbell}. We denote by $D_i$ the $i$-th dumbbell skeleton, where $i = 1, \dots, 8$. 
 At this point in the argument, it is not apparent why precisely these
eight skeletons yield the leading contribution. In fact, our analysis will reveal the graph-theoretic properties that single them out as the leading skeletons; see Section \ref{sec:small_skeletons} below for the details.

\begin{figure}[ht!]
\begin{center}
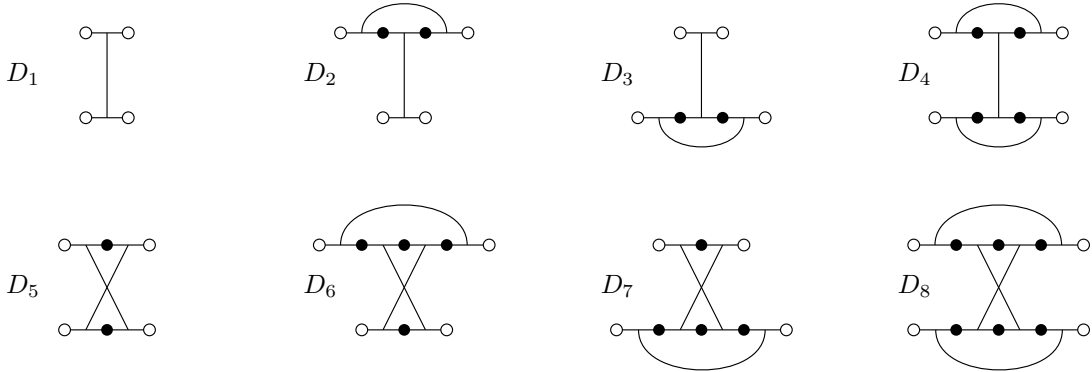
\end{center}
\caption{The eight dumbbell skeletons $D_1, \dots, D_8$. \label{fig: dumbbell}} 
\end{figure}

We now \emph{define} $\cal V_{\rm{main}}$ as the contribution of the dumbbell skeletons:
\begin{equation} \label{def_V_main}
\cal V_{\rm{main}} \;\deq\; \sum_{i = 1}^8 \cal V(D_i)\,.
\end{equation}

\begin{proposition}[Dumbbell skeletons] \label{prop: leading term}
Under the assumptions of Proposition \ref{prop: main}, the contribution of the dumbbell skeletons defined in \eqref{def_V_main} satisfies (i), and (ii), and (iii) of Proposition \ref{prop: main}.
\end{proposition}
\begin{proof}
See \cite[Propositions 3.4 and 3.7]{EK4}. 
\end{proof}

While the proof of Proposition \ref{prop: leading term} is given in the companion paper \cite{EK4}, here we explain how to obtain the (approximate) expression \eqref{skeletons_approx} from the definition \eqref{def_V_main}. The main work, performed in \cite[Sections 3.3 and 3.4]{EK4}, is the asymptotic analysis of 
the right-hand side of \eqref{skeletons_approx}, which was outlined in Section \ref{sec: heuristic leading term}.

We first focus on the most important skeleton, $D_8$. See Figure \ref{fig: D8} for our choice of labelling the vertex labels and the multiplicities of the bridges of $D_8$.
\begin{figure}[ht!]
\begin{center}
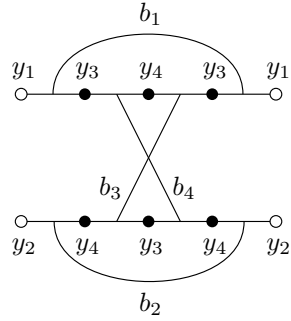
\end{center}
\caption{The skeleton $D_8$. We indicate the independent labels $y_1, \dots, y_4$ next to their associated vertices, and the multiplicities $b_1, \dots, b_4$ next to their associated bridges of $D_8$. \label{fig: D8}} 
\end{figure}
In particular, $Q(D_8)$ consists of four blocks, which are assigned the independent summation vertices $x_1, \dots, x_4$. From \eqref{def of V Sigma} we get
\begin{multline*}
\cal V(D_8) \;=\; \sum_{b_1, b_2, b_3, b_4 \geq 1} \ind{b_3 + b_4 \geq 3} \indb{b_1 + b_2 + b_3 + b_4 \leq M^\mu / 2}
\\
\times 2 \re \pb{\wt \gamma_{2 b_1 + b_3 + b_4}(E_1,\phi_1)} \, 2 \re \pb{\wt \gamma_{2 b_2 + b_3 + b_4}(E_2, \phi_2)}
\sum_{y_1, y_2, y_3, y_4 \in \bb T} (S^{b_1})_{y_1 y_3} (S^{b_2})_{y_2 y_4} (S^{b_3})_{y_3 y_4} (S^{b_4})_{y_3 y_4}\,.
\end{multline*}
Here we used that $B(D_8) = \{(b_1,b_2,b_3,b_4) \col b_1, b_2, b_3, b_4 \geq 1 \,,\, b_3 + b_4 \geq 3\}$, as may be easily checked from the definition of $\fra R$. Similarly, we may compute $\cal V(D_i)$ for $i = 1, \dots, 7$; it is not hard to see that all of them arise from the expression for $\cal V(D_8)$ by setting $b_1$, $b_2$, or $b_4$ to be zero; setting a multiplicity $b_i$ to be zero amounts to removing the corresponding bridge from the skeleton. Since the skeleton has to be connected, $b_3$ and $b_4$ cannot both be zero and
we choose to assign $b_3$ to the bridge with nonzero multiplicity.
The eight combinations generated by $b_1 = 0$ or $b_1 \neq 0$, $b_2 = 0$ or $b_2 \neq 0$, $b_4 = 0$ or $b_4 \neq 0$ correspond precisely to the eight graphs $D_1, \dots, D_8$  (the case $b_3 \geq 1$, $b_4=0$ corresponds to the first four graphs, $D_1, \dots, D_4$,
while $b_3, b_4 \geq 1$ corresponds to $D_5, \ldots, D_8$).
Moreover, recalling \eqref{general S},  we can perform the sum over $y_1, \dots, y_4$:
\begin{equation*}
\sum_{y_1, y_2, y_3, y_4 \in \bb T} (S^{b_1})_{y_1 y_3} (S^{b_2})_{y_2 y_4} (S^{b_3})_{y_3 y_4} (S^{b_4})_{y_3 y_4} \;=\; \cal I^{b_1 + b_2} \tr S^{b_3 + b_4}\,,
\end{equation*}
where we defined
\begin{equation} \label{def_I}
\cal I \;\equiv\; \cal I_M \;\deq\; \frac{M}{M - 1}\,.
\end{equation}
The choice of the symbol $\cal I$ suggests that for most purposes $\cal I$ should be thought of as $1$.
Putting everything together, we find
\begin{multline}\label{VDdef}
\cal V_{\rm{main}} \;=\; \sum_{b_1, b_2 = 0}^\infty \sum_{(b_3, b_4) \in \cal A} \indb{b_1 + b_2 + b_3 + b_4 \leq M^\mu / 2}
\\
\times 2 \re \pb{\wt \gamma_{2 b_1 + b_3 + b_4}(E_1,\phi_1)} \, 2 \re \pb{\wt \gamma_{2 b_2 + b_3 + b_4}(E_2, \phi_2)} \, \cal I^{b_1 + b_2} \tr S^{b_3 + b_4}\,, 
\end{multline}
where
\begin{equation} \label{def_cal_A}
\cal A \;\deq\; \pb{\h{1,2,\dots} \times \h{0,1,\dots}} \setminus \hb{(2,0), (1,1)}\,.
\end{equation}
Note that here we exclude the two cases where $b_3 + b_4 = 2$, since in those cases it may be easily checked that $Q(\cal G(\Sigma, \f b))$ violates the defining condition of $\fra R$. In all other cases, this condition is satisfied.

Next, we use \eqref{bound on gamma} to decouple the upper bound in the summations over $b_1$, $b_2$, $b_3$, and $b_4$. Using \eqref{bound on gamma} we easily find
\begin{multline*}
\cal V_{\rm{main}} \;=\; \sum_{b_1, b_2 = 0}^{\infty} \sum_{(b_3, b_4) \in \cal A} \, 2 \re \pb{\wt \gamma_{2 b_1 + b_3 + b_4}(E_1,\phi_1)} \, 2 \re \pb{\wt \gamma_{2 b_2 + b_3 + b_4}(E_2, \phi_2)} \, \cal I^{b_1 + b_2} \tr S^{b_3 + b_4}
\\
+ O_q(N M^{-q})\,.
\end{multline*}
Similarly, using \eqref{gamma - g} to replace $\wt \gamma$ by $\gamma$, we get
\begin{multline} \label{common expression for V(D)}
\cal V_{\rm{main}} \;=\; \sum_{b_1, b_2 = 0}^{\infty} \sum_{(b_3, b_4) \in \cal A} \, 2 \re \pb{\gamma_{2 b_1 + b_3 + b_4} * \psi_1^\eta}(E_1) \, 2 \re \pb{\gamma_{2 b_2 + b_3 + b_4} * \psi^\eta_2}(E_2) \, \cal I^{b_1 + b_2} \tr S^{b_3 + b_4}
\\
+ O_q(N M^{-q})\,.
\end{multline}
This is the precise version of \eqref{skeletons_approx}. 
For the asymptotic analysis of the right-hand side of \eqref{common expression for V(D)}, 
see  \cite[Sections 3.3 and 3.4]{EK4}.

\subsection{The error terms: large skeletons} \label{sec:large_Sigma}
We now focus on the essence of the proof of Proposition \ref{prop: main}: the estimate of the non-dumbbell 
skeletons. We have to estimate the contribution to the right-hand side of \eqref{F in terms of skeletons} of all skeletons $\Sigma$ in the set
\begin{equation}\label{frastar}
\fra S^* \;\deq\; \fra S \setminus \{D_1, \dots, D_8\}\,.
\end{equation}
It turns out that when estimating $\cal V(\Sigma)$ we are faced with two independent difficulties. First, strong oscillations  in the
$\f b$-summations in the definition of  
$\cal V(\Sigma)$  \eqref{def of V Sigma}
give rise to cancellations which have to be exploited carefully. Second, due to the combinatorial complexity of the skeletons, the size of $\fra S^*$ grows exponentially with $M$, which means that we have to deal with combinatorial estimates. It turns out that these two difficulties may be effectively decoupled: if $\abs{\Sigma}$ is small then only the first difficulty matters, and if $\abs{\Sigma}$ is large then only the second one matters. The sets of small and large skeletons are defined as
\begin{equation}\label{smalllarge}
\fra S^\leq \;\equiv\; \fra S^\leq_K \;\deq\; \hb{\Sigma \in \fra S^* \col \abs{\Sigma} \leq K}\,, \qquad
\fra S^>  \;\equiv\; \fra S^>_K \;\deq\; \hb{\Sigma \in \fra S^* \col \abs{\Sigma} > K}\,,
\end{equation}
where $K \in \N$ is a cutoff, independent of $N$, to be fixed later.

In this subsection, we deal with large $\abs{\Sigma}$, i.e.\ we estimate $\sum_{\Sigma \in \fra S^>} \cal V(\Sigma)$.  The only input on $\wt \gamma_n(E_i, \phi_i)$ that the argument of this subsection requires is the estimate \eqref{bound on gamma}. In particular, in this subsection we deal with both cases \textbf{(C1)} and \textbf{(C2)} simultaneously.

\begin{proposition} \label{prop:large_Sigma} 
For large enough $K$, depending on $\mu$, we have
\begin{equation}
\sum_{\Sigma \in \fra S^>_K} \abs{\cal V(\Sigma)} \;\leq\; C_K N M^{-2}\,.
\end{equation}
\end{proposition}

Recall that, according to Proposition \ref{prop: main}, the value of the main terms (the dumbbell skeletons)
 is larger than $N M^{-1}$. 
The rest of this subsection is devoted to the proof of Proposition \ref{prop:large_Sigma}.
We begin by introducing the following construction, which we shall make use of throughout the remainder of the paper. See Figure \ref{fig:Z_Sigma} for an illustration.

\begin{definition} \label{def:Z_Sigma}
Let $\Sigma \in \fra S$ be a skeleton pairing. We define a graph $\cal Y(\Sigma)$ on the vertex set $V(\cal Y(\Sigma)) \deq Q(\Sigma)$ as follows. Each bridge $\{e,e'\} \in \Sigma$ gives rise to the edge $\{q,q'\}$ of $\cal Y(\Sigma)$, where $q$ and $q'$ are defined as the blocks of $Q(\Sigma)$ that contain $a(e)$ and $b(e)$ respectively. (Note that, by definition of $Q$, we also have $a(e') \in q$ and $b(e') \in q'$).
We call  $\cal Y(\Sigma)$ the graph associated with $\Sigma$.
\end{definition}
\begin{figure}[ht!]
\begin{center}
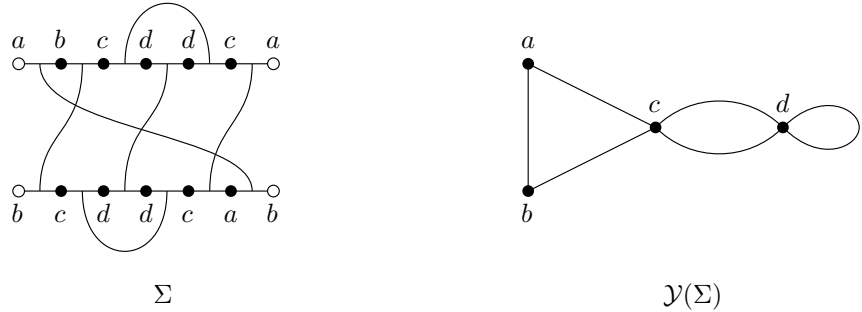
\end{center}
\caption{A skeleton pairing $\Sigma$ together with its associated graph $\cal Y(\Sigma)$.
 In $\Sigma$ we use the letters $a,b,c,d$ next to the vertices to indicate the four blocks of $Q(\Sigma)$. (We emphasize that the vertices of $\cal Y(\Sigma)$, unlike those of $\Sigma$, are not classified using colours; our use of black dots in the right-hand
picture has no mathematical relevance.) 
\label{fig:Z_Sigma}}
\end{figure}

Recall Definition \ref{def:C_Gamma} for the meaning of $\cal C(\Sigma)$. 
Then $\cal Y(\Sigma)$ is simply obtained as a minor of  $\cal C(\Sigma)$
after contracting (identifying) vertices that belong to the same blocks
of $Q(\Sigma)$ and replacing every  pair of edges of $\cal C(\Sigma)$ forming a bridge  with a
single edge. In particular, the skeleton bridges of $\Sigma$ become the edges of $\cal Y(\Sigma)$,
i.e.\ $\Sigma$ and $E(\cal Y(\Sigma))$ may be canonically identified. Similarly,
$Q(\Sigma)$ is canonically identified with $V(\cal Y(\Sigma))$, the vertex set
of the associated graph.

\begin{lemma} \label{lem:Z_Sigma_connected}
For any $\Sigma \in \fra S$ the associated graph $\cal Y(\Sigma)$ is connected.
\end{lemma}
\begin{proof}
This follows immediately from the definition of $\cal Y(\Sigma)$ and the fact that $\Sigma \in \fra M_c$.
\end{proof}

Next, let $\Sigma \in \fra S$ be fixed. Starting from the definition \eqref{def of V Sigma}, we use \eqref{bound on gamma} to get
\begin{equation} \label{cal_V_estimate}
\abs{\cal V(\Sigma)} \;\leq\; C \sum_{\f b \in \N^\Sigma} \indBB{2 \sum_{\sigma \in \Sigma} b_\sigma \leq M^\mu }
\sum_{\f y \in \bb T^{Q(\Sigma)}} \sum_{\f x \in \bb T^{V(\Sigma)}} \pBB{\prod_{q \in Q(\Sigma)} \prod_{i \in q} \ind{x_i = y_q}} \pBB{\prod_{\{e, e'\} \in \Sigma} (S^{b_{\{e,e'\}}})_{x_e}}\,.
\end{equation}
For future reference we note that the right-hand side of \eqref{cal_V_estimate} may also be written without the partition $Q(\Sigma)$ as
\begin{equation}
C \sum_{\f b \in \N^\Sigma} \indBB{2 \sum_{\sigma \in \Sigma} b_\sigma \leq M^\mu} \sum_{\f x \in \bb T^{V(\Sigma)}}  I_0(\f x) \prod_{\{e,e'\} \in \Sigma} J_{\{e,e'\}}(\f x) \, \pb{S^{b_{\{e,e'\}}}}_{x_e}\,,
\end{equation}
where $I_0$ was defined in \eqref{def_I_ind} and $J_{\{e,e'\}}$ in \eqref{def_I_sigma}.
Recall that the free variables in \eqref{cal_V_estimate} are $\f y$. Using $\cal Y(\Sigma)$, we may rewrite \eqref{cal_V_estimate} in the form
\begin{equation} \label{cal_V_estimate2}
\abs{\cal V(\Sigma)} \;\leq\; C \sum_{\f b \in \N^{E(\cal Y(\Sigma))}} \indBB{2 \sum_{e \in E(\cal Y(\Sigma))} b_e \leq M^\mu }
\sum_{\f y \in \bb T^{V(\cal Y(\Sigma))}} \pBB{\prod_{e \in E(\cal Y(\Sigma))} (S^{b_e})_{y_e}}\,,
\end{equation}
where we recall the convention $y_{\{q,q'\}} = (y_q, y_{q'})$.

Let
\begin{equation} \label{def_Qb}
Q_b(\Sigma) \;\deq\; \hb{q \in Q(\Sigma) \col q \text{ contains a black vertex of } V(\Sigma)}\,.
\end{equation}
It is easy to see that $\abs{Q(\Sigma) \setminus Q_b(\Sigma)} \leq 2$.
Next, we state the fundamental counting rule behind our estimates; its analogue in \cite{EK1} was called the 2/3-rule.  It says that each block of $Q(\Sigma)$ contains at least three
vertices,  with the possible exception of blocks consisting exclusively of white vertices.

\begin{lemma}[2/3-rule] \label{lem:2/3_rule}
Let $\Sigma \in \fra S$. 
For all $q \in Q_b(\Sigma)$ we have $\abs{q} \geq 3$. Moreover,
\begin{equation} \label{3Qb_estimateshort}
\abs{Q_b(\Sigma)} \;\leq \;  \frac{2}{3} \abs{\Sigma} + \frac{2}{3}\,.
\end{equation}
\end{lemma}
\begin{proof}
By definition of $\fra S$, we have that $\abs{q} \geq 2$. Now suppose that $\abs{q} = 2$. Let $i \in q$ be a black vertex of $V(\Sigma)$. Since $\abs{q} = 2$, we conclude that the two bridges of $\Sigma$ touching $i$ (see Definition \ref{def:ladders} (ii)) are parallel. This is in contradiction with the definition of $\fra S$. Finally, \eqref{3Qb_estimateshort} follows directly from $\abs {q} \geq 3$, since
\begin{equation} \label{3Qb_estimate}
3 \abs{Q_b(\Sigma)} \;\leq\; \sum_{q \in Q_b(\Sigma)} \abs{q} \;\leq\; \abs{V(\Sigma)} \;=\; 2 \abs{\Sigma} + 2\,. \qedhere
\end{equation}
\end{proof}
Since $\abs{Q(\Sigma)} \leq \abs{Q_b(\Sigma)} + 2$, we get from \eqref{3Qb_estimateshort} that
\begin{equation} \label{Q_Sigma_bound}
\abs{Q(\Sigma)} \;\leq\;  \frac{2 \abs{\Sigma}}{3} + \frac{8}{3}\,.
\end{equation}
Next, using Lemma \ref{lem:Z_Sigma_connected} we choose some (immaterial) spanning tree $\cal T$ of $\cal Y(\Sigma)$. Clearly, $\abs{E(\cal T)} = \abs{Q(\Sigma)} - 1$ and $\abs{E(\cal Y(\Sigma))} = \abs{\Sigma}$, so that \eqref{Q_Sigma_bound} yields
\begin{equation} \label{E_Z_diff_leq}
\absb{E(\cal Y(\Sigma)) \setminus E(\cal T)} \;\geq\; \frac{\abs{\Sigma}}{3} - \frac{5}{3}\,.
\end{equation}
We now sum over $\f y$ in \eqref{cal_V_estimate2}, using the estimates, valid for any $b \leq M^\mu$,
\begin{equation} \label{S_basic_bounds}
\sum_{z} (S^b)_{yz} \;\leq\; C \,, \qquad (S^b)_{yz} \;\leq\; \frac{C}{M}\,,
\end{equation}
which are easy consequences of $S_{yz} \leq C M^{-1}$ and $\sum_z S_{yz} = \cal I$. In the product on the right-hand side of \eqref{cal_V_estimate2}, we estimate each factor associated with $\{q,q'\} \notin E(\cal T)$ by $C M^{-1}$, using the second estimate of
\eqref{S_basic_bounds}. We then sum out all of the $\f y$-labels, starting from the leaves of $\cal T$ (after some immaterial choice of root), at each summation using the first estimate of \eqref{S_basic_bounds}. This yields
\begin{equation*}
\sum_{\f y \in \bb T^{V(\cal Y(\Sigma))}} \pBB{\prod_{e \in E(\cal Y(\Sigma))} (S^{b_e})_{y_e}} \;\leq\; N \pbb{\frac{C}{M}}^{\abs{E(\cal Y(\Sigma)) \setminus E(\cal T)}} \;\leq\; N \pbb{\frac{C}{M}}^{\abs{\Sigma}/3 - 5/3}\,,
\end{equation*}
where in the last step we used \eqref{E_Z_diff_leq}. The factor $N$ results from the summation over the label associated with the root of $\cal T$. Thus we find from \eqref{cal_V_estimate2}
\begin{multline*}
\abs{\cal V(\Sigma)} \;\leq\; N \pbb{\frac{C}{M}}^{\abs{\Sigma}/3 - 5/3} \sum_{\f b \in N^{E(\cal Y(\Sigma))}} \indBB{2 \sum_{e \in E(\cal Y(\Sigma))} b_e \leq M^\mu}
\\
=\; N \pbb{\frac{C}{M}}^{\abs{\Sigma}/3 - 5/3} \binom{[M^\mu/2] - 1}{\abs{\Sigma} - 1}
\;\leq\; N \pbb{\frac{C}{M}}^{\abs{\Sigma}/3 - 5/3} \frac{M^{\mu \abs{\Sigma}}}{(\abs{\Sigma} - 1) !}\,.
\end{multline*}

Next, for any $m \in \N$, a simple combinatorial argument shows that the number of skeleton pairings $\Sigma \in \fra S$ satisfying $\abs{\Sigma} = m$ is bounded by 
\begin{equation} \label{comb_bound}
(2m + 1) \, \frac{(2m)!}{m! 2^m} \;\leq\; C^m \, m!\,;
\end{equation}
here the factor $\frac{(2m)!}{m! 2^m}$ is the number of pairings of $2m$ edges, and the factor $2m + 1$ is the number of graphs $\cal C$ with $2m$ edges.
We therefore conclude that
\begin{equation*}
\sum_{\Sigma \in \fra S^>} \abs{\cal V(\Sigma)} \;\leq\; N \sum_{m = K}^\infty C^m m! \pbb{\frac{C}{M}}^{m/3 - 5/3} \frac{M^{\mu m}}{(m - 1) !} \;\leq\;
N M^{5/3} \sum_{m = K}^\infty \p{C M^{\mu - 1/3}}^{m} \;\leq\; C_K N M^{5/3 + K(\mu - 1/3)}\,.
\end{equation*}
Choosing $K$ large enough completes the proof of Proposition \ref{prop:large_Sigma}.

We conclude this subsection by summarizing the origin of the restriction $\mu < 1/3$ (and hence $\rho < 1/3$), as it appears in the preceding proof of Proposition \ref{prop:large_Sigma}. The total contribution of a skeleton is determined by a competition between its \emph{size} (given by the number of bridges) and its \emph{entropy factor} (given by the number of independent summation labels $\f y$). Each bridge yields, after resummation, a factor $(M\eta)^{-1}$, so that the size of the graph is $(M\eta)^{-s}$ where $s = \abs{\Sigma}$ is the number of ladders. The entropy factor is $M^\ell$ where $\ell = \abs{Q(\Sigma)}$ is the number independent summation labels. The 2/3-rule from Lemma \ref{lem:2/3_rule} states roughly that $\ell \leq 2b /3$. The sum of the contributions of all skeletons is convergent if $(M\eta)^{-s} M^\ell \ll 1$, which, by the 2/3-rule, holds provided that $\eta\gg M^{-1/3}$.

\subsection{The error terms: small skeletons} \label{sec:small_skeletons}
We now focus on the estimate of the small skeletons, i.e.\ we estimate $\cal V(\Sigma)$ for $\Sigma \in \fra S^\leq$
(recall the splitting \eqref{smalllarge}).
 The details of the following estimates will be somewhat different for the two cases \textbf{(C1)} and \textbf{(C2)}; for definiteness, we focus on the (harder) case \textbf{(C2)}, i.e.\ we assume that $\phi_1$ and $\phi_2$ both satisfy \eqref{non_Cauchy}. The analogue of the following result in the case \textbf{(C1)} is given in Proposition \ref{prop:small_Sigma_c1} at the end of this subsection.

\begin{proposition} \label{prop:small_Sigma}
Suppose that $\phi_1$ and $\phi_2$ satisfy \eqref{non_Cauchy}. Suppose moreover that \eqref{D leq kappa} holds for some small enough $c_* > 0$. Then for any fixed $K \in \N$ and small enough $\delta > 0$ in Proposition \ref{prop: expansion with trunction} there exists a constant $c_0 > 0$ such that 
\begin{equation} \label{sum_V_Sigma_small}
\sum_{\Sigma \in \fra S^\leq_K} \abs{\cal V(\Sigma)} \;\leq\; \frac{C_K N}{M} R_2(\omega + \eta)  M^{-c_0}\,,
\end{equation}
 where we recall the definition of $R_2$ from \eqref{def_R2}.
\end{proposition}

Note that, by Proposition \ref{prop: leading term}, the size of the dumbbell skeletons is
\begin{equation}\label{true}
\abs{\cal V_{\rm main}} \;\asymp\;  \frac{N}{M} \pB{1 + \ind{d \leq 3} (\omega + \eta)^{d/2 - 2} + \ind{d = 4} \absb{\log (\omega + \eta)}}\,, 
\end{equation}
unless $d = 2$ and $\omega \gg \eta$, in which case we have
\begin{equation}\label{true2}
\abs{\cal V_{\rm main}} \;\asymp\;  \frac{N}{M} (1 + \abs{\log \omega})\,.
\end{equation}
We conclude that the right-hand side of \eqref{sum_V_Sigma_small} is much smaller than the contribution of the dumbbell skeletons. In particular, the proof of
Proposition \ref{prop:small_Sigma} reveals why precisely the dumbbell skeletons provide the leading contributions.

In this section we give a sketch of the proof of  Proposition \ref{prop:small_Sigma}, followed by 
the actual proof in the next section.
  As explained at the beginning of Section \ref{sec:large_Sigma}, the combinatorics of the summation over $\Sigma$ 
are now trivial, since the cardinality of the set $\fra S^\leq \equiv \fra S^\leq_K$
depends only on $K$, which is fixed. However, the brutal estimate of \eqref{cal_V_estimate}, which neglects the oscillations present in the coefficients $\gamma$, is not good enough. For small skeletons, it is essential to exploit these oscillations.

First we undo the truncation in the definition of $\wt \gamma_{n_i}$ 
and use \eqref{gamma - g} to replace $\wt \gamma_{n_i}$ with $\psi^\eta_i * \gamma_{n_i}$  in the  definition \eqref{def of V Sigma} of $\cal V(\Sigma)$.
Then we rewrite the real parts in \eqref{def of V Sigma} using \eqref{2re2re}
this gives rise to two terms, and we focus on
 the first 
one, which we call $\cal V'(\Sigma)$. (The other one may be estimate in exactly the same way and is in fact 
smaller.)  The summation over $\f b$ in \eqref{def of V Sigma} can  now be performed explicitly using geometric series. The result is that each skeleton bridge $\sigma \in \Sigma$ encodes  an entry of the quantity $\cal Z(\sigma)$, which is roughly a resolvent of $S$ multiplied by a phase $\alpha$, i.e.\ $(1-\alpha S)^{-1}$. 
It turns out that these phases $\alpha$ depend strongly on the type of bridge they belong to. We split the set of skeleton bridges $\Sigma = \Sigma_d \sqcup \Sigma_c$ into the ``domestic bridges'' which join edges within the same component of $\cal C$ and ``connecting bridges'' which join edges in different components of $\cal C$; see Definition \ref{def_split_Sigma} below for more details. The critical regime is when $\alpha \approx 1$, which yields a singular resolvent $(1 - \alpha S)^{-1}$ (see the discussion on the spectrum of $S$ in Section \ref{sec: heuristic leading term}). 
The phase $\alpha$ associated with a domestic bridge is separated away from $1$, which yields a regular resolvent.
 (This may also be interpreted as strong oscillations in the geometric series of the resolvent expansion.)
 The phase $\alpha$ associated with a connecting bridge is close to $1$ and the associated resolvent is therefore much more  singular. More precisely (see Lemma \ref{lem: bounds on S} below),  we find that these resolvents $\cal Z(\sigma)$ satisfy the bounds
\begin{equation} \label{domestic_bound_sketch}
\absb{\cal Z(\sigma)_{yz}} \;\lesssim\; M^{-1}\,, \qquad
\sum_{z} \absb{\cal Z(\sigma)_{yz}} \;\lesssim\; 1
\end{equation}
for  domestic bridges  $\sigma \in \Sigma_d$ and
\begin{equation} \label{connecting_bound_sketch}
\absb{\cal Z(\sigma)_{yz}} \;\lesssim\; M^{-1} R_2(\omega + \eta)\,, \qquad
\sum_z \absb{\cal Z(\sigma)_{yz}} \;\lesssim\; M^\mu\,,
\end{equation}
for  connecting bridges  $\sigma \in \Sigma_c$. (Recall the definition of $R_2$ from \eqref{def_R2}.)

Using the bounds \eqref{domestic_bound_sketch} and \eqref{connecting_bound_sketch} we get a simple bound on $\cal V'(\Sigma)$. The rest of the argument is purely combinatorics and power counting:
we have to make sure that for any $\Sigma \in \fra S^\leq$ this bound is small enough, i.e.\ $o(N/M)$. Without loss of generality we may assume that $\Sigma$ does not contain a bridge that touches (see Definition \ref{def:ladders}) the 
two white vertices of the same component of $\cal C$. Indeed, if $\Sigma$ contains such a bridge, we can sum
up the  (coinciding) labels of the two white vertices  using the second bound of \eqref{domestic_bound_sketch}, which effectively
removes such a bridge, as depicted in Figure \ref{fig:SSbar} below. 
 In particular, we have $Q_b(\Sigma) = Q(\Sigma)$. (Recall the definitions of $Q(\Sigma)$ after \eqref{indicator function behind Q} and of $Q_b(\Sigma)$ from \eqref{def_Qb}).

We perform the summation over the labels $\f x$ as in Section \ref{sec:large_Sigma}: by choosing a spanning tree on the graph $\cal Y(\Sigma)$. Recall that there is a canonical bijection between the edges of $\cal Y(\Sigma)$ and the bridges of $\Sigma$. Denote by $\Sigma_t$ the bridges associated with the spanning tree of $\cal Y(\Sigma)$. The combinatorics rely on the following quantities:
\begin{align*}
\ell &\;\deq\; \abs{Q(\Sigma)} \;=\; \abs{V(\cal Y(\Sigma))} \;= \;  \text{number of independent labels}\,,
\\
s_d &\;\deq\; \abs{\Sigma_d} \;=\; \text{number of domestic bridges}\,,
\\
s_{t} &\;\deq\; \abs{\Sigma_c \cap \Sigma_t} \;=\; \text{number of connecting tree bridges}\,,
\\
s_{l} &\;\deq\; \abs{\Sigma_c \setminus \Sigma_t} \;=\; \text{number of connecting loop (i.e.\ non-tree) bridges}\,.
\end{align*}
Note that the total number of bridges is $s \deq \abs{\Sigma} = s_d + s_{t} + s_{l}$.
Moreover, $s\geq \ell -1$ since  $\cal Y(\Sigma)$ is connected and $s_t\leq\ell -1$ since
$s_t$ is part of a spanning tree. 
From the 2/3-rule in \eqref{3Qb_estimateshort} we conclude that $\abs{q} \geq 3$ for all $q\in Q(\Sigma)$ and
\begin{equation} \label{2/3_rule_sketch}
\ell \;\leq\; \frac{2 (1 + s)}{3}\,.
\end{equation}
Using the bounds \eqref{domestic_bound_sketch} and
 \eqref{connecting_bound_sketch}, we sum over the labels $\f x$ associated with the vertices of $\Sigma$, and find the estimate
\begin{equation} \label{main_bound_sketch}
\abs{\cal V'(\Sigma)} \;\lesssim\; N M^{\ell - s - 1} R_2^{s_{l}} M^{\mu s_{t}}.
\end{equation}
Indeed, the root of the spanning tree gives rise to a factor $N$; each one of the $s - \ell + 1$ bridges not associated with the spanning tree gives rise to a factor $M^{-1}$; each one of the $s_{l}$ connecting loop bridges gives rise to an additional factor $R_2$; and each one of the $s_{t}$ connecting tree bridges gives rise to a factor $M^\mu$.

It is instructive to compare the upper bound \eqref{main_bound_sketch} for $\Sigma$ being a dumbbell to the true size of the dumbbell skeletons from \eqref{true}.
 Since we exclude pairings with bridges touching the two white vertices of the same component of $\cal C$, we may take $\Sigma$ to be $D_1$ or $D_5$ (see Figure \ref{fig: dumbbell}). Of these two, $D_5$ saturates the 2/3-rule and is of leading order. For $\Sigma = D_5$ we have $\ell = 2$, $s = 2$, $s_l = 1$, $s_t = 1$. Hence the bound \eqref{main_bound_sketch} reads
\begin{equation} \label{estimate_D5}
\abs{\cal V'(D_5)} \;\lesssim\; \frac{N}{M} R_2(\omega + \eta) M^\mu\,.
\end{equation}
This is in general much larger than the true size  \eqref{true}; 
they become comparable for $\omega + \eta \asymp M^{-\mu}$ (i.e.\ on very small scales), which is ruled out by our assumptions on $\omega$ and $\eta$. 

Now we explain how the estimate on $\cal V'(\Sigma)$ can be improved
if $\Sigma$ \emph{is not a dumbbell skeleton}. We rely on two simple but
fundamental observations.
 First, if $\Sigma$ does not saturate the 2/3-rule then the right-hand side of \eqref{main_bound_sketch} contains an extra power of $M^{-1/3}$ as compared to the leading term \eqref{estimate_D5}.
 Second, if $\Sigma$ saturates the 2/3-rule and is not a dumbbell skeleton
then $\Sigma$ must contain a domestic bridge (joining edges within the same component of $\cal C$).
Having a domestic bridge implies that $s_l + s_t \leq s - 1$ instead of the trivial bound $s_l + s_t \leq s$. This implies that the power of one of the large factors $R_2$ or $M^\mu$ on the right-hand side of \eqref{main_bound_sketch} will be reduced by one; as it turns out, this is sufficient to make 
the right-hand side of \eqref{main_bound_sketch} subleading.
Note that the absence of such domestic bridges in $\Sigma$ 
is the key feature that singles out the dumbbells  among all
skeletons that saturate the $2/3$-rule. This explains why the
leading contribution in \eqref{F in terms of skeletons} comes from the dumbbell skeletons.

We now  explain these two scenarios more precisely.  For the rest of this 
subsection we suppose that $\Sigma$ is not a dumbbell skeleton. Hence
\begin{equation} \label{s-lgeq1}
s \;\geq\; 3 \,, \qquad \ell \;\geq\; 2 \,, \qquad s - \ell \;\geq\; 1\,;
\end{equation}
the first two estimates are immediate, and the last one follows from the first combined with \eqref{2/3_rule_sketch} and the fact that $s - \ell \in \N$.

Suppose first that $\Sigma$ saturates the 2/3-rule \eqref{2/3_rule_sketch}. Then $\abs{q} = 3$ for all $q \in Q(\Sigma)$, and it is not too hard to see that $\Sigma$ must contain a domestic bridge, i.e.\ $s_d \geq 1$. Roughly, this follows from the observation that in order to get a block of size three, the bridges touching the vertices of this block must be as in Figure \ref{fig:three-lump} below. Plugging \eqref{2/3_rule_sketch} into \eqref{main_bound_sketch} yields
\begin{equation*}
\abs{\cal V'(\Sigma)} \;\lesssim\; \frac{N}{M} M^{2/3 - s/3} R_2^{s_{l}} M^{\mu s_{t}}
\;\leq\; \frac{N}{M} M^{2/3 - s/3} R_2^{s_{l} + s_t - \ell + 1} M^{\mu (\ell - 1)}
\;\leq\; \frac{N}{M} M^{2/3 - s/3} R_2^{s - \ell} M^{\mu (\ell - 1)}\,,
\end{equation*}
where the second step follows from $R_2 \leq M^\mu$ and the third step from $s_l + s_t \leq s - 1$ (since $s_d \geq 1$). We conclude that
\begin{equation*}
\abs{\cal V'(\Sigma)} \;\lesssim\; \frac{N}{M} M^{1/3} (M^{-1/3} R_2)^{s - \ell} (M^{-1/3} M^\mu)^{\ell - 1} \;\ll\; \frac{N}{M} R_2\,,
\end{equation*}
where we used \eqref{s-lgeq1} and  $R_2 + M^\mu \leq M^{1/3}$.

Next, consider the case where $\Sigma$ does not saturate the 2/3-rule \eqref{2/3_rule_sketch}.
In this case it may well be that $s_d = 0$. However, if \eqref{2/3_rule_sketch} is not saturated, then there must exist a $q \in Q(\Sigma)$ satisfying $\abs{q} \geq 4$. Thus \eqref{2/3_rule_sketch} improves to
\begin{equation*}
\ell \;\leq\; \frac{1}{3} + \frac{2 s}{3}\,.
\end{equation*}
Thus we find that
\begin{equation*}
\abs{\cal V'(\Sigma)} \;\lesssim\; \frac{N}{M} M^{1/3 - s/3} R_2^{s_{l}} M^{\mu s_{t}}\,.
\end{equation*}
Note that we have $s_t + s_l \leq s$ and $s_t \leq \ell - 1$.
 Using $R_2 \leq M^\mu$ we therefore get
\begin{equation*}
\abs{\cal V'(\Sigma)} \;\lesssim\; \frac{N}{M} M^{1/3 - s/3} R_2^{s_{l}} M^{\mu s_{t}} \;\leq\; \frac{N}{M} M^{1/3 - s/3} R_2^{s - \ell + 1} M^{\mu (\ell - 1)} \;=\; \frac{N}{M} M^{1/3} (M^{-1/3} R_2)^{s - \ell + 1} (M^{-1/3} M^\mu)^{\ell - 1}\,.
\end{equation*}
From \eqref{s-lgeq1} we therefore get 
$\abs{\cal V'(\Sigma)} \ll \frac{N}{M} R_2$. This concludes the sketch of the proof of Proposition \ref{prop:small_Sigma}.

\subsection{Proof of Proposition \ref{prop:small_Sigma}} \label{proof:small_Sigma}

We begin the proof by rewriting \eqref{def of V Sigma} in a form where the oscillations in the summation over $\f b$ may be effectively exploited. This consists of three steps, each of which results in negligible errors of order $C_q N M^{-q}$ for any $q > 0$.
 In the first step, we decouple the $\f b$-summations by replacing the indicator
 function $\indb{2 \sum_{\sigma \in \Sigma} b_\sigma \leq M^\mu}$ with the product 
$\prod_{\sigma \in \Sigma} \ind{b_\sigma \leq M^\mu}$, using the estimate \eqref{bound on gamma}.
 In the second step, we replace the factors $\wt \gamma_{n_i(\Sigma, \f b)}(E_1, \phi_i)$ 
with $(\gamma_{n_i(\Sigma, \f b)} * \psi_i^\eta)(E_i)$, using the estimate \eqref{gamma - g}.
These two steps are analogous to the steps from \eqref{VDdef} to
 \eqref{common expression for V(D)}.

In the third step, we truncate in the tails of the functions $\psi_i$ on the scale $M^{\delta/2}$, where $\delta > 0$ is the constant from Proposition \ref{prop: expansion with trunction}. To that end, we choose a smooth, nonnegative, symmetric function $\chi$ satisfying $\chi(E) = 1$ for $\abs{E} \leq 1$ and $\chi(E) = 0$ for $\abs{E} \geq 2$. We split $\psi_i = \psi^{\leq}_i + \psi^{>}_i$, where
\begin{equation} \label{phi_cutoff1}
\psi^{\leq}_i(E) \;\deq\; \psi_i(E) \chi(M^{-\delta/2} E)\,, \qquad
\psi^{>}_i(E) \;\deq\; \psi_i(E) \pb{1 - \chi(M^{-\delta/2} E)}\,
\end{equation}
This yields the splitting $\psi^\eta_i = \psi^{\leq, \eta}_i + \psi^{>, \eta}_i$ of the rescaled test function $\psi^\eta(E) = \eta^{-1}\psi(\eta^{-1}E)$. This splitting is done on the scale $\eta M^{\delta/2}$, and we have
\begin{equation} \label{supp_phi_trunc}
\supp \psi_i^{\leq, \eta} \;\subset\; [-2 \eta M^{\delta/2}, 2 \eta M^{\delta/2}]\,.
\end{equation}
Moreover, recalling \eqref{non_Cauchy} and using the trivial bound $\abs{\gamma_n(E)} \leq C$ we find
\begin{equation} \label{phi_geq_bound}
\absb{(\psi_i^{>,\eta} * \gamma_{n})(E_i)} \;\leq\; C_q M^{-q}
\end{equation}
for any $q > 0$.
The truncation of the third step is the replacement of $(\gamma_{n_i(\Sigma, \f b)} * \psi_i^{\eta})(E_i)$ with $(\gamma_{n_i(\Sigma, \f b)} * \psi_i^{\leq, \eta})(E_i)$, using \eqref{phi_geq_bound}.

Applying these three steps to the definition \eqref{def of V Sigma} yields
\begin{multline} \label{V_cutoff}
\cal V(\Sigma) \;=\; \sum_{\f b \in B(\Sigma)} \pBB{\prod_{\sigma \in \Sigma} \ind{b_\sigma \leq M^\mu}} \, 2 \re \pb{\gamma_{n_1(\Sigma, \f b)} * \psi_1^{\leq, \eta}}(E_1)
\, 2 \re \pb{\gamma_{n_2(\Sigma, \f b)} * \psi_2^{\leq, \eta}}(E_2)
\\
\times \sum_{\f y \in \bb T^{Q(\Sigma)}} \sum_{\f x \in \bb T^{V(\Sigma)}} \pBB{\prod_{q \in Q(\Sigma)} \prod_{i \in q} \ind{x_i = y_q}} \pBB{\prod_{\{e, e'\} \in \Sigma} (S^{b_{\{e,e'\}}})_{x_e}} + O_{q,\Sigma}(N M^{-q})\,.
\end{multline}
The errors arising from each of the three steps are estimated using \eqref{bound on gamma}, \eqref{gamma - g}, and \eqref{phi_geq_bound} respectively. The summations over $\f b$ and $\f y$ in the error terms are performed brutally, exactly as in the proof of Proposition \ref{prop:large_Sigma} (in fact here we only need that $\cal Y(\Sigma)$ be connected); we omit the details.

Next, we use \eqref{2re2re} to write $\cal V(\Sigma) =  2 \re  (\cal V'(\Sigma) + \cal V''(\Sigma)) + O_{q,\Sigma}(N M^{-q})$,  where
\begin{multline} \label{V_cutoff2}
\cal V'(\Sigma) \;\deq\; \sum_{\f b \in B(\Sigma)} \pBB{\prod_{\sigma \in \Sigma} \ind{b_\sigma \leq M^\mu}} \, \pb{\gamma_{n_1(\Sigma, \f b)} * \psi_1^{\leq, \eta}}(E_1)
\, \pb{\ol{\gamma_{n_2(\Sigma, \f b)}} * \psi_2^{\leq, \eta}}(E_2)
\\
\times \sum_{\f y \in \bb T^{Q(\Sigma)}} \sum_{\f x \in \bb T^{V(\Sigma)}} \pBB{\prod_{q \in Q(\Sigma)} \prod_{i \in q} \ind{x_i = y_q}} \pBB{\prod_{\{e, e'\} \in \Sigma} (S^{b_{\{e,e'\}}})_{x_e}}\,,
\end{multline}
and $\cal V''(\Sigma)$ is defined similarly but without the complex conjugation on $\gamma_{n_2(\Sigma, \f b)}$.
 We shall gives the details of the estimate for the larger error term, $\cal V'(\Sigma)$. The term $\cal V''(\Sigma)$ may be estimated using an almost identical argument; we sketch the minor differences below.

In order to estimate the right-hand side of \eqref{V_cutoff}, we shall have to classify the bridges of $\Sigma$ into three classes according to the following definition.

\begin{definition} \label{def_split_Sigma}
For $i = 1,2$ we define
\begin{equation*}
\Sigma_i \;\deq\; \hb{\sigma \in \Sigma \col \sigma \subset E(\cal C_i)}\,,
\end{equation*}
the set of bridges consisting only of edges of $\cal C_i$. We abbreviate $\Sigma_d \deq \Sigma_1 \cup \Sigma_2$ (the set of ``domestic bridges''). We also define $\Sigma_c \deq \Sigma \setminus \Sigma_d$, the set of bridges connecting the two components of $\cal C$. Moreover, for $\# = 1,2,c,d$ we introduce the set $E_\#(\cal Y(\Sigma))$ defined as the subset of $E(\cal Y(\Sigma))$ encoded by $\Sigma_\#$ under the canonical identification $\Sigma \simeq E(\cal Y(\Sigma))$, according to Definition \ref{def:Z_Sigma}.
\end{definition}
Since each $\sigma \in \Sigma_c$ contains one edge of $\cal C_1$ and one edge of $\cal C_2$, and each $\sigma \in \Sigma_i$ contains two edges of $\cal C_i$, we find that the number of edges in the $i$-th chain $\cal C_i(n_i)$ of the graph
$\cal C(n_1, n_2)$ with pairing
 $\Pi = (\Sigma, \f b)$ is
\begin{equation*}
n_i(\Sigma, \f b) \;=\; \sum_{\sigma \in \Sigma_c} b_\sigma + 2 \sum_{\sigma \in \Sigma_i} b_\sigma\,.
\end{equation*}
Here we identify $\Pi$ with $(\Sigma, \f b)$, as remarked after Lemma \ref{lem: skeletons}.

We may now plug into \eqref{V_cutoff2} the explicit expression for $\gamma_n$ from \eqref{claim about gamma n}, at which point it is convenient to introduce the abbreviations
\begin{equation}
T(E) \;\deq\; \frac{2}{1 - (M - 1)^{-1} \ee^{2 \ii \arcsin E}}\,, \qquad A_i \;\deq\; \arcsin E_i\,.
\end{equation}
Thus we get from \eqref{V_cutoff2} 
\begin{multline*}
\cal V'(\Sigma) \;=\; \sum_{\f b \in B(\Sigma)} \pBB{\prod_{\sigma \in \Sigma} \ind{b_\sigma \leq M^\mu}}
\\ 
\times
\pBB{T(E_1) \ol{T(E_2)} \, \ee^{\ii (A_1 - A_2)} \, \prod_{\sigma \in \Sigma_1} (-\ee^{2 \ii A_1})^{b_\sigma} \prod_{\sigma \in \Sigma_2} (-\ee^{-2 \ii A_2})^{b_\sigma} \prod_{\sigma \in \Sigma_c} \ee^{\ii (A_1 - A_2) b_\sigma}} * \psi_1^{\leq, \eta}(E_1) * \psi_2^{\leq, \eta}(E_2)
\\
\times \sum_{\f y \in \bb T^{Q(\Sigma)}} \sum_{\f x \in \bb T^{V(\Sigma)}} \pBB{\prod_{q \in Q(\Sigma)} \prod_{i \in q} \ind{x_i = y_q}} \pBB{\prod_{\{e, e'\} \in \Sigma} (S^{b_{\{e,e'\}}})_{x_e}}\,.
\end{multline*}
 Here, by a slight abuse of notation, we write $(\varphi * \chi)(E) \equiv \varphi(E) * \chi(E)$. 
Using the associated graph $\cal Y(\Sigma)$ from Definition \ref{def:Z_Sigma}, we may rewrite this as
\begin{multline} \label{V_cutoff3}
\cal V'(\Sigma) \;=\; \sum_{\f b \in \{1, \dots, [M^\mu]\}^{E(\cal Y(\Sigma))}} 
 \ind{\f b \in B(\Sigma)} 
\\ 
\times
\pBB{T(E_1) \ol{T(E_2)} \, \ee^{\ii (A_1 - A_2)} \, \prod_{e \in E_1(\cal Y(\Sigma))} (-\ee^{2 \ii A_1})^{b_e} \prod_{e \in E_2(\cal Y(\Sigma))} (-\ee^{-2 \ii A_2})^{b_e} \prod_{e \in E_c(\cal Y(\Sigma))} \ee^{\ii (A_1 - A_2) b_e}}
\\
* \psi_1^{\leq, \eta}(E_1) * \psi_2^{\leq, \eta}(E_2)
\, \sum_{\f y \in \bb T^{V(\cal Y(\Sigma))}} \pBB{\prod_{e \in E(\cal Y(\Sigma))} (S^{b_e})_{y_e}}\,,
\end{multline}
where we used the canonical identification between $\Sigma$ and $E(\cal Y(\Sigma))$ to rewrite the set $B(\Sigma) \subset \N^\Sigma$ from Lemma \ref{lem: skeletons} as a subset $B(\Sigma) \subset \N^{E(\cal Y(\Sigma))}$ (by a slight abuse of notation). Also, to avoid confusion, we emphasize that the expressions $E_i$ and $E_i(\cal Y(\Sigma))$ have nothing to do with each other.

Next, we split $\cal V'(\Sigma) = \cal V'_0(\Sigma) - \cal V'_1(\Sigma)$ using the splitting  $\ind{\f b \in B(\Sigma)} = 1 - \ind{\f b \notin B(\Sigma)}$ in \eqref{V_cutoff3}.
We first focus on main term, $\cal V_0'(\Sigma)$. In the definition of $\cal V'_0(\Sigma)$, we may sum the geometric series associated with each summation variable $b_e$ to get
\begin{multline} \label{V_cutoff4}
\cal V'_0(\Sigma) \;=\; \sum_{\f y \in \bb T^{V(\cal Y(\Sigma))}}
\Biggl(T(E_1) \ol{T(E_2)} \, \ee^{\ii (A_1 - A_2)} \, \prod_{e \in E_1(\cal Y(\Sigma))} Z(-\ee^{2 \ii A_1} S)_{y_e} \prod_{e \in E_2(\cal Y(\Sigma))} Z(-\ee^{-2 \ii A_2} S)_{y_e}
\\
\times \prod_{e \in E_c(\cal Y(\Sigma))} Z(\ee^{\ii (A_1 - A_2)} S)_{y_e} \Biggr)
* \psi_1^{\leq, \eta}(E_1) * \psi_2^{\leq, \eta}(E_2)\,,
\end{multline}
where we abbreviated
\begin{equation} \label{def_Z}
Z(x) \;\deq\; \sum_{b = 1}^{[M^\mu]} x^b \;=\; \frac{x (1 - x^{[M^\mu]})}{1 - x}\,
\end{equation}
for any quantity $x$, which may be a number or a matrix.
The explicit summation over $\f b$ exploits the cancellations associated with the highly oscillating summands. From now on, we shall freely estimate the summation over $\f y$ by taking the absolute value inside the sum.

On the right-hand side of \eqref{V_cutoff4}, each edge $e \in E(\cal Y(\Sigma))$ encodes a symmetric matrix of the form $Z(\alpha S)$, where $\abs{\alpha} = 1$. In order to estimate the right-hand side of \eqref{V_cutoff4}, we therefore need appropriate resolvent bounds on the entries of $Z(\alpha S)$. To that end, we improve the second bound of \eqref{S_basic_bounds} using the following
local decay bound. Recall the definition of $\cal I$ from \eqref{def_I}. 
\begin{lemma} \label{lem: lct}
For all $b \in \N$ we have
\begin{equation*}
(\cal I^{-1} S^b)_{yz} \;\leq\; \frac{C}{M b^{d/2}} + \frac{C}{N}
\end{equation*}
for some constant $C$ depending only on $f$.
\end{lemma}
\begin{proof}
This follows from a standard local central limit theorem; see for instance the proof in \cite[Section 3]{So1}. 
\end{proof}
In particular, for $1 \leq b \leq (L / W)^2$ we have
\begin{equation} \label{lclt}
(S^b)_{yz} \;\leq\; \frac{C}{M b^{d/2}}\,.
\end{equation}
Recalling \eqref{LW_assump} and \eqref{def_R2},  we find from \eqref{lclt} that for $\abs{\alpha} \leq 1$ we have
\begin{equation} \label{Z_no_osc}
\abs{Z(\alpha S)_{yz}} \;\leq\; \frac{C}{M} R_2(M^{-\mu})\,.
\end{equation}

The bound \eqref{Z_no_osc} is sharp if $\alpha = 1$, i.e.\ if the sum in \eqref{def_Z} is not oscillating. If oscillations are present, we get better bounds which we record in the following lemma. It is a special case of  \cite[Proposition  3.5]{EK4}.
\begin{lemma} \label{lem: bounds on S}
Let $S$ be as in \eqref{general S} and $\alpha \in \C$ satisfy $\abs{\alpha} \leq 1$ and $\abs{1 - \alpha} \geq 4 / M + (W/L)^2$.
There exists a constant $C > 0$, depending only on $d$ and the profile function $f$, such that
\begin{equation} \label{bound res 1}
\normbb{\frac{1}{1 - \alpha S}}_{\ell^\infty \to \ell^\infty} \;\leq\; \frac{C \log N}{2 - \abs{1 + \alpha}}\,.
\end{equation}
Under the same assumptions we have
\begin{equation} \label{bound res 2}
\sup_{x,y} \absbb{\pbb{\frac{S}{1 - \alpha S}}_{xy}} \;\leq\; \frac{C}{M} R_2(\abs{1 - \alpha})\,,
\end{equation}
where the constant $C$ depends only on $d$ and $f$.
\end{lemma}

From \eqref{S_basic_bounds}, \eqref{Z_no_osc}, and Lemma \ref{lem: bounds on S} we get
\begin{equation} \label{edge_bounds}
\abs{Z(\alpha S)_{yz}} \;\leq\; \frac{C}{M} \min \hb{R_2(\abs{1 - \alpha}) \,,\, R_2(M^{-\mu})} \,, \qquad
\sum_{z} \abs{Z(\alpha S)_{yz}} \;\leq\; C \min \hBB{\frac{\log N}{2 - \abs{1 + \alpha}} \,,\, M^\mu}\,,
\end{equation}
We apply \eqref{edge_bounds}  to estimating \eqref{V_cutoff4} via the following key estimate. 

\begin{lemma}
Let $v = (v_1, v_2)$ and denote by $A_{i, v} \deq \arcsin(E_i - v_i)$ the value of $A_i$ in the convolution integral \eqref{V_cutoff4}. For small enough $\delta > 0$ and $\abs{v_1}, \abs{v_2} \leq 2 \eta M^{\delta/2}$ (i.e.\ $v_1$ and $v_2$ in the support of the convolution integral \eqref{V_cutoff4}) we have \begin{equation} \label{domestic_bound}
\absb{Z(-\ee^{\pm 2 \ii A_{i,v}} S)_{yz}} \;\leq\; \frac{C}{M}\,, \qquad
\sum_{z} \absb{Z(-\ee^{\pm 2 \ii A_{i,v}} S)_{yz}} \;\leq\; C \log N
\end{equation}
and
\begin{equation} \label{connecting_bound}
\absb{Z(\ee^{\ii (A_{1,v} - A_{2,v})} S)_{yz}} \;\leq\; \frac{C}{M} M^{2 \delta} R_2(\omega + \eta)\,, \qquad
\sum_z \absb{Z(\ee^{\ii (A_{1,v} - A_{2,v})} S)_{yz}} \;\leq\; C M^\mu\,.
\end{equation}
\end{lemma}

\begin{proof}
To prove \eqref{domestic_bound}, we set $\alpha_i  = -\ee^{\pm 2 \ii A_{i,v}}$, in which case an elementary estimate 
yields $2 - \abs{1 + \alpha_i} \geq c$.
Similarly,
we have $\abs{1 - \alpha_i} \geq c$, which yields $R_2(\abs{1 - \alpha_i}) \leq C$.
Now \eqref{domestic_bound} follows from \eqref{edge_bounds} and \eqref{D leq kappa}.

To prove \eqref{connecting_bound}, we set $\alpha = \ee^{\ii (A_{1,v} - A_{2,v})}$. In order to estimate $Z(\alpha S)_{yz}$, we distinguish two cases according to whether  $\eta \leq M^{-\delta} \omega$. Suppose first that $\eta \leq M^{-\delta} \omega$. Then we have  $\abs{1 - \alpha} \asymp \omega (1 + O(\omega)) \geq c\omega$. We therefore find from the first inequality of \eqref{edge_bounds} that
\begin{equation*}
\abs{Z(\alpha S)_{yz}} \;\leq\; \frac{C}{M} R_2(\omega) \;\leq\; \frac{C}{M} R_2(\omega + \eta)\,,
\end{equation*}
where in the second step we used that $\omega + \eta \leq 2 \omega$ and that $R_2$ is monotone decreasing for small enough arguments 
 (see its definition in \eqref{def_R2}).
 On the other hand, if $\omega < M^\delta \eta$ then we get from \eqref{edge_bounds} that
\begin{multline*}
\abs{Z(\alpha S)_{yz}} \;\leq\; \frac{C}{M} R_2(M^{-\mu}) \;\leq\; \frac{C}{M} \pb{(\omega + \eta)  M^{\mu}}^{1/2} \, R_2(\omega + \eta)
\\
\leq\; \frac{C}{M} M^{(\delta + \mu - \rho)/2} R_2(\omega + \eta) \;\leq\; \frac{C}{M} M^{2 \delta} R_2(\omega + \eta)\,,
\end{multline*}
where in the second step we used that $(\omega + \eta) M^\mu \geq 1$, and  in the last step that $\mu - \rho < 3 \delta$. 
Putting both cases together we get \eqref{connecting_bound}.
\end{proof}

Next, we plug the estimates \eqref{domestic_bound} and \eqref{connecting_bound} into \eqref{V_cutoff4} and sum over $\f y$; we use \eqref{domestic_bound} for $e \in E_i(\cal Y(\Sigma))$ with $i = 1,2$, and \eqref{connecting_bound} for $e \in E_c(\cal Y(\Sigma))$. 
We perform the summation over $\f y$ as in Section \ref{sec:large_Sigma}: by choosing an arbitrary spanning tree $\cal T$ of $\cal Y(\Sigma)$ along with an arbitrary root of $\cal T$. In the summation over $\f y$ on the right-hand side of \eqref{V_cutoff4}, each edge $e \in E(\cal Y(\Sigma))$ encodes a matrix entry that we estimate as follows. For $e \in E_d(\cal Y(\Sigma)) \setminus E(\cal T)$ we use the first estimate of \eqref{domestic_bound}, for $e \in E_d(\cal Y(\Sigma)) \cap E(\cal T)$ the second estimate of \eqref{domestic_bound}, for $e \in E_c(\cal Y(\Sigma)) \setminus E(\cal T)$ the first estimate of \eqref{connecting_bound}, and for $e \in E_c(\cal Y(\Sigma)) \cap E(\cal T)$ the second estimate of \eqref{connecting_bound}. The result is
\begin{align*}
\abs{\cal V'_0(\Sigma)} &\;\leq\; C^{\abs{\Sigma}} N M^{-\abs{E_d(\cal Y(\Sigma)) \setminus E(\cal T)}} \p{\log N}^{\abs{E_d(\cal Y(\Sigma)) \cap E(\cal T)}}
\\
&\qquad \times \pbb{\frac{M^{2 \delta} R_2(\omega + \eta)}{M}}^{\abs{E_c(\cal Y(\Sigma)) \setminus E(\cal T)}} M^{\mu \abs{E_c(\cal Y(\Sigma)) \cap E(\cal T)}}
\\
&\;\leq\; \frac{N}{M} \, \frac{(\log N)^{\abs{\Sigma}}}{M^{\abs{\Sigma} - \abs{Q(\Sigma)}}} \,\pb{ M^{2 \delta} R_2(\omega + \eta)}^{\abs{E_c(\cal Y(\Sigma)) \setminus E(\cal T)}} \, M^{\mu \abs{E_c(\cal Y(\Sigma)) \cap E(\cal T)}}\,,
\end{align*}
where we used that $\abs{E(\cal Y(\Sigma))} = \abs{\Sigma}$ and $\abs{E(\cal T)} = \abs{Q(\Sigma)} - 1$. As before, the factor $N$ arises from the summation over the label of $\f y$ associated with the root of $\cal T$.

Next, we remark that the above proof may be repeated verbatim for the other error term, $\cal V_1'(\Sigma)$. This case is in fact easier: since  $\N^{E(\cal Y(\Sigma))} \setminus B(\Sigma)$  is a finite set  (see Lemma \ref{lem: skeletons}),  we do not have to exploit the cancellations from the summation over $\f b$. Repeating the above argument for $\cal V_1'(\Sigma)$, with the right-hand sides of the corresponding estimates from \eqref{domestic_bound} and \eqref{connecting_bound} replaced with $C/M$, $C$, $C/M$, and $C$ respectively, we find
\begin{equation} \label{V'estimate1}
\abs{\cal V'(\Sigma)} \;\leq\; \cal R(\Sigma) \;\deq\; \frac{N}{M} \, \frac{ M^{3 \delta \abs{\Sigma}}}{M^{\abs{\Sigma} - \abs{Q(\Sigma)}}} \,  R_2(\omega + \eta)^{\abs{E_c(\cal Y(\Sigma)) \setminus E(\cal T)}} \, M^{\mu \abs{E_c(\cal Y(\Sigma)) \cap E(\cal T)}}\,.
\end{equation}
In order to show that $\cal R(\Sigma)$ is small enough, we shall use a graph-theoretic argument to derive appropriate bounds on the exponents. It relies on the following further partition of the set $\Sigma_d$ according to whether a bridge touches both endpoints (white vertices) of a chain.

\begin{definition} \label{def_S^1}
We partition $\Sigma_d = \Sigma_d^0 \sqcup \Sigma_d^1$, where
\begin{equation*}
\Sigma_d^1 \;\deq\; \hb{\sigma \in \Sigma_d \col \text{$\sigma$ touches $a(\cal C_i)$ and $b(\cal C_i)$ for some $i = 1,2$}}\,.
\end{equation*}
We also use $E_d^0(\cal Y(\Sigma))$ and $E_d^1(\cal Y(\Sigma))$ to denote the corresponding disjoint subsets of $E_d(\cal Y(\Sigma))$.
\end{definition}
Note that $\Sigma_d^1$ may contain at most two bridges: one only touching 
the white vertices of $\cal C_1$ and one only touching the white vertices of $\cal C_2$. See Figure \ref{fig:bridges} for an illustration of these three types of bridges.
\begin{figure}[ht!]
\begin{center}
\includegraphics{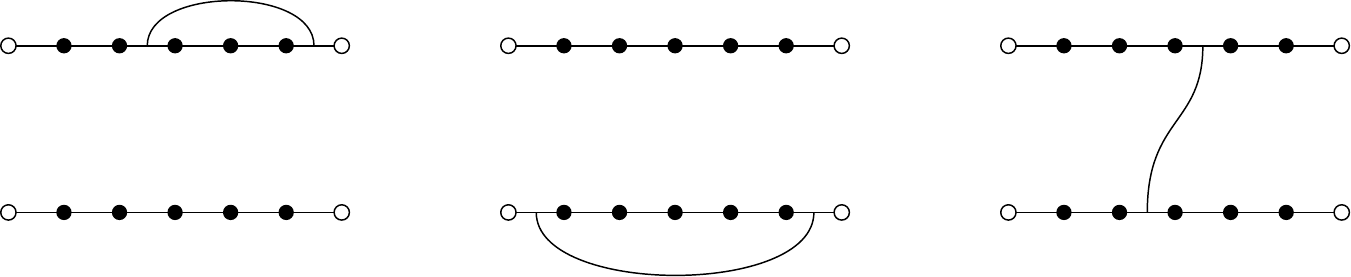}
\end{center}
\caption{A bridge in $\Sigma_d^0$ (left), $\Sigma_d^1$ (centre), and $\Sigma_c$ (right). \label{fig:bridges}}
\end{figure}

For the following counting arguments, for definiteness it will be convenient to assume that $\Sigma^1_d = \emptyset$. Hence, we first show that skeleton
pairings with $\Sigma^1_d \ne \emptyset$ can be easily estimated by
those with $\Sigma^1_d = \emptyset$, at the expense of an unimportant factor.
The following lemma states this fact precisely.
Let
\begin{equation*}
\ol{\fra S}^{\leq} \;\deq\; \hb{\Sigma \in \fra S^\leq \col \Sigma_d^1 = \emptyset}\,.
\end{equation*}
\begin{lemma} \label{lem:two_e_edges}
For each $\Sigma \in \fra S^\leq$ there exists a $\ol \Sigma \in \ol{\fra S}^\leq$ such that and $\cal R(\Sigma) \leq (\log N)^2 \cal R(\ol \Sigma)$.
\end{lemma}
\begin{proof}
The operation $\Sigma \mapsto \ol \Sigma$ amounts to simply removing all bridges of $\Sigma_d^1$ from $\Sigma$. Instead of a formal definition, we refer to Figure \ref{fig:SSbar} for a graphical depiction of this operation.
\begin{figure}[ht!]
\begin{center}
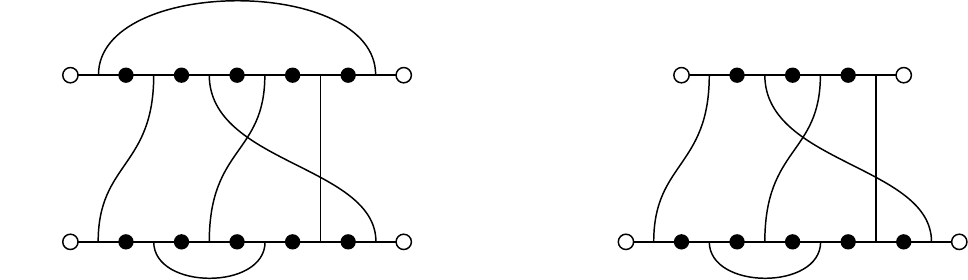
\end{center}
\caption{The operation $\Sigma \mapsto \ol \Sigma$. \label{fig:SSbar}}
\end{figure}
By definition of $Q(\cdot)$, we find that the operation $\Sigma \mapsto \ol \Sigma$ amounts to removing  any of the two vertices $\{a(\cal C_1), b(\cal C_1)\}$ and $\{a (\cal C_2), b(\cal C_2)\}$ that belongs to $Q(\Sigma)$. 
This results in a removal of the corresponding number of leaves from the spanning tree $\cal T$. (The removed bridges always correspond to leaves
in $\cal T$. In particular, $\abs{E_c(\cal Y(\Sigma)) \setminus E(\cal T)}$
and  $\abs{E_c(\cal Y(\Sigma)) \cap E(\cal T)}$
are remain unchanged by this removal.) Note that if $\Sigma \in \fra S^{\leq}$ then $\ol \Sigma \in \ol {\fra S}^\leq$, since by construction if $\Sigma \notin \{D_1, \dots, D_8\}$ then $\ol \Sigma \notin \{D_1, \dots, D_8\}$. The claim now follows easily from  the bound \eqref{V'estimate1} with argument $\Sigma$, as well as the observations  that $\abs{\Sigma} - \abs{Q(\Sigma)} = \abs{\ol \Sigma} - \abs{Q(\ol \Sigma)}$, that $\abs{\Sigma_d} \leq \abs{\ol \Sigma_d} + 2$, that $\abs{\Sigma} \leq \abs{\ol \Sigma} + 2$, and that the two last exponents on the right-hand side of \eqref{V'estimate1} are the same for $\Sigma$ and $\ol \Sigma$. 
\end{proof}

By Lemma \ref{lem:two_e_edges}, it suffices to estimate $\cal R(\Sigma)$ for $\Sigma \in \ol {\fra S}^\leq$. For $\Sigma \in \ol {\fra S}^\leq$ we have $\Sigma_d^0 = \Sigma_d$. Moreover, if there is a bridge touching $a(\cal C_1)$ and $a(\cal C_2)$ as well as a bridge touching $b(\cal C_1)$ and $b(\cal C_2)$, we find that all four white vertices constitute a single block of $Q(\Sigma)$. Otherwise, since $\Sigma_d^1 = \emptyset$, every block of $Q(\Sigma)$ contains a black vertex, so that $Q_b(\Sigma) = Q(\Sigma)$, where $Q_b(\Sigma)$ was defined in \eqref{def_Qb}. Either way, recalling Lemma \ref{lem:2/3_rule}, we conclude for $\Sigma \in \ol{\fra S}^\leq$ and $q \in Q(\Sigma)$ that
\begin{equation} \label{2/3-Sbar}
\abs{q} \;\geq\; 3\,.
\end{equation}

In order to complete the estimate of \eqref{V'estimate1}, and hence the proof of Proposition \ref{prop:small_Sigma}, we shall have to distinguish between the case where $\abs{q} = 3$ for all $q \in Q(\Sigma)$ and the case where there exists a $q \in Q(\Sigma)$ with $\abs{q} > 3$.

\begin{lemma} \label{lem:saturated2/3}
Suppose that $\Sigma \in \ol{\fra S}^\leq$ and $\abs{q} = 3$ for all $q \in Q(\Sigma)$.
Then
\begin{equation} \label{V'saturated}
\cal R(\Sigma) \;\leq\; \frac{N}{M} R_2(\omega + \eta) M^{3 \mu - 1}  M^{C \delta }\,.
\end{equation}
\end{lemma}

\begin{lemma} \label{lem:non_saturated2/3}
Suppose that $\Sigma \in \ol{\fra S}^\leq$ and there exists a $\ol q \in Q(\Sigma)$ with $\abs{\ol q} > 3$. 
Then
\begin{equation} \label{V'nonsaturated}
\cal R(\Sigma) \;\leq\; \frac{N}{M} R_2(\omega + \eta) M^{\mu - 1/3}  M^{C \delta}\,.
\end{equation}
\end{lemma}

\begin{proof}[Proof of Lemma \ref{lem:saturated2/3}]
We first claim that there is at least
one domestic bridge, i.e.\ that $\Sigma_d \neq \emptyset$.

Clearly, $\abs{V(\Sigma)}$ is even. Recall that $\bigcup_{q \in Q(\Sigma)} q = V(\Sigma)$. Since each block of $Q(\Sigma)$ has size $3$, we conclude that $\abs{V(\Sigma)}$ is multiple of $3$, and hence of $6$. A simple exhaustion of all possible pairings $\Sigma \in \ol{\fra S}^\leq$ that saturate the first inequality in \eqref{3Qb_estimate} shows that there is no such $\Sigma$ satisfying $\absb{\bigcup_{q \in Q(\Sigma)} q} = 6$. (In fact, any connected pairing with at most six vertices is a dumbbell pairing,
which are excluded by the definition \eqref{frastar} of $\fra S^*$.)
 Hence we find that $\absb{\bigcup_{q \in Q(\Sigma)} q} \geq 12$, so that $\abs{Q(\Sigma)} \geq 4$.

Next, note that $Q(\Sigma)$ contains  at most  two blocks $q$ that contain white vertices of $\Sigma$,  since $\Sigma$ contains four white vertices, and, for each $i \in \{1,2\}$, those of $\cal C_i$ are in the same block of $Q(\Sigma)$. (Recall items (ii) and (iii) after \eqref{indicator function behind Q}). 
Since $\abs{Q(\Sigma)} \geq 4$, we find that there is a block $q \in Q(\Sigma)$ that contains only black vertices. Let $\Sigma(q)$ be the set of bridges of $\Sigma$ touching a vertex of $q$; see Figure \ref{fig:three-lump}. By definition of $Q(\Sigma)$, we have $\abs{\Sigma(q)} = 3$. Now if all vertices of $q$ belong to the same connected component of $\cal C$, then $\Sigma(q) \subset \Sigma_d$. Otherwise, let $q = \{i,j,k\}$ with $j$ and $k$ belonging to the same connected component of $\cal C$. Then both bridges touching $i$ are in $\Sigma_c$; the remaining bridge of $\Sigma(q)$ must touch both $j$ and $k$, and is therefore in $\Sigma_d$. Either way, we find that $\Sigma_d \neq \emptyset$, as claimed above.
\begin{figure}[ht!]
\begin{center}
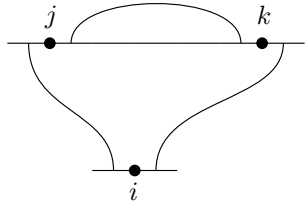
\end{center}
\caption{A block $q = \{i,j,k\} \in Q(\Sigma)$ along with the three bridges of $\Sigma(q)$. We do not draw the other vertices or bridges. \label{fig:three-lump}}
\end{figure}

For the following, abbreviate $s_l \deq \abs{E_c(\cal Y(\Sigma)) \setminus E(\cal T)}$ and $s_t \deq \abs{E_c(\cal Y(\Sigma)) \cap E(\cal T)}$.
From the saturated inequality  \eqref{3Qb_estimate} we find
\begin{equation} \label{QSS}
\abs{Q(\Sigma)} \;=\; \frac{2 \abs{\Sigma}}{3}  + \frac{2}{3}\,.
\end{equation}
Plugging this into \eqref{V'estimate1} yields
\begin{equation} \label{R_Sigma_bound}
\cal R(\Sigma) \;=\; \frac{N}{M} \, M^{3 \delta \abs{\Sigma}}  M^{2/3 - \abs{\Sigma}/3} \, R_2(\omega + \eta)^{s_l} \, M^{\mu s_t}\,.
\end{equation}

Recall that $\abs{\Sigma_d} \geq 1$ and $\abs{\Sigma_d} + s_l + s_t = \abs{\Sigma}$. Moreover, $s_t \leq \abs{E(\cal Y(\Sigma)) \cap E(\cal T)} = \abs{Q(\Sigma)} - 1$. Since $R_2(\omega + \eta)  \leq M^\mu$, we conclude
\begin{multline*}
R_2(\omega + \eta)^{s_l} M^{\mu s_t} \;\leq\;
R_2(\omega + \eta)^{s_l + s_t - \abs{Q(\Sigma)} + 1} M^{\mu (\abs{Q(\Sigma)} - 1)}
\\
\;=\; R_2(\omega + \eta)^{\abs{\Sigma} - \abs{\Sigma_d} - \abs{Q(\Sigma)} + 1} M^{\mu (\abs{Q(\Sigma)} - 1)}
\;\leq\; R_2(\omega + \eta)^{\abs{\Sigma} - \abs{Q(\Sigma)}} M^{\mu (\abs{Q(\Sigma)} - 1)}\,.
\end{multline*}

Thus we get
\begin{align*}
\cal R(\Sigma) &\;\leq\; \frac{N}{M} \, M^{3 \delta \abs{\Sigma}} M^{2/3 - \abs{\Sigma}/3} R_2(\omega + \eta)^{\abs{\Sigma} - \abs{Q(\Sigma)}} \, M^{\mu (\abs{Q(\Sigma)} - 1)}
\\
&\;=\; \frac{N M^{3 \delta}  M^{1/3}}{M} \, \pb{M^{-1/3 + 3 \delta} R_2(\omega + \eta)}^{\abs{\Sigma} - \abs{Q(\Sigma)}} \, (M^{\mu - 1/3 + 3 \delta})^{\abs{Q(\Sigma)} - 1}
\\
&\;\leq\; \frac{N}{M} R_2(\omega + \eta) M^{3 \mu - 1}  M^{15 \delta}\,,
\end{align*}
where in the last step we used \eqref{QSS} to get $\abs{\Sigma} - \abs{Q(\Sigma)} = \abs{\Sigma}/3 - 2/3 \geq 1$, as well as $\abs{Q(\Sigma)} \geq 4$ and $\mu <1/3$. Here we chose $\delta > 0$ in Proposition \ref{prop: expansion with trunction} small enough that $\mu < 1/3 - 3\delta$. 
We also used that 
$M^{-1/3 + 3 \delta} R_2(\omega + \eta) \leq 1$. This concludes the proof.
\end{proof}

\begin{proof}[Proof of Lemma \ref{lem:non_saturated2/3}]
Since $\abs{\ol q} \geq 4$ and all other blocks of $Q(\Sigma)$ have size at least 3 by \eqref{2/3-Sbar}, we find that
\begin{equation} \label{QSS_4}
\abs{Q(\Sigma)} \;\leq\; 1 + \frac{2 \abs{\Sigma} + 2 - \abs{\ol q}}{3} \;\leq\; \frac{2 \abs{\Sigma}}{3} + \frac{1}{3}\,,
\end{equation}
where $2 \abs{\Sigma} + 2 - \abs{\ol q}$ is the number of vertices of $\Sigma$ not in $\ol q$. Note the improvement of \eqref{QSS_4} over \eqref{QSS}. Using the notation of the proof of Lemma \ref{lem:saturated2/3}, we get from \eqref{V'estimate1}, in analogy to \eqref{R_Sigma_bound},
\begin{equation*}
\cal R(\Sigma) \;\leq\; \frac{N}{M} \, M^{3 \delta \abs{\Sigma}}  M^{1/3 - \abs{\Sigma}/3} \,  R_2(\omega + \eta)^{s_l} \, M^{\mu s_t}\,.
\end{equation*}
Now we proceed as in the proof of Lemma \ref{lem:saturated2/3}, using $\abs{\Sigma_d} \geq 0$, $\abs{\Sigma_d} + s_l + s_t = \abs{\Sigma}$, and $s_t \leq \abs{Q(\Sigma)} - 1$. We get
\begin{align*}
\cal R(\Sigma) &\;\leq\; \frac{N M^{1/3}}{M}  \, \pb{M^{-1/3 + 3 \delta} R_2(\omega + \eta)}^{\abs{\Sigma} - \abs{Q(\Sigma)} + 1} \, (M^{\mu- 1/3 + 3 \delta})^{\abs{Q(\Sigma)} - 1}
\\
&\;\leq\; \frac{N}{M} R_2(\omega + \eta) M^{\mu - 1/3}   M^{6 \delta} \,.
\end{align*}
In the last step we used that $\abs{\Sigma} - \abs{Q(\Sigma)} \geq 0$, which follows from \eqref{QSS_4} and from
 $\abs{\Sigma} \geq 3$ for $\Sigma \in \ol {\fra S}^{\leq}$, and that $\abs{Q(\Sigma)} \geq 2$. (In fact, since $\Sigma \notin \{D_1, \dots, D_8\}$ one may easily check that $\abs{Q(\Sigma)} \geq 3$.)
This concludes the proof.
\end{proof}

From Lemmas \ref{lem:saturated2/3}, \ref{lem:non_saturated2/3}, and \ref{lem:two_e_edges}, we conclude that for all $\Sigma \in \fra S^\leq$ we have
\begin{equation} \label{V'_leq_R}
\abs{\cal V'(\Sigma)} \;\leq\; \cal R(\Sigma) \;\leq\; \frac{N}{M} R_2(\omega + \eta) M^{3 \mu - 1}  M^{4 \delta K}\,.
\end{equation}
In order to conclude the proof of Proposition \ref{prop:small_Sigma}, we need an analogous estimate of $\cal V''(\Sigma)$. This may be obtained by repeating the above argument almost verbatim; the only nontrivial difference is that, on the right-hand side of \eqref{V_cutoff4}, the factor $Z(\ee^{\ii (A_1 - A_2)} S)_{y_e}$ associated with the edge $e \in E_c(\cal Y(\Sigma))$ is replaced with $Z(\ee^{\ii (A_1 + A_2)} S)_{y_e}$. 
Since $\abs{1 - \ee^{\ii (A_1 + A_2)}} \geq c$ on the support of the convolution integral, 
 we replace \eqref{connecting_bound} with
\begin{equation*}
\absb{Z(-\ee^{\ii (A_1 + A_2)} S)_{yz}} \;\leq\; \frac{C}{M} \;\leq\; \frac{C}{M} R_2(\omega + \eta)  \,, \qquad
\sum_z \absb{Z(\ee^{\ii (A_1 - A_2)} S)_{yz}} \;\leq\; C M^\mu\,.
\end{equation*}
Thus we find, for any $\Sigma \in \fra S^\leq$, that
\begin{equation} \label{V''_leq_R}
\abs{\cal V''(\Sigma)} \;\leq\; \cal R(\Sigma)\,.
\end{equation}
Hence Proposition \ref{prop:small_Sigma} follows from \eqref{V'_leq_R}, \eqref{V''_leq_R}, and the observation that $\fra S^\leq$ is a finite set that is independent of $N$.
This concludes the proof of Proposition \ref{prop:small_Sigma}.

We conclude this subsection with an analogue of Proposition \ref{prop:small_Sigma} in the case \textbf{(C1)}. Its proof follows along the same lines as that of Proposition \ref{prop:small_Sigma}, and is omitted.
\begin{proposition} \label{prop:small_Sigma_c1}
Suppose that $\phi_1$ and $\phi_2$ satisfy \eqref{Cauchy}. Suppose moreover that \eqref{D leq kappa} holds for some small enough $c_* > 0$. Then for any fixed $K \in \N$ we have \eqref{sum_V_Sigma_small}.
\end{proposition}

For future reference we emphasize that the only information about the matrix entries of $Z(\cdot)$ that is required for the estimate \eqref{sum_V_Sigma_small} to hold is \eqref{domestic_bound} and \eqref{connecting_bound}. Thus, the conclusion of the above argument may be formulated in the following more general form.

\begin{lemma}
Let $\Sigma \notin \{D_1, \dots D_8\}$, and suppose that we have a family of matrices $\cal Z(\sigma, E_1, E_2, L) \equiv Z(\sigma)$ parametrized by $\sigma \in \Sigma$ satisfying
\begin{equation} \label{domestic_bound2}
\abs{\cal Z_{xy}(\sigma)} \;\leq\; \frac{C}{M}\,, \qquad
\sum_{y} \abs{\cal Z_{xy}(\sigma)} \;\leq\; C \log N
\end{equation}
for $\sigma \in \Sigma_1 \cup \Sigma_2$ and
\begin{equation} \label{connecting_bound2}
\abs{\cal Z_{xy}(\sigma)} \;\leq\; \frac{C}{M} M^{2 \delta} R_2(\omega + \eta)\,, \qquad
\sum_y \abs{\cal Z_{xy}(\sigma)} \;\leq\; C M^\mu\,
\end{equation}
for $\sigma \in \Sigma_c$.

Then for small enough $\delta$ there exists a $c_0 > 0$ such that
\begin{multline*}
\sum_{\f y \in \bb T^{Q(\Sigma)}} \sum_{\f x \in \bb T^{V(\Sigma)}} \pBB{\prod_{q \in Q(\Sigma)} \prod_{i \in q} \ind{x_i = y_q}} \pBB{\prod_{\{e, e'\} \in \Sigma} \cal Z_{x_e}(\{e,e'\})}
\\
=\;
\sum_{\f x \in \bb T^{V(\Sigma)}} I_0(\f x)
\pBB{\prod_{\{e,e'\} \in \Sigma} J_{\{e,e'\}} \cal Z_{x_e}(\{e,e'\})} \;\leq\; \frac{C_\Sigma N}{M} R_2(\omega + \eta)  M^{-c_0}\,.
\end{multline*}
\end{lemma}

\subsection{Conclusion of the proof of Proposition \ref{prop: main} and Theorems \ref{thm: main result}--\ref{thm: Theta 1}} \label{sec:conclusion}

We may now conclude the proof of Proposition \ref{prop: main}. 
As indicated before, the error terms $\cal E$ resulting from 
the simplifications {\bf (S1)}--{\bf (S3)} are small; the precise 
statement is the following proposition that is proved in \cite{EK4}.
\begin{proposition} \label{prop:calE}
The error term $\cal E$ in \eqref{F in terms of skeletons} arising from the simplifications {\bf (S1)}--{\bf (S3)} satisfies
\begin{equation}
\abs{\cal E} \;\leq\; \frac{CN}{M} M^{-c_0} R_2(\omega + \eta)\,,
\end{equation}
for some constant $c_0 > 0$.
\end{proposition}
\begin{proof}
This is an immediate consequence of Propositions 4.5, 4.6, and 4.15 
 in \cite{EK4}.
\end{proof}
Combining Propositions \ref{prop:large_Sigma}, \ref{prop:small_Sigma}, \ref{prop:small_Sigma_c1}, and \ref{prop:calE} yields
\begin{equation*}
\wt F^\eta(E_1, E_2) \;=\; \cal V_{\rm{main}} + \frac{N}{M} \pbb{O \pB{M^{-1} + M^{-c_0} R_2(\omega + \eta) } + O_q(N M^{-q})}\,.
\end{equation*}
Together with Proposition \ref{prop: leading term}, this concludes the proof of Proposition \ref{prop: main}. 

 Using \eqref{F - wt F}, \eqref{LW_assump}, and Proposition \ref{prop: main} we therefore get, for $H$ as in Section \ref{sec: setup},
\begin{equation} \label{F_eta_main}
F^\eta(E_1, E_2) \;=\; \cal V_{\rm{main}} + \frac{N}{M} \pbb{O \pB{M^{-1} + M^{-c_0} R_2(\omega + \eta) } + O_q(N M^{-q})}\,.
\end{equation}

In order to compute the left-hand side of \eqref{EY_result}, and hence conclude the proof of Theorems \ref{thm: main result}--\ref{thm: Theta 1}, we need to control the denominator of \eqref{EY_result} using the following result.

\begin{lemma} \label{lem:EY}
For $E \in [-1+ \kappa, 1 - \kappa]$ we have
\begin{equation}\label{semicir}
\E \, Y^\eta_\phi(E) \;=\; 4 \sqrt{1 - E^2} + O(\eta) \; = \; 2\pi\nu(E) + O(\eta) \,.
\end{equation}
\end{lemma}

\begin{proof}
In the case \textbf{(C1)} we have
\begin{equation*}
\E \, Y^\eta_\phi(E) \;=\; \E \, \frac{1}{N} \tr \phi^\eta(H/2 - E) \;=\; \E \frac{1}{N} \im \tr \frac{4}{H - 2 (E + \ii \eta)} \;=\; 4 \im m(2E + 2 \ii \eta) + O(M^{-2/3 + c})
\end{equation*}
for any $c > 0$. Here in the last step we introduced the Stieltjes transform of the semicircle law, $m(z)$, and invoked \cite[Theorem 2.3 and Equation (7.6)]{EKYY4}. The claim then follows from the estimate
$4 \im m(2 E + 2 \ii \eta) = 4 \sqrt{1 - E^2} + O(\eta)$, 
which itself follows from \cite[Equations (3.3) and (3.5)]{EKYY3}.

In the case \textbf{(C2)}, we first split $\phi^\eta = \phi^{\leq, \eta} + \phi^{>, \eta}$ as in \eqref{phi_cutoff1}. The contribution of $\phi^{>, \eta}$ is small by the strong decay of $\phi$. The error in the main term,
\begin{equation*}
\E \, \frac{1}{N} \im \tr \phi^{\leq,\eta}(H/2 - E) - \frac{1}{2\pi}\int_{-2}^2 \dd x \, \sqrt{4 - x^2} \, \phi^{\leq, \eta}(x/2 - E)\,,
\end{equation*}
may be estimated using \cite[Theorem 2.3]{EKYY4} and Helffer-Sj\"ostrand functional calculus, as in e.g.\ \cite[Section 7.1]{EKYY4}; we omit the details. Then the claim follows from $\frac{1}{2\pi}\int_{-2}^2 \dd x \, \sqrt{4 - x^2} \, \phi^{\leq, \eta}(x/2 - E) = 4 \sqrt{1 - E^2} + O(\eta)$.
\end{proof}

Now we define
\begin{equation} \label{def:theta}
\Theta_{\phi_1,\phi_2}^\eta(E_1, E_2) \;\deq\; \frac{(LW)^d}{N^2} \, \frac{\cal V_{\mathrm{main}}}{\E Y^\eta_{\phi_1}(E_1) \E Y^\eta_{\phi_2}(E_2)}\,.
\end{equation}
Then Theorems \ref{thm: Theta 2} and \ref{thm: Theta 1} follow from Lemma \ref{lem:EY} and \eqref{F_eta_main}, 
recalling \eqref{Dconst} and \eqref{Qconst}. 
 Moreover, Theorem \ref{thm: main result} follows from \eqref{ThetatoF} and Lemma \ref{lem:EY}.
This concludes the proof of Theorems \ref{thm: main result}--\ref{thm: Theta 1} under the simplifications {\bf (S1)}--{\bf (S3)}.

\section{The real symmetric case ($\beta = 1$)} \label{sec:sym}

In this section we explain the changes needed to the arguments of Section \ref{sec: main argument} to prove Theorems \ref{thm: main result}, \ref{thm: Theta 2}, and \ref{thm: Theta 1} for $\beta = 1$ instead of $\beta = 2$. 
The  difference is that for $\beta = 1$ we have $\E H_{xy}^2 = S_{xy}$, while for $\beta = 2$ we have $\E H_{xy}^2 = 0$ (in addition to $\E H_{xy} H_{yx} = \E \abs{H_{xy}}^2 = S_{xy}$, which is valid in both cases).
 This leads to additional terms for $\beta = 1$, which may be conveniently tracked in our graphical notation by introducing \emph{twisted bridges}, in analogy to Section 9 of \cite{EK1}. As it turns out, allowing twisted bridges results in eight new dumbbell skeletons, called $\wt D_1, \dots, \wt D_8$ below, each of which has the same value $\cal V(\cdot)$ as its counterpart without a tilde. 
Hence, for $\beta = 1$  the leading term  is simply twice the leading term of $\beta = 2$, 
 which accounts for the trivial prefactor $2/\beta$ in the
final formulas.
Any other skeleton may be estimated by a trivial modification of the argument from Sections \ref{sec:large_Sigma}--\ref{proof:small_Sigma}.
As in Section \ref{sec: main argument}, we make the simplifications {\bf(S1)}--{\bf(S3)}, and do not deal with the errors terms $\cal E$ resulting from them. They are handled in \cite{EK4}.

We now give a more precise account of the proof for $\beta = 1$. We start from \eqref{sum over pairings after simplification}, which remains unchanged. Since $\E H_{xy} H_{xy} = \E H_{xy} H_{yx} = S_{xy}$, \eqref{sum over pairings after simplification 2} holds for $\beta = 1$ without the indicator function $\ind{x_e \neq x_{e'}}$ that was present for $\beta = 2$. Hence \eqref{sum over pairings after simplification 3} also holds without the indicator function $\ind{x_e \neq x_{e'}}$. We now write
\begin{equation*}
\ind{[x_e] = [x_{e'}]} \;=\; \ind{[x_e] = [x_{e'}]} \ind{x_e \neq x_{e'}} + \ind{x_e = x_{e'}} \;\eqd\; J_{\{e,e'\}}(\f x) + \wt J_{\{e,e'\}}(\f x)\,,
\end{equation*}
in self-explanatory notation (recall that $J_{\{e,e'\}}(\f x)$ was already defined in \eqref{def_I_sigma}). Thus \eqref{expansion in terms of P S} becomes
\begin{equation} \label{expansion for beta1}
\avgb{\tr H^{(n_1)}\,; \tr H^{(n_2)}} \;=\; \sum_{\Pi \in \fra M_c(E(\cal C))} \sum_{\f x} I(\f x) \pBB{\prod_{\{e, e'\} \in \Pi} \pb{J_{\{e,e'\}}(\f x) + \wt J_{\{e,e'\}}(\f x)} \, S_{x_e}} + \cal E\,.
\end{equation}
Multiplying out the parentheses in \eqref{expansion for beta1} yields $2^{\abs{\Pi}}$ terms, each of which is characterized by the set of bridges of $\Pi$ associated with a factor $J$; the other bridges are associated with a factor $\wt J$. We call the former \emph{straight bridges} and the latter \emph{twisted bridges}. 
This terminology originates from the fact that a twisted bridge forces the labels of the adjacent vertices to
coincide on opposite sides of the bridge; see Figure \ref{fig:twisted bridges} for an illustration. 
\begin{figure}[ht!]
\begin{center}
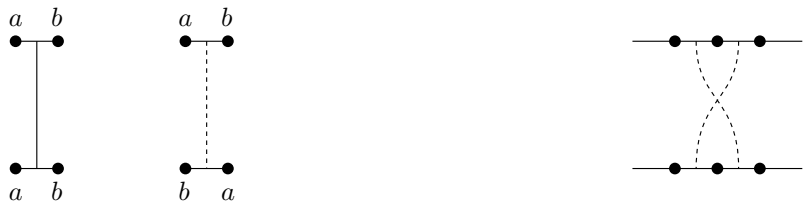
\end{center}
\caption{Left picture: a straight bridge (left) and a twisted bridge (right); labels with the same name are forced to coincide by the bridge. Right picture:
two antiparallel twisted bridges, which form an antiladder of size two.
\label{fig:twisted bridges}} 
\end{figure}
More formally, we assign to each bridge of $\Pi$ a binary tag, \emph{straight} or \emph{twisted}. We represent straight bridges (as before) by solid lines and twisted bridges by dashed lines.

Next, we extend the definition of skeletons from Section \ref{sec_42} to pairings containing twisted bridges. Recall that the key observation behind the definition of a skeleton was that parallel straight bridges yield a large contribution but a small combinatorial complexity. Now \emph{antiparallel} twisted bridges behave analogously, whereby two bridges $\{e_1, e_1'\}$ and $\{e_2, e_2'\}$ are \emph{antiparallel} if $b(e_1) = a(e_2)$ and $b(e_1') = a(e_2')$.
(Recall that they are parallel if $b(e_1) = a(e_2)$ and $b(e_2') = a(e_1')$.) See Figure \ref{fig:twisted bridges} for an illustration. An \emph{antiladder} is a sequence of bridges such that two consecutive bridges are antiparallel. All of Section \ref{sec_41}, in particular the partition $Q(\Pi)$, may now be taken over with trivial modifications.

As in Section \ref{sec_42}, to each tagged pairing $\Pi$ we assign a tagged skeleton $\Sigma$ with associated multiplicities $\f b$. The skeleton $\Sigma$ is obtained from $\Pi$ by successively collapsing \emph{parallel straight bridges} and \emph{antiparallel twisted bridges} until none remains. 
 Parallel twisted bridges and antiparallel straight bridges remain unaltered. The skeleton $\Sigma$ inherits the tagging of its bridges in the natural way: two parallel straight bridges are collapsed into a single straight bridge, and two antiparallel twisted bridges are collapsed into a single twisted bridge. See Figure \ref{fig: twisted skeleton} for an illustration.
\begin{figure}[ht!]
\begin{center}
\includegraphics{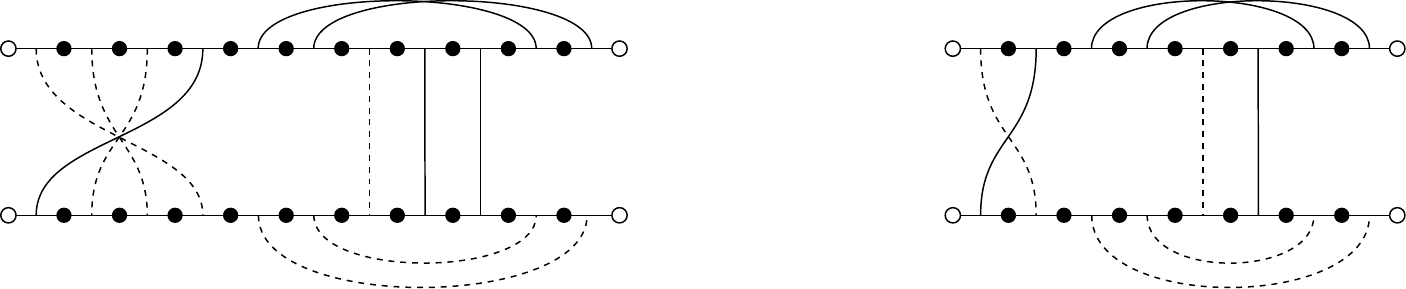}
\end{center}
\caption{A tagged pairing (left) and its tagged skeleton (right). \label{fig: twisted skeleton}} 
\end{figure}
We take over all notions from Section \ref{sec_42}, such as $\cal V(\cdot)$, with the appropriate straightforward modifications for tagged skeletons.

Allowing twisted bridges leads to a further eight skeleton graphs, which we denote by $\wt D_1, \dots, \wt D_8$, 
whose contribution is of leading order. They are the same graphs as $D_1, \dots, D_8$
from Figure \ref{fig: dumbbell}, except that the (one or two) vertical antiparallel straight bridges 
(depicted by solid lines) are replaced with the same number of vertical parallel twisted bridges (depicted by dashed lines).
We use the notations
\begin{equation*}
\cal V_{\rm{main}} \;\deq\; \sum_{i = 1}^8 \cal V(D_i) \,, \qquad \wt {\cal V}_{\rm{main}} \;\deq\; \sum_{i = 1}^8 \cal V(\wt D_i)\,.
\end{equation*}
We record the following simple result, whose proof is immediate.
\begin{lemma}
If $\beta = 1$ then for $i = 1, \dots, 8$ we have $\cal V(D_i) = \cal V(\wt D_i)$.
\end{lemma}
For $\beta = 1$ we may therefore write
$\cal V_{\rm{main}} + \wt{\cal V}_{\rm{main}} = 2 \cal V_{\rm{main}}$. 
Thus, the main term for $\beta = 1$ is simply twice the main term for $\beta = 2$.

What remains is the estimate of $\cal V(\Sigma)$ for $\Sigma \notin \{D_1, \dots, D_8, \wt D_1, \dots, \wt D_8\}$. We proceed exactly as in Sections \ref{sec:large_Sigma}--\ref{proof:small_Sigma}. The key observation is that the 2/3-rule from Lemma \ref{lem:2/3_rule} remains true thanks to the definition of skeletons. When estimating the large skeletons (without making use of oscillations) in Section \ref{sec:large_Sigma}, we get an extra factor $2^m$ to the left-hand side of \eqref{comb_bound} arising from the sum over all possible taggings of a skeleton; this factor is clearly immaterial. Finally, the argument of Sections \ref{sec:small_skeletons} and \ref{proof:small_Sigma} may be taken over with merely cosmetic changes.  The set $\Sigma_d^1$ from Definition \ref{def_S^1} remains unchanged, and in particular only contains straight bridges. Note that the basic graph-theoretic argument from Lemmas \ref{lem:saturated2/3} and \ref{lem:non_saturated2/3} remains unchanged. In particular, exactly as in the proof of Lemma \ref{lem:saturated2/3}, if all blocks of $Q(\Sigma)$ have size three and $\Sigma$ is not a dumbbell skeleton, then $\Sigma$ contains a domestic bridge (which may be straight or twisted). 
 This concludes the proof of Theorems \ref{thm: main result}--\ref{thm: Theta 1} for the case $\beta = 1$ under the simplifications {\bf (S1)}--{\bf (S3)}.

{\small
\providecommand{\bysame}{\leavevmode\hbox to3em{\hrulefill}\thinspace}
\providecommand{\MR}{\relax\ifhmode\unskip\space\fi MR }
\providecommand{\MRhref}[2]{%
  \href{http://www.ams.org/mathscinet-getitem?mr=#1}{#2}
}
\providecommand{\href}[2]{#2}

}

\end{document}